\documentclass[10pt,reqno]{amsart}
\usepackage{latexsym,amsmath,amssymb,amscd}
\usepackage[all]{xy}
\usepackage{enumerate}
\usepackage{color}
\usepackage
{hyperref}

\makeatletter
\newcommand\@dotsep{4.5}
\def\@tocline#1#2#3#4#5#6#7{\relax
	\ifnum #1>\c@tocdepth 
	\else
	\par \addpenalty\@secpenalty\addvspace{#2}%
	\begingroup \hyphenpenalty\@M
	\@ifempty{#4}{%
		\@tempdima\csname r@tocindent\number#1\endcsname\relax
	}{%
		\@tempdima#4\relax
	}%
	\parindent\z@ \leftskip#3\relax \advance\leftskip\@tempdima\relax
	\rightskip\@pnumwidth plus1em \parfillskip-\@pnumwidth
	#5\leavevmode\hskip-\@tempdima #6\relax
	\leaders\hbox{$\m@th
		\mkern \@dotsep mu\hbox{.}\mkern \@dotsep mu$}\hfill
	\hbox to\@pnumwidth{\@tocpagenum{#7}}\par
	\nobreak
	\endgroup
	\fi}
\makeatother 

\begin{document}


\makeatletter
\@addtoreset{figure}{section}
\def\thefigure{\thesection.\@arabic\c@figure}
\def\fps@figure{h,t}
\@addtoreset{table}{bsection}

\def\thetable{\thesection.\@arabic\c@table}
\def\fps@table{h, t}
\@addtoreset{equation}{section}
\def\theequation{
\arabic{equation}}
\makeatother

\newcommand{\bfi}{\bfseries\itshape}

\newtheorem{theorem}{Theorem}[section]
\newtheorem{corollary}[theorem]{Corollary}
\newtheorem{definition}[theorem]{Definition}
\newtheorem{example}[theorem]{Example}
\newtheorem{examples}[theorem]{Examples}
\newtheorem{lemma}[theorem]{Lemma}
\newtheorem{notation}[theorem]{Notation}
\newtheorem{convention}[theorem]{Convention}
\newtheorem{proposition}[theorem]{Proposition}
\newtheorem{remark}[theorem]{Remark}
\numberwithin{equation}{section}

\renewcommand{\1}{{\bf 1}}
\newcommand{\A}{{\mathcal{A}}}
\newcommand{\AB}{{\mathbb{AB}}}
\newcommand{\Ad}{{\rm Ad}}
\newcommand{\anc}{{\bf a}}
\newcommand{\Aut}{{\rm Aut}\,}
\newcommand{\ad}{{\rm ad}}
\newcommand{\Atiyah}{\mathfrak{Atiyah}}
\newcommand\BHLG{\mathbb{BHLG}}
\newcommand\BLA{\mathbb{BLA}}
\newcommand{\bm}{{\bf m}}
\newcommand{\card}{{\rm card}}
\newcommand{\Ci}{{\mathcal C}^\infty}
\newcommand{\Der}{{\rm Der}\,}
\newcommand{\de}{{\rm d}}
\newcommand{\ee}{{\rm e}}
\newcommand{\Ehr}{\mathfrak{Ehr}}
\newcommand{\End}{{\rm End}\,}
\newcommand{\Fr}{{\mathsf{F}}}
\newcommand{\ev}{{\rm ev}}
\newcommand{\gauge}{\mathfrak{gauge}}
\newcommand{\GL}{{\rm GL}}
\newcommand{\Gr}{{\rm Gr}}
\newcommand{\graf}{{\mathsf{G}}}
\newcommand{\Hom}{{\rm Hom}}
\newcommand{\hotimes}{\widehat{\otimes}}
\newcommand{\id}{{\rm id}}
\newcommand{\ie}{{\rm i}}
\newcommand{\inc}{{\bf j}}
\newcommand{\gl}{{{\mathfrak g}{\mathfrak l}}}
\newcommand{\Ker}{{\rm Ker}\,}
\newcommand\LAB{{LAB}}
\newcommand{\Lie}{\text{\bf L}}
\newcommand{\LieF}{\mathfrak{Lie}}
\newcommand{\PB}{{\mathbb{PB}}}
\newcommand{\pr}{{\rm pr}}
\newcommand{\Ran}{{\rm Ran}\,}
\renewcommand{\Re}{{\rm Re}\,}
\newcommand{\spa}{{\rm span}\,}
\newcommand{\Tr}{{\rm Tr}\,}
\newcommand{\U}{{\rm U}}

\newcommand{\Ac}{{\mathcal A}}
\newcommand{\Bc}{{\mathcal B}}
\newcommand{\Cc}{{\mathcal C}}
\newcommand{\Dc}{{\mathcal D}}
\newcommand{\Ec}{{\mathcal E}}
\newcommand{\Fc}{{\mathcal F}}
\newcommand{\Gc}{{\mathcal G}}
\newcommand{\Hc}{{\mathcal H}}
\newcommand{\Jc}{{\mathcal J}}
\newcommand{\Kc}{{\mathcal K}}
\newcommand{\Lc}{{\mathcal L}}
\renewcommand{\Mc}{{\mathcal M}}
\newcommand{\Nc}{{\mathcal N}}
\newcommand{\Oc}{{\mathcal O}}
\newcommand{\Pc}{{\mathcal P}}
\newcommand{\Rc}{{\mathcal R}}
\newcommand{\Sc}{{\mathcal S}}
\newcommand{\Tc}{{\mathcal T}}
\newcommand{\Vc}{{\mathcal V}}
\newcommand{\Uc}{{\mathcal U}}
\newcommand{\Xc}{{\mathcal X}}
\newcommand{\Yc}{{\mathcal Y}}
\newcommand{\Zc}{{\mathcal Z}}
\newcommand{\Wc}{{\mathcal W}}

\newcommand{\Ag}{{\mathfrak A}}
\newcommand{\Bg}{{\mathfrak B}}
\newcommand{\Cg}{{\mathfrak C}}
\newcommand{\Fg}{{\mathfrak F}}
\newcommand{\Gg}{{\mathfrak G}}
\newcommand{\Ig}{{\mathfrak I}}
\newcommand{\Jg}{{\mathfrak J}}
\newcommand{\Lg}{{\mathfrak L}}
\newcommand{\Mg}{{\mathfrak M}}
\newcommand{\Pg}{{\mathfrak P}}
\newcommand{\Sg}{{\mathfrak S}}
\newcommand{\Xg}{{\mathfrak X}}
\newcommand{\Yg}{{\mathfrak Y}}
\newcommand{\Zg}{{\mathfrak Z}}

\newcommand{\ag}{{\mathfrak a}}
\newcommand{\bg}{{\mathfrak b}}
\newcommand{\dg}{{\mathfrak d}}
\renewcommand{\gg}{{\mathfrak g}}
\newcommand{\hg}{{\mathfrak h}}
\newcommand{\kg}{{\mathfrak k}}
\newcommand{\mg}{{\mathfrak m}}
\newcommand{\n}{{\mathfrak n}}
\newcommand{\og}{{\mathfrak o}}
\newcommand{\pg}{{\mathfrak p}}
\newcommand{\sg}{{\mathfrak s}}
\newcommand{\tg}{{\mathfrak t}}
\newcommand{\ug}{{\mathfrak u}}
\newcommand{\zg}{{\mathfrak z}}

\newcommand{\bi}{\mathbf{i}}
\newcommand{\bs}{\mathbf{s}}
\newcommand{\bt}{\mathbf{t}}

\newcommand{\BB}{\mathbb B}
\newcommand{\CC}{{\mathbb C}}
\newcommand{\EE}{\mathbb E}
\newcommand{\FF}{\mathbb F}
\newcommand{\HH}{{\mathbb H}}
\newcommand{\MM}{{\mathbb M}}
\newcommand{\NN}{\mathbb N}
\newcommand{\QQ}{\mathbb Q}
\newcommand{\RR}{{\mathbb R}}
\newcommand{\TT}{{\mathbb T}}
\newcommand{\ZZ}{\mathbb Z}

\newcommand{\Q}{{\textbf{Q}}}
\newcommand{\tto}{\rightrightarrows}

\title[Transitive Lie algebroids and \Q-manifolds]{Transitive Lie algebroids and \Q-manifolds}

\author[D. Belti\c t\u a, F. Pelletier]{Daniel Belti\c t\u a, Fernand Pelletier$^\dagger$} 

\address{Institute of Mathematics ``Simion Stoilow'' 
	of the Romanian Academy, 
	P.O. Box 1-764, Bucharest, Romania}
\email{Daniel.Beltita@imar.ro, beltita@gmail.com}

\address{Unit\'e Mixte de Recherche 5127 CNRS, Universit\'e  Savoie Mont Blanc, Laboratoire de Math\'ematiques (LAMA), Campus Scientifique, 73370 Le Bourget-du-Lac, France}
\email{fernand.pelletier@univ-smb.fr, fer.pelletier@gmail.com}


\thanks{$^\dagger$Fernand Pelletier passed away on 19.01.2026.}


\begin{abstract} 
	We introduce the notions of locally trivial \Q-groupoid and  principal \Q-bundle, 
	which are the appropriate versions of locally trivial Lie groupoids and principal fibre bundles in the framework of R. Barre's \Q-manifolds. 
	As an application, we prove that every transitive Lie algebroid over a second countable, smooth manifold arises from the Atiyah sequence of a certain principal \Q-bundle. Consequently, transitive Lie algebroids over second countable, smooth manifolds are integrated to locally trivial \Q-groupoids. 
\end{abstract}

\maketitle

\tableofcontents

\section{Introduction
}

Besides differential geometry, the Lie algebroids play an important role in some areas of physics, classical mechanics, and completely integrable systems. 
For one thing, several problems in these areas can be stated in terms of Poisson brackets,  which provide basic examples of Lie algebroids, cf. e.g., \cite{DoJa21}, \cite{DoJa23}, and the references therein.  
Another important class of Lie algebroids arises from actions of Lie algebras on smooth manifolds. 
Integration of these Lie algebra actions to Lie group actions is an old problem, 
and its study involved certain quotients of Lie groups by non-closed subgroups, 
cf. \cite{KaMi04} and \cite{KaMi22}. 
In this paper we develop tools that allow us to present the category of \emph{all} transitive Lie algebroids over second countable, smooth manifolds as the target of a Lie functor defined on a category of groupoids with a strong differentiable flavor. 

It is well known that Sophus Lie's third classical theorem does not carry over to Lie algebroids, and its range of validity was precisely described in \cite{CrFe03}.
As  there exist transitive Lie algebroids that do not arise from any Lie groupoid, 
it follows that the classical notion of differentiable manifold is not wide enough for the purposes of integrating all transitive Lie algebroids to differentiable groupoids. 
It is therefore remarkable that one can completely solve that integrability problem over second countable, smooth manifolds  by using one of the narrowest generalizations of smooth manifolds that were proposed so far, namely, R.~Barre's \Q-manifolds \cite{Br73}, and we show in the present paper how one can do that. 

Before presenting our results in more detail, let us note that the integrability problems for some classes of Lie algebroids were previously solved in the framework of wider notions of differentiability, such as the ones provided by diffeological spaces or even Fr\"olicher spaces, cf. e.g., \cite{Ig95}, \cite{An23}, \cite{Vi23}. 
In order to illustrate the gaps between the various notions of differentiable spaces that we have mentioned so far, namely 
$$\{\text{smooth manifolds}\} \subset
\{\text{\Q-manifolds}\} \subset
\{\text{Fr\"olicher spaces}\} \subset
\{\text{diffeological spaces}\}  $$
we briefly recall some basic examples: 
\begin{enumerate}[$1^\circ$]
	\item If $\Lambda$ is a countable subgroup of the additive group $\RR$, then $\RR/\Lambda$ is always a \Q-manifold, but it is a smooth manifold if and only if $\Lambda$ is discrete in~$\RR$. 
	\item If $\Lambda\subseteq\RR$ is dense, then $\RR/\Lambda$ is a 1-dimensional \Q-manifold, but it is a trivial Fr\"olicher space, i.e., it  has the same Fr\"olicher structure as a singleton $\{*\}$.
	\item There is an equivalence relation on $\RR^2$ whose corresponding quotient space $\RR^2/\sim$ is a Fr\"olicher space which is not a \Q-manifold, cf. \cite{Pr79}. 
	\item More generally, the leaf space of any regular foliation without holonomy is a Fr\"olicher space but, in general, not a \Q-manifold, even on Banach manifolds, cf. \cite[Prop. 3.13]{BP22}.  
\end{enumerate}
A detailed discussion of the relations between the above notions of differentiable spaces is beyond the scope of the present paper, especially since  
several recent accounts are available., e.g., 
\cite{BtKaWa25}, \cite{GMW24}, and  \cite[App. A]{BP22}. 

The impossibility of integrating arbitrary transitive Lie algebroids to Lie group\-oids, 
which was pointed out in \cite{AM85}, is encoded in the occurrence of certain quotients of Lie groups by \emph{countable} normal subgroups that may not be discrete, hence their corresponding quotient groups are not Lie groups in general. 
(See also \cite{Me21}. 
Integrability of the related algebraic structures of Leibniz or Courant algebroids was studied in \cite{LW16} and \cite{LW20}.)
We show in the present paper that the above countability property allows one to integrate the transitive Lie algebroids over second countable, smooth manifolds to transitive groupoids that are endowed with \Q-manifold structures which are compatible with the groupoid structures, so we call them \Q-groupoids 
(Corollary~\ref{integration}). 
Our approach consists of two main steps: 
\begin{itemize}
	\item[Step 1.] The theory of principal bundles in the framework of \Q-manifolds is developed sufficiently to allow us to show that every abstract Atiyah sequence over a second countable, smooth manifold arises from a principal \Q-bundle (Theorem~\ref{AP_th}). 
	\item[Step 2.] We show that the classical relation between smooth principal bundles and transitive Lie groupoids carries over to the setting of \Q-manifolds 
	(Theorems \ref{GaugeQGroupoid} and \ref{isoloctricQgroupoid}). 
\end{itemize}
While these steps are the same as in the classical framework of smooth manifolds, 
their realization in the framework of \Q-manifolds involves quite a few additional difficulties due to the lack of injectivity of \Q-charts and other specific aspects. 
See for instance the proof of Theorem~\ref{isoloctricQgroupoid}, or the discussion of the Lie functor prior to Proposition~\ref{Liefunctor_prop}, which must deal with the fact 
that, for a \Q-groupoid $\Gc\tto M$, the unit map $\1_M\colon M\to\Gc$ may not be an embedding 
even in the special case when the groupoid $\Gc$ is a group. 
The tools for addressing these specific issues are developed in Appendices \ref{AppA} and \ref{AppB}.  

In some more detail, our paper is organized as follows: 
Section~\ref{Sect2} contains preliminary facts on Lie algebroids and on \Q-manifolds, particularly \Q-groups. 
In Sections \ref{Sect3}--\ref{Sect4} we introduce the notion of principal \Q-bundle, 
and their corresponding Atiyah functor. 
In Section~\ref{Sect5} we achieve the Step~1 above, specifically the integration of 
abstract Atiyah sequences over second countable, smooth manifolds  (Theorem~\ref{AP_th}). 
In Section~\ref{Sect6} we introduce the \Q-groupoids and we develop the Lie theory for them to the extent that is needed for our present purposes in this paper. 
Section~\ref{Sect7} contains a discussion of the special properties of locally trivial \Q-groupoids. 
In Section~\ref{Sect8} we study the gauge groupoids of principal \Q-bundles and we achieve Step~2 above, which leads to the integration of general transitive Lie algebroids over second countable, smooth manifolds  to locally trivial \Q-groupoids (Corollary~\ref{integration}). 
The paper also contains two appendices. 
We discuss in Appendix~\ref{AppA} a feature that distinguishes the \Q-groups and \Q-groupoids from their classical counterparts, specifically, the fact that the space of units is not an embedded submanifold in general, and yet it does have the property of initial submanifold. 
This property plays a key role in the construction of the Lie algebroid of a \Q-groupoid.  
We emphasize that this is not merely a technical issue but is one of the 
key points in the integration of general transitive Lie algebroids. 
Finally, in Appendix~\ref{AppB}, we establish some properties of vector distributions on \Q-manifolds that are also needed in the construction of the Lie functor for \Q-groupoids.

\section{Preliminaries}
\label{Sect2}

In this section we collect a few auxiliary facts 
that are divided into two main topics: Lie algebroids (Subsection~\ref{subsect2.1}) 
and \Q-groups (Subsection~\ref{subsect2.2}). 
Our basic references for Lie groupoids and Lie algebroids are \cite{Ma05} and \cite{MoMr03}, 
and we refer to  \cite{Br73} for \Q-manifolds and \Q-groups.  
(See also \cite{BPZ19} and \cite{BP22}, as well as \cite{Mo75} and \cite{Mo77}.) 

\subsection{Preliminaries on Lie algebroids} 
\label{subsect2.1}

\begin{definition}[trivial Lie algebroid]
\label{trivLiealg_def}
\normalfont 
If $U$ is a smooth manifold and $\gg$ is a finite-dimensional real Lie algebra, 
then the \emph{trivial Lie algebroid} with base $U$ and fibre~$\gg$ is the Whitney sum of vector bundles 
$$A:=TU\oplus(U\times\gg)\to U$$ 
whose anchor is given by the first Cartesian projection $\anc:=\pr_1\colon TU\oplus(U\times\gg)\to TU$ and with the Lie bracket on the section space $\Gamma(A)\simeq\Xc(U)\oplus\Ci(U,\gg)$ defined by 
\begin{equation}
\label{trivLiealg_def_eq1}
[X\oplus v,Y\oplus w]:=[X,Y]\oplus(X(w)-Y(v)+[v,w])
\end{equation}
for all $X,Y\in\Xc(U)$ and $v,w\in\Ci(U,\gg)$. 
See for instance \cite[Ex. 3.3.3]{Ma05}. 
\end{definition}

\begin{remark}[morphisms of Lie algebroids]
\label{trivLiealg_morph}
\normalfont 
Recall from \cite[\S 4.2]{Ma05} that if $q\colon A\to M$ is a Lie algebroid with its anchor $\anc\colon A\to TM$, 
and $f\colon N\to M$ is a smooth map, then one defines 
\begin{equation}
\label{trivLiealg_morph_eq1}
f^{!!}A:=TN\oplus_{TM} A:=\{(X,v)\in TN\times A\mid (Tf)(X)=\anc(v)\}. 
\end{equation}
If one of the following conditions is satisfied: 
\begin{itemize}
	\item $f$ is a surjective submersion, or 
	\item the Lie algebroid $A$ is transitive, 
\end{itemize}
then $f^{!!}A$ has the natural structure of a Lie algebroid 
with its bundle projection $q^{!!}\colon f^{!!}A\to N$, $q^{!!}(X,v):=n$ if $X\in T_nN$ 
and with its anchor $\anc^{!!}\colon f^{!!}A\to TN$, $\anc^{!!}(X,v):=X$. 

Now assume additionally that $q'\colon A'\to N$ is a Lie algebroid with its anchor $\anc'\colon A'\to TN$, and $\varphi\colon A'\to A$ is a morphism of vector bundles over the mapping $f\colon N\to M$. 
Then  $\varphi$ is a \emph{morphism of Lie algebroids} if and only if the following conditions are satisfied: 
\begin{enumerate}[{\rm(i)}]
	\item\label{trivLiealg_morph_item1} $\anc\circ \varphi=Tf\circ\anc'$; 
	\item\label{trivLiealg_morph_item2} the mapping 
	\begin{equation}
	\label{trivLiealg_morph_eq2}
	\varphi^{!!}\colon A'\to f^{!!}A, \quad \varphi^{!!}(v'):=(\anc'(v'),\varphi(v')), 
\end{equation}
is a morphism of Lie algebroids over the identity map $\id_N$, i.e., $\anc^{!!}\circ \varphi^{!!}=\anc'$ (which obviously holds true) and $\varphi^{!!}\circ[\sigma_1,\sigma_2]=[\varphi^{!!}\circ\sigma_1,\varphi^{!!}\circ\sigma_2]$ for all $\sigma_1,\sigma_2\in\Gamma(A')$. 
\end{enumerate}
See \cite[\S 4.3]{Ma05} for more details. 
\end{remark}

\begin{remark}[localization of morphisms of Lie algebroids]
	\label{trivLiealg_morph_loc}
	\normalfont 
We recall from \cite[Prop. 3.3.2]{Ma05} that if $q\colon A\to M$ is a Lie algebroid with its anchor $\anc\colon A\to TM$, then for every open subset $U\subseteq M$ 
the restricted vector bundle 
$$q_U:=q\vert_{q^{-1}(U)}\colon A_U:=q^{-1}(U)\to U$$ 
has the natural structure of a Lie algebroid 
with its anchor 
$$\anc_U:=\anc\vert_{A_U}\colon A_U\to TU$$ 
and 
with its Lie bracket $\Gamma(A_U)\times \Gamma(A_U)\to \Gamma(A_U)$ obtained by restriction of the Lie bracket $\Gamma(A)\times \Gamma(A)\to \Gamma(A)$. 

Now let $q'\colon A'\to N$ be another Lie algebroid with its anchor $\anc'\colon A'\to TN$ and $\varphi\colon A'\to A$ a morphism of vector bundles over the mapping $f\colon N\to M$. 
We claim that \emph{if the Lie algebroid $A$ is transitive, then $\varphi$ is a morphism of Lie algebroids if and only if for every $n\in N$ there exist open subsets $U\subseteq N$ and $V\subseteq M$ with $n\in U$, $f(n)\in V$, 
$f(U)\subseteq V$ (hence $\varphi(A'_U)\subseteq A_V$), such that the commutative diagram 
\begin{equation*}
\xymatrix{A'_U \ar[d]_{q'_U} \ar[r]^{\varphi_{UV}} & A_V \ar[d]^{q_V} \\
U \ar[r]_{f_{UV}}& V}
\end{equation*}
is a morphism of Lie algebroids, where $f_{UV}:=f\vert_U\colon U\to V$,   $\varphi_{UV}:=\varphi\vert_{A'_U}\colon A'_U\to A_V$, $q'_U:=q'\vert_{A'_U}\colon A'_U\to U$, and $q_V:=q\vert_{A_V}\colon A_V\to V$}. 

We prove that claim in three steps: 

Step 1: If $N=M$ and $f=\id_M$, then the assertion is clear. 

Step 2: For any  open subsets $U\subseteq N$ and $V\subseteq M$ with 
$f(U)\subseteq V$, 
we have 
\begin{align*}
	f_{UV}^{!!}A_V
	&=TU\oplus_{TV} A_V=\{(X,v)\in TU\times A_V\mid (Tf_{UV})(X)=\anc(v)\} \\
	&=\{(X,v)\in TN\times A\mid (Tf)(X)=\anc(v),\ X\in TU \} \\
	&=(q^{!!})^{-1}(U) \\
	&=(f^{!!}A)_U
\end{align*}
and  
$$\varphi_{UV}^{!!}\colon A'_U\to f_{UV}^{!!}A_V, \quad 
\varphi_{UV}^{!!}(v'):=(\anc'_U(v'),\varphi_{UV}(v'))=(\anc'(v'),\varphi(v'))$$
hence 
\begin{equation*}
	\varphi_{UV}^{!!}=\varphi^{!!}\vert_{A'_U}\colon A'_U\to (f^{!!}A)_U.
\end{equation*}
\par 
Step 3: 
For the general case, we first note that if the Lie algebroid $A$ is transitive, then for every open subset $V$ the restricted Lie algebroid $A_V$ is transitive as well, 
hence the above map $\varphi_{UV}\colon A'_U\to A_V$ is a morphism of Lie algebroids over $f_{UV}\colon U\to V$ if and only if the conditions \eqref{trivLiealg_morph_item1}--\eqref{trivLiealg_morph_item2} in Remark~\ref{trivLiealg_morph} are satisfied. 
By Step~1 above, these conditions \eqref{trivLiealg_morph_item1}--\eqref{trivLiealg_morph_item2} for $\varphi_{UV}^{!!}$ 
are satisfied for all $U\in\Uc$ and $V\in\Vc$ 
for some open coverings $\Uc$ of $N$ and $\Vc$ of $M$, 
if and only if \eqref{trivLiealg_morph_item1}--\eqref{trivLiealg_morph_item2}  
are satisfied for $\varphi^{!!}$, 
that is, if and only if  $\varphi$ is a morphism of Lie algebroids over $f\colon N\to M$. 
This completes the proof of our claim. 
\end{remark}

\begin{remark}
	\label{trivLiealg_end}
	\normalfont 
	The endomorphisms (in particular automorphisms) of trivial Lie algebroids over the identity map of the base were explicitly described in \cite[\S 2]{BlKuWa02} and \cite[Ex. 3.3.3]{Ma05}, as follows. 
	Assuming the notation of Definition~\ref{trivLiealg_def}, 
	let $\varphi\colon A\to A$ be an endomorphism of the vector bundle $A=TU\oplus(U\times\gg)\to U$ over the identity map $\id_U$ of the base $U$. Then $\varphi$ is a morphism of Lie algebroids if and only if there exist a smooth $\gg$-valued 1-form $\omega\colon TU\to\gg$ and a smooth map $\varphi^+\colon \U\to\Hom(\gg)$ taking values in the space of Lie algebra endomorphisms of $\gg$, satisfying the conditions 
	\begin{enumerate}[{\rm(i)}]
		\item $\varphi(X,v)=(X,\omega(X)+\varphi^+(m)v)$ if $m\in U$, $X\in T_mU$, and $v\in\gg$; 
		\item $\de \omega(X,Y)+[\omega(X),\omega(Y)]=0$ if $X,Y\in\Xc(U)$; 
		\item $X(\varphi^+(v))-\varphi^+(X(v))+[\omega(X),\varphi^+(v)]=0$ if $X\in\Xc(U)$ and $v\in\Ci(U,\gg)$. 
	\end{enumerate}
	The last condition is the same as \cite[\S 2, Th. 2.1(W2)]{BlKuWa02}, since 
	\begin{equation*}
		\de\varphi(X)\cdot v=[\varphi\cdot v,\omega(X)]\text{ if } X\in\Xc(U)\text{ and }v\in\Ci(U,\gg)
	\end{equation*}
	where $(\varphi\cdot v)(m)=\varphi(m)v(m)$ and $(\de\varphi(X)\cdot v)(m)=(\de\varphi(X(m))(v(m))$ 
	for every $m\in M$. 
	(Compare with \eqref{deriv_eq1}.) 
	This description of endomorphisms of trivial Lie algebroids is needed in Example~\ref{triv}.
\end{remark}

\begin{definition}[Lie algebroid atlas]
	\label{atlas}
	\normalfont 
Let $N$ be a smooth manifold and $\gg$ be a finite-dimensional real Lie algebra. 
A \emph{Lie algebroid atlas with fibre~$\gg$}
for a vector bundle $p\colon A\to N$, 
associated with an open covering $(U_i)_{i\in I}$ of $N$,
is a family of 
isomorphisms of vector bundles 
$(S_i\colon TU_i\oplus(U_i\times\gg)\to A_{U_i})_{i\in I}$ 
such that if $i,j\in I$ and $U_{ij}:=U_i\cap U_j\ne\emptyset$ then 
the overlap isomorphism of vector bundles  
$$S_i^{-1}\circ S_j\colon TU_{ij}\oplus(U_{ij}\times\gg)\to TU_{ij}\oplus(U_{ij}\times\gg)$$
is an isomorphism of (trivial) Lie algebroids. 
Here $A_{U_i}$ denotes the restriction of the vector bundle $A$ to $U_i\subseteq N$ for  $i\in I$. 
See \cite[Def. 5.4.1]{Ma05} for more details. 
\end{definition}

The following fact is essentially obtained by the method of proof of \cite[Th. 8.2.6]{Ma05}.

\begin{lemma}
\label{atlas_lemma}
In Definition~\ref{atlas}, the vector bundle $p\colon A\to N$ has a unique structure of transitive Lie algebroid for which the restriction map $\rho^N_{U_i}\colon\Gamma(A)\to \Gamma(A_{U_i})$ is a Lie algebra homomorphism, $S_i\colon TU_i\oplus(U_i\times\gg)\to A_{U_i}$ is a Lie algebroid isomorphism, and the anchor maps $a\colon A\to TN$ and $a_i\colon A_{U_i}\to U_i$ satisfy $a\vert_{A_{U_i}}=a_i$ for every $i\in I$. 
\end{lemma}

\begin{proof}
The uniqueness assertion is clear, so we will discuss only the existence of a Lie algebroid structure on the vector bundle $p\colon A\to N$ with the properties indicated in the statement. 
To this end, for every $i\in I$, we endow $A_{U_i}$ with the Lie algebroid structure for which $S_i$ is an isomorphism of Lie algebroids. 
Similarly, since the restriction $S_{ij}$ of $S_i$ to $TU_{ij}\oplus(U_{ij}\times\gg)$ 
is an isomorphism onto $A_{U_{ij}}$, we can endow $A_{U_{ij}}$ with the Lie algebroid structure for which $S_{ij}$ is an isomorphism of Lie algebroids. 
Then the restriction mapping $\rho^{U_i}_{U_{ij}}\colon\Gamma(A_{U_i})\to\Gamma(A_{U_{ij}})$ is a morphism of Lie algebras and is $\Ci(U_i)$-linear. 
We now note the exact sequence of vector spaces 
\begin{equation*}
	\xymatrix{
	0\to \Gamma(A) \ar[r]^{\rho_0\quad\ }& \prod\limits_{i\in I}\Gamma(A_{U_i}) \ar[r]^{\rho_1-\rho_2\quad\ } & \prod\limits_{(i,j)\in I\times I}\Gamma(A_{U_{ij}}) 
}
\end{equation*}
where 
the mappings $\rho_0,\rho_1,\rho_2$ are defined by 
\begin{align*}
\rho_0(\sigma)&:=(\sigma\vert_{U_i})_{i\in I} \\
\rho_1((\sigma_i)_{i\in I})&:=(\sigma_i\vert_{U_{ij}})_{(i,j)\in I\times I}, \\
\rho_2((\sigma_i)_{i\in I})&:=(\sigma_j\vert_{U_{ij}})_{(i,j)\in I\times I}
\end{align*}
we set $\Gamma(A_{U_{ij}}):=\{0\}$ and $\sigma_i\vert_{U_{ij}}=\sigma_j\vert_{U_{ij}}=0$ if $U_i\cap U_j=\emptyset$. 

Since $\rho_1$ and $\rho_2$ are Lie algebra morphisms by the above considerations, 
it follows that $\Ker(\rho_1-\rho_2)$ is a subalgebra of the Lie algebra 
$\prod\limits_{i\in I}\Gamma(A_{U_i})$, and a $\Ci(N)$-sub\-module as well. 
Then the Lie bracket of $\Ker(\rho_1-\rho_2)$ can be transported to $\Gamma(A)$ via the linear isomorphism $\rho_0\colon \Gamma(A)\to \Ker(\rho_1-\rho_2)$. 

In order to define the anchor $a\colon A\to TN$, 
we use the anchor maps $a_i\colon A_{U_i}\to TU_i$ 
of the Lie algebroids $A_{U_i}$ 
for all $i\in I$. 
Specifically, we define $a(v):=a_i(v) $ if $v\in A_{U_i}$. 
The mapping $a$ is well defined since the anchor maps of the trivial Lie algebroids $TU_i\oplus(U_i\times\gg)$ are compatible with the restriction from $U_i$ to $U_{ij}$. 
It is clear that the mapping $a\colon A\to TN$ defined in this way is a vector bundle morphism. 

Moreover, for every $f\in\Ci(N)$, $\sigma_1,\sigma_2\in\Gamma(A)$, and $x\in N$ 
we select $i\in I$ with $x\in U_i$ and we have by the above construction of the bracket $[\cdot,\cdot]$ on $\Gamma(A)$, 
\begin{align*}
[\sigma_1,f\sigma_2](x)
&=[\sigma_1\vert_{U_i},f\vert_{U_i}\sigma_2\vert_{U_i}]_{A_{U_i}}(x) \\
&=f(x)[\sigma_1\vert_{U_i},\sigma_2\vert_{U_i}]_{A_{U_i}}(x)
+(a_i\circ \sigma_1\vert_{U_i})(f\vert_{U_i})(x)\sigma_2(x) \\
&=f(x)[\sigma_1,\sigma_2](x)+(a\circ \sigma_1)(f)(x)\sigma_2(x) 
\end{align*}
This completes the construction of the Lie algebroid structure on the vector bundle $p\colon A\to N$. 
\end{proof}

\subsection{Preliminaries on \Q-groups}
\label{subsect2.2}

We first recall the following definition from \cite{Br73}. 
(See also \cite[Def. 2.1]{BPZ19}.) 

\begin{definition}\label{Q_def}
	\normalfont
	Let $\pi\colon M\to S$ be a mapping from a smooth manifold to a set. 
	The mapping $\pi$ is called a \emph{\Q-chart} if it satisfies the following conditions: 
	\begin{enumerate}[{\rm(a)}]
		\item\label{Q_ex_a} For all $x,y\in M$ with $\pi(x)=\pi(y)$, there exist neighbourhoods $U$ of $x$ and $U'$ of $x'$ in $M$ and a diffeomorphism $h\colon U\to U'$ with $h(x)=y$ and 
		$\pi(h(z))=\pi(z)$ for all $z\in U$.
		\item\label{Q_ex_b} For every smooth manifold $T$ and any smooth maps $f,h\colon T\to M$ with $\pi\circ f=\pi\circ h$, the set $\{t\in T\mid f(t)=h(t)\}$ is open in $T$. 
	\end{enumerate} 
If $\pi$ is surjective then it is called a \emph{\Q-atlas}. 
Two \Q-charts $\pi_j\colon M_j\to S$, $j=1,2$, are said to be \emph{equivalent} if their disjoint union $\pi_1\sqcup M_1\sqcup M_2\to S$ is again a \Q-chart. 
A \emph{\Q-manifold} is a set $S$ together with an equivalence class of \Q-atlasses. 

If $S_1$ and $S_2$ are \Q-manifolds, then a mapping $f\colon S_1\to S_2$ is called \emph{\Q-smooth} or \emph{morphism of \Q-manifolds}  if for every $s_j\in S_j$, $j=1,2$, with $f(s_1)=f(s_2)$ there exist \Q-charts $\pi_j\colon M_j\to S_j$, a smooth mapping  
$\widehat{f}\colon M_1\to M_2$, and points $m_j\in M_j$, $j=1,2$, with $\widehat{f}(m_1)=m_2$, $\pi_j(m_j)=s_j$ for $j=1,2$, and $f\circ \pi_1=\pi_2\circ\widehat{f}$.  
A \emph{\Q-diffeomorphism} is a bijective \Q-smooth mapping whose inverse is also \Q-smooth.

A \emph{\Q-group} is a \Q-manifold endowed with a group structure whose corresponding group operation and inversion are \Q-smooth maps. 
\end{definition}

\begin{remark}[\Q-immersions and \Q-submersions]
\label{rem:subm}
\normalfont 
In the framework of Definition~\ref{Q_def}, 
a mapping $f\colon S_1\to S_2$ is said to be a \emph{\Q-submersion} (respectively, \emph{\Q-immersion}) if there exist \Q-atlases $\pi_j\colon M_j\to S_j$ for $j=1,2$ and a submersion (respectively, immersion) 
$\widehat{f}\colon M_1\to M_2$ in the usual sense with $f\circ \pi_1=\pi_2\circ\widehat{f}$.  
A subset  
$S'\subseteq S_1$ is an  \emph{immersed \Q-submanifold}  
if it has the structure of a \Q-manifold with the initial manifold property: 
its inclusion  map $S'\hookrightarrow S_1$ is \Q-smooth 
and every \Q-smooth map $h\colon T\to S_1$ with $h(T)\subseteq S'$ is also \Q-smooth as a mapping into $S'$. 
These notions share some of the basic features of the classical notions. 
We mention the following ones for later use: 
\begin{enumerate}[{\rm(i)}]
	\item\label{rem:subm_1} 
	If $f\colon S_1\to S_2$ is a \Q-submersion then for every point $s\in S_2$ the subset 
	$f^{-1}(s)$ has the structure of an immersed \Q-submanifold. 
	See 
	\cite[Ch. 1, \S 2, no. 5, Prop., page 243]{Br73} for a more general result of this type. 
	\item\label{rem:subm_2} 
	If $f\colon S_1\to S_2$  and $g\colon S_2\to S_3$ are \Q-smooth maps, $f$ is surjective, 
	and $g\circ f$ is a \Q-submersion, then $g$ is a \Q-submersion, too.  
	See 
	\cite[Ch. 1, \S 2, no. 1, Prop., page 239]{Br73}, 
	where one should add the hypothesis that $f$ is surjective.  
	\item\label{rem:subm_3}  If $f\colon S_1\to S_2$ is a \Q-submersion, and $S_3$ is a \Q-manifold, 
	then a mapping $g\colon S_2\to S_3$ is \Q-smooth if and only if $g\circ f$ is \Q-smooth. 
	See 
	\cite[Ch. 1, \S 3, no. 1, Prop., page 245]{Br73}. 
	\item\label{rem:subm_4}  If $f\colon S_1\to S_2$ is  \Q-smooth and $g\colon S_2\to S_3$ are \Q-smooth and $g\circ f$ is a \Q-immersion, then $f$ is a \Q-immersion. 
See 
\cite[Ch. 1, \S 2, no. 1, Prop., page 239]{Br73}. 
\end{enumerate}
\end{remark}

\begin{remark}[description of \Q-groups]
\label{rem:Q-gr}
\normalfont 
If $\widetilde{G}$ is a Lie group, then a normal subgroup $\Lambda\subseteq\widetilde{G}$ is called \emph{pseudo-discrete} if every smooth map $\gamma\colon\RR\to\widetilde{G}$ with $\gamma(\RR)\subseteq\Lambda$ is constant. 
A  normal pseudo-discrete subgroup $\Lambda$ is necessarily contained in the center of $\widetilde{G}$ if this last group is connected, cf. \cite[Prop. 5.3]{BP22}. 
Every countable subgroup is pseudo-discrete, cf. \cite[Ch. I, \S 1, no. 3, Ex.]{Br73}.  
A normal subgroup $\Lambda\subseteq\widetilde{G}$ is pseudo-discrete if and only if the quotient map $\widetilde{G}\to\widetilde{G}/\Lambda$ is a \Q-atlas, 
by \cite[Prop. 3.5]{BPZ19}. 
If this is the case, then the above \Q-atlas makes the quotient group $\widetilde{G}/\Lambda$ into a \Q-group. 
Conversely, for every \Q-group $G$, there exist a connected, simply connected Lie group $\widetilde{G}$ with a pseudo-discrete central subgroup $\Lambda\subseteq \widetilde{G}$ 
and an isomorphism of \Q-groups $G=\widetilde{G}/\Lambda$. 
See \cite[Ch. 2, \S 5, no. 3, Cor., page 254]{Br73}. 
\end{remark}

\begin{remark}[tangent \Q-group]
	\label{TG_struct}
	\normalfont 
	There exists a natural tangent functor from the category of \Q-manifolds into itself that recovers the usual tangent functor on the subcategory of smooth manifolds, cf. \cite[Ch. 1, \S 1, no. 4]{Br73} and also \cite[\S 2.4]{BP22}. 
	In that framework, one can develop a few elements of Lie theory for \Q-groups. 
	Thus, we recall from \cite[Ch. 2, \S 2]{Br73} that if $G$~is a \Q-group with its Lie algebra~$\gg$ 
	then we have the adjoint action 
	$$\Ad_G\colon G\to\Aut(\gg),$$ 
	which is a \Q-smooth group morphism, and the corresponding semidirect product 
	$\gg\rtimes_{\Ad_G}G$ has the natural structure of a \Q-group. 
	(Here $\gg$ is regarded as the abelian Lie group $(\gg,+)$.) 
	
	Moreover, if $\bm\colon G\times G\to G$ is the group operation of $G$ 
	then 
	the tangent space $TG$ is a \Q-group 
	with its group operation $T\bm\colon TG\times TG\to TG$. 
	The mapping (to be called the \emph{right trivialization} of $TG$)
	$$\rho_G\colon TG\to \gg\rtimes_{\Ad_G}G,\quad v\mapsto (\omega_G(v),\tau_G(v))$$
	is an isomorphism of \Q-groups, where 
	$$\omega_G\colon TG\to \gg$$ 
	is 
	the right-invariant Maurer-Cartan form and 
	$$\tau_G\colon TG\to G$$ is the tangent bundle projection of~$G$. 
	For the inverse of the right trivialization we use the following notation which is intended to suggest right translations by elements of~$G$: 
	$$\rho_G^{-1}\colon \gg\rtimes_{\Ad_G}G\to TG,\quad \rho_G^{-1}(v,g)=:vg\in T_gG.$$
	It is convenient to define also 
	$$(\forall v\in\gg)(\forall g\in G)\quad gv:=\underbrace{(\Ad_G(g)v)}_{\in\gg}g\in T_gG$$
	In this notation we have for all $g_1,g_2\in G$ and $v_1,v_2\in\gg$
	\begin{equation}
		\label{TG_struct_eq1}
		T\bm(v_1g_1,v_2g_2)=(v_1+\Ad_G(g_1)v_2)(g_1g_2):=\rho_G^{-1}(v_1+\Ad_G(g_1)v_2,g_1g_2)
	\end{equation}
	and 
	\begin{equation}
		\label{TG_struct_eq2}
		(\forall v\in\gg)\quad \omega_G^{-1}(v)=vG.
	\end{equation}
\end{remark}

In the following remark we extend the Darboux derivative to functions that take values in a \Q-group. 
We refer to \cite[\S 5.1]{Ma05} for the classical case of functions that take values in a Lie group. 

\begin{remark}[Darboux derivative]
\label{deriv}
\normalfont 
Let $G$ be a \Q-group. 
For every smooth manifold $M$ and every \Q-smooth function $f\colon M\to G$, 
the $\gg$-valued 1-form 
\begin{equation*}
	\Delta_G(f)\colon TM\to \gg,\quad \Delta_G(f):=\omega_G\circ Tf=f^*(\omega_G)
\end{equation*}
is called the \emph{Darboux derivative} of $f$. 
Just as in \cite[\S 5.1]{Ma05},  if $\varphi\colon G\to H$ is a morphism of \Q-groups, then $\Delta_H(\varphi\circ f)=(T_\1\varphi)\circ\Delta_G(f)$. 
In particular, for $H:=\Aut(\gg)$ and $\varphi:=\Ad_G\colon G\to\Aut(\gg)$ we obtain 
\begin{equation*}
	\Delta_{\Aut(\gg)}(\Ad_G\circ f)=\ad_\gg\circ \Delta_G(f)\colon TM\to\Der(\gg).
\end{equation*} 
That is, for every vector field $X\in\Xc(M)$ and every smooth function $v\colon M\to\gg$ we have 
\begin{equation*}
\underbrace{(\Delta_{\Aut(\gg)}(\Ad_G\circ f))(X)}_{\quad\qquad\qquad\in\Ci(M,\Der(\gg))}v
=[\underbrace{(\Delta_G(f))(X)}_{\in\Ci(M,\gg)},v].
\end{equation*}
Now let $G$ and $H$ be \Q-groups, $f\colon M\to H$ be a \Q-smooth map, 
and $\varphi\colon\gg\to\hg$ be a Lie algebra morphism. 
If we define 
$$\widetilde{\varphi}\colon M\to \Hom(\gg,\hg),\quad 
\widetilde{\varphi}(x):=\Ad_H(f(x))\circ\varphi$$
then 
for every $X\in\Xc(M)$ and every smooth function $v\colon M\to\gg$ we have 
\begin{equation}
\label{deriv_eq1}
X(\widetilde{\varphi}\cdot v)=\widetilde{\varphi}\cdot X(v)-[(\Delta_H(f))(X),\widetilde{\varphi}\cdot v].
\end{equation}
This formula can be obtained by the method of proof of \cite[Prop. 5.1.1]{Ma05}.
\end{remark}

\section{Principal \Q-bundles}
\label{Sect3}

In this section we adapt the classical notion of principal bundle (cf. \cite[Ch. I, \S 5]{KN63}) to the setting of \Q-manifolds, to the extent that is needed for the integration of transitive Lie algebroids.

\begin{definition}
\label{Q-bundle_def}
\normalfont 
Let $G$ be a \Q-group and $N$ be a smooth manifold. 
A \emph{principal \Q-bundle} over $N$ with structure group $G$ 
is a \Q-manifold $P$ endowed with a free right \Q-smooth action 
\begin{equation}
\label{Q-bundle_def_eq1}
P\times G\to P, \quad (p,g)\mapsto pg,
\end{equation}
and a \Q-smooth mapping 
$$\beta\colon P\to N$$ 
satisfying the following conditions: 
\begin{itemize}
	\item For every $p\in P$ we have $\beta^{-1}(\beta(p))=pG$. 
	\item For every point $n\in N$ there exist an open neighborhood $U\subseteq N$ 
	and a $G$-equivariant \Q-diffeomorphism $\psi$ for which the diagram 
	\begin{equation}
		\label{Q-bundle_def_eq2}
	\xymatrix{
    \beta^{-1}(U) \ar[dr]_{\beta\vert_{\beta^{-1}(U)}} \ar[rr]^{\psi}& & U\times G \ar[dl]^{\pr_1}\\
     & U & 
	}
\end{equation}
	is commutative, where $\pr_1$ denotes the Cartesian projection onto the first factor.
\end{itemize}
\end{definition}

\begin{remark}
\label{local}
\normalfont
The commutative diagram~\eqref{Q-bundle_def_eq2} means that there exists a \Q-smooth mapping $\varphi\colon U\to G$ such that 
$$(\forall p\in\beta^{-1}(U))\quad \psi(p)=(\beta(p),\varphi(p))$$
and the $G$-equivariance property of $\psi$ means 
\begin{equation}
\label{local_eq1}
(\forall g\in G)(\forall  p\in\beta^{-1}(U))\quad \varphi(pg)=\varphi(p)g.
\end{equation}
As in the classical case, it follows that there exist an open covering $N=\bigcup\limits_{i\in I}U_i$ and for every $i\in I$ a $G$-equivariant \Q-diffeomorphism $\psi_i\colon \beta^{-1}(U_i)\to U_i\times G$, $p\mapsto (\beta(p),\varphi_i(p))$ as above. 
For every $i,j\in I$ with $U_i\cap U_j\ne\emptyset$ it follows by \eqref{local_eq1} that the following mapping 
\begin{equation}
	\label{local_eq1.5}
\psi_{ji}\colon U_i\cap U_j\to G,\quad \psi_{ji}(\beta(p)):=\varphi_j(p)\varphi_i(p)^{-1}
\end{equation}
is correctly defined, since $\beta^{-1}(\beta(p))=pG$. 
Moreover, we have 
\begin{equation}
\label{local_eq2}
\psi_j\circ\psi_i^{-1}\colon (U_i\cap U_j)\times G\to (U_i\cap U_j)\times G,\quad 
(n,g)\mapsto(n,\psi_{ji}(n)g).
\end{equation}
Using the fact that the multiplication in $G$ is \Q-smooth, it then directly follows that 
the mapping $\psi_{ji}\colon U_i\cap U_j\to G$ is \Q-smooth. 
The mappings $\psi_{ji}$ are called the \emph{transition functions} of the principal \Q-bundle with respect to the open covering $N=\bigcup\limits_{i\in I}U_i$ 
and it is clear that they satisfy the \emph{cocycle condition} 
\begin{equation}
\label{local_eq3}
	\psi_{ki}(n)=\psi_{kj}(n)\psi_{ji}(n)\text{ for }n\in U_i\cap U_j\cap U_k
\end{equation}
if $i,j,k\in I$ with $U_i\cap U_j\cap U_k\ne\emptyset$. 
\end{remark}

Motivated by the above remark, we make the following definition: 

\begin{definition}
\normalfont
\label{transition_def}
Let $G$ be a \Q-group and $N$ be a smooth manifold 
with an open covering $N=\bigcup\limits_{i\in I}U_i$. 
A family of \emph{$G$-valued transition functions with respect to the covering $\{U_i\}_{i\in I}$} 
 is a family of \Q-smooth mappings $\psi_{ji}\colon U_i\cap U_j\to G$ 
 defined for all $i,j\in I$ with $U_i\cap U_j\ne\emptyset$, 
satisfying the cocycle condition~\eqref{local_eq3}. 
\end{definition}

\begin{remark}
\normalfont 
Assume the setting of Definition~\ref{transition_def}. 
It is easily seen that if $H$ is another \Q-group and $\Theta\colon G\to H$ is a \Q-smooth group homomorphism, then the family of mappings $\Theta\circ\psi_{ji}\colon U_i\cap U_j\to H$ is a family of $H$-valued transition functions with respect to the covering $\{U_i\}_{i\in I}$. 
\end{remark}

\begin{definition}\label{AP_def_morph}
\normalfont
 For $i=1,2$, let $G_i$ be a \Q-group and $\beta_i:P_i\to N_i$ be a principal \Q-bundle with structure group $G_i$.   
 Assume that we have a homomorphism of \Q-groups $\Theta:G_1\to G_2$. 
 A  \emph{morphism of principal \Q-bundles between $P_1$ and $P_2$} is a \Q-smooth map $F\colon P_1\to P_2$ over a smooth map $f\colon N_1\to N_2$ such that 
 $\beta_2\circ F=f\circ \beta_1$ and which is equivariant, that is, $F(pg)=F(p)\Theta(g)$ for all $p\in P_1$ and $g\in G_1$.  
 The morphism $F$ is an \emph{isomorphism  of principal \Q-bundles} if $\Theta: G_1\to G_2$ is an isomorphism of \Q-groups and $\Phi$ and $\phi$ are $\Q$-diffeomorphisms.
 \end{definition}

Proposition~\ref{bundle} below is a version of \cite[Ch. I, Prop. 5.2]{KN63} for principal \Q-bundles. 
{In order to motivate the definition of the equivalence relation~\eqref{ildePsim}, 
we make the following remark on the special case of smooth principal bundles. 
\begin{remark}
\label{bundle_smooth}
\normalfont 
Let $\beta\colon P\to N$ be a smooth principal bundle whose structure group is a Lie group~$G$. 
Assume that $N=\bigcup\limits_{i\in I}U_i$ is an open covering 
and for every $i\in I$ we have a smooth cross-section $\sigma_i\colon U_i\to P$, 
hence $\beta\circ\sigma_i=\id_{U_i}$. 
Then the mapping 
\begin{equation}
\label{bundle_smooth_eq1}
\tau_i\colon U_i\times G\to\beta^{-1}(U_i),\quad  (x,g)\mapsto\sigma_i(x)g, 
\end{equation}
is a diffeomorphism, and we denote its inverse by $\psi_i\colon \beta^{-1}(U_i)\to U_i\times G$, $\psi_i(p)=:(\beta(p),\varphi_i(p))$,  
hence 
\begin{equation*} 
	(\forall p\in \beta^{-1}(U_i))\quad p=\sigma_i(\beta(p))\varphi_i(p).
\end{equation*} 
Therefore, if $p\in \beta^{-1}(U_i)\cap \beta^{-1}(U_j)$, 
then $\sigma_i(\beta(p))\varphi_i(p)=p=\sigma_j(\beta(p))\varphi_j(p)$. 
Thus, using \eqref{local_eq1.5}, we obtain 
\begin{equation} 
\label{bundle_smooth_eq2}
	(\forall n\in U_i\cap U_j)\quad \sigma_i(n)=\sigma_j(n)\psi_{ji}(n).
\end{equation} 
Consequently, if $i_1,i_2\in I$, 
then we can describe the points where the maps $\tau_{i_1}$ and $\tau_{i_2}$ 
defined in \eqref{bundle_smooth_eq1}
	take the same value. 
Specifically,  
if $(n_1,g_1)\in U_{i_1}\times G$, and $(n_2,g_2)\in U_{i_2}\times G$, then we have 
\begin{align*}
\tau_{i_1}(n_1,g_1)=\tau_{i_1}(n_1,g_1)
& \iff n_1=n_2=:n\in U_{i_1}\cap U_{i_2}\text{ and }
\sigma_{i_1}(n)g_1=\sigma_{i_2}(n)g_2 \\
& \iff n_1=n_2=n \text{ and }g_2g_1^{-1}=\psi_{i_2i_1}(n).
\end{align*}
This last condition should be compared with \eqref{ildePsim} in the proof of Proposition~\ref{bundle}. 
\end{remark}
}

\begin{proposition}
\label{bundle}
Let $G$ be a \Q-group and $N$ be a smooth manifold. 
Assume that $N=\bigcup\limits_{i\in I}U_i$ is an open covering and 
we have a family of $G$-valued transition functions $\psi_{ji}\colon U_i\cap U_j\to G$ with respect to the covering $\{U_i\}_{i\in I}$. 
Then there exists a principal \Q-bundle $\beta\colon P\to N$ whose transition functions with respect to the above open covering of $N$ are the functions $\psi_{ji}$. 
\end{proposition}

\begin{proof}
Since $G$ is a \Q-group, there exist a Lie group $\widetilde{G}$ with a pseudo-discrete central subgroup $\Lambda\subseteq \widetilde{G}$ with $G=\widetilde{G}/\Lambda$. 
The corresponding quotient map 
$$\pi_G\colon \widetilde{G}\to G, \quad \widetilde{g}\mapsto \widetilde{g}\Lambda,$$ 
is also a \Q-atlas. 
(See Remark~\ref{rem:Q-gr}.) 

We then consider the disjoint union of smooth manifolds 
$$\widetilde{P}:=\bigsqcup_{i\in I}(\{i\}\times U_i\times\widetilde{G})
$$
with the equivalence relation defined by 
\begin{equation}\label{ildePsim}
(i_1,n_1,\widetilde{g}_1)\sim(i_2,n_2,\widetilde{g}_2)
\iff 
\begin{cases}
n_1=n_2=:n\in U_{i_1}\cap U_{i_2} & \\
\pi_G(\widetilde{g}_2)=\psi_{i_2i_1}(n)\pi_G(\widetilde{g}_1).
\end{cases}
\end{equation}
We will prove that the corresponding quotient map 
$$\pi_P\colon \widetilde{P}\to P$$ is a \Q-atlas, 
the mapping 
$$\beta\colon P\to N,\quad \beta(\pi_P(i,n,\widetilde{g})):=n$$
is a principal \Q-bundle with respect to the right action  
\begin{equation}\label{actionONP}
A\colon P\times G\to G,\quad (\pi_P(i,n,\widetilde{g}_1),\pi_G(\widetilde{g}_2))\mapsto \pi_P(i,n,\widetilde{g}_1)\pi_G(\widetilde{g}_2)=:\pi_P(i,n,\widetilde{g}_1\widetilde{g}_2)
\end{equation}
and its transition functions with respect to the open covering $N=\bigcup\limits_{i\in I}U_i$ are $\psi_{ji}$. 
In fact, we only need to prove that $\pi$ is a \Q-atlas, and then the other assertions 
are straightforward. 
(See the proof of  \cite[Ch. I, Prop. 5.2]{KN63} for more details in the classical case.)

For proving that $\pi$ is a \Q-atlas, we first check the existence of transversal local diffeomorphisms. 
To this end assume $i_1,i_2\in I$, $n_1=n_2=:n_0\in U_{i_1}\cap U_{i_2}$, and
\begin{equation}
\label{bundle_proof_eq1} 
\pi_G(\widetilde{g}_2)=\psi_{i_2i_1}(n)\pi_G(\widetilde{g}_1). 
\end{equation}
Since  $\psi_{i_2i_1}\colon U_{i_1}\cap U_{i_2}\to G$  is \Q-smooth, 
there exist  
an open subset $U\subseteq U_{i_1}\cap U_{i_2}$ with $n\in U$ and 
a smooth mapping $\widetilde{\psi}_{i_2i_1}\colon U_{i_1}\cap U_{i_2}\to\widetilde{G}$ with $\pi_G\circ \widetilde{\psi}_{i_2i_1}=\psi_{i_2i_1}\vert_U$. 
Then the equality~\eqref{bundle_proof_eq1} 
implies $\pi_G(\widetilde{g}_2)=\pi_G(\widetilde{\psi}_{i_2i_1}(n_0)\widetilde{g}_1)\in G=\widetilde{G}/\Lambda$, 
hence there exists $\lambda\in\Lambda$ with $\widetilde{g}_2=\lambda\widetilde{\psi}_{i_2i_1}(n_0)\widetilde{g}_1$. 
(Recall from Remark~\ref{rem:Q-gr} that the pseudo-discrete subgroup $\Lambda$ is contained in the center of~$\widetilde{G}$.)
If we define 
$$h\colon U\times\widetilde{G}\to U\times\widetilde{G},\quad (n,\widetilde{g})\mapsto (n,\lambda\widetilde{\psi}_{i_2i_1}(n)\widetilde{g})$$
then it is clear that $h$ is a diffeomorphism satisfying $h(n,\widetilde{g}_1)=(n,\widetilde{g}_2)$ and moreover $h(n,\widetilde{g})\sim (n,\widetilde{g})$ for every $(n,\widetilde{g})\in U\times\widetilde{G}$, 
hence $h$ is a transversal local diffeomorphism. 

To complete the proof of the fact that $\pi$ is a \Q-atlas, let $T$ be any smooth manifold and assume that $\gamma_1,\gamma_2\colon T\to X$ are smooth mappings satisfying $\pi\circ\gamma_1=\pi\circ\gamma_2$, that is, 
\begin{equation}
\label{bundle_proof_eq2}
(\forall t\in T)\quad \gamma_1(t)\sim\gamma_2(t). 
\end{equation}
We must show that if $t_0\in T$ has the property $\gamma_1(t_0)=\gamma_2(t_0)$, 
then $t_0$ has an open neighborhood $T_0\subseteq T$ with $\gamma_1(t)=\gamma_2(t)$ for every $t\in T_0$. 
To this end we write 
$$\gamma_j(t)=:(i_j,n_j(t),\widetilde{g}_j(t))\text{ for }t\in T\text{ and }j=1,2.$$  
The hypothesis \eqref{bundle_proof_eq2} actually implies $n_1(t)=n_2(t)=:n(t)$ for every $t\in T$. 
On  the other hand, the hypothesis $\gamma_1(t_0)=\gamma_2(t_0)$ implies 
$i_1=i_2=:i$ and $\widetilde{g}_1(t_0)=\widetilde{g}_2(t_0)$. 

In particular $n(t_0)\in U_i$ and we select 
a connected open neighborhood $T_0$ of $t_0\in T$ with 
$n(t)\in U$ for every $t\in T_0$. 
Since $\gamma_1(t)\sim\gamma_2(t)$ and $i_1=i_2=i$, 
we then obtain $q(\widetilde{g}_2(t))=
q(\widetilde{g}_1(t))$ 
hence there exists $\lambda(t)\in\Lambda=\Ker \pi_G$ with 
\begin{equation}
\label{bundle_proof_eq3}
\widetilde{g}_2(t)=\lambda(t)
\widetilde{g}_1(t)
\text{ for every }t\in T_0.
\end{equation}
Since the mappings $\widetilde{g}_2(\cdot),
\widetilde{g}_1(\cdot)$ are smooth, it then follows that the mapping $\lambda(\cdot)$ is smooth. 
On the other hand, this mapping takes values in the pseudo-discrete subgroup $\Lambda$ and its domain $T_0$ is connected, hence $\lambda(\cdot)$ is a constant mapping. 
We have noted above that 
$\widetilde{g}_1(t_0)=\widetilde{g}_2(t_0)$, 
hence $\lambda(t_0)=\1\in\widetilde{G}$, and then $\lambda(t)=\1\in\widetilde{G}$ for every $t\in T_0$. 
Then \eqref{bundle_proof_eq3} implies 
$\widetilde{g}_2(t)=\widetilde{g}_1(t)$ for every $t\in T_0$. 
Thus $\gamma_1(t)=\gamma_2(t)$ for every $t\in T_0$, and this completes the proof. 
\end{proof}

\begin{remark}
	\label{principalbundleQchart}
	\normalfont
 Let $G$ be a \Q-group and  $\beta:P\to N$ be a principal \Q-bundle.  
 Using the notation in the proof of Proposition~\ref{bundle},  
 let  $N=\bigcup\limits_{i\in I}U_i$ be an open covering  such that $P\vert_{U_i}$ is trivializable for every $i\in I$ 
with  $G$-valued transition functions $\psi_{ji}\colon U_i\cap U_j\to G$. 
Let $\widetilde{G}$ be a Lie group with a pseudo-discrete normal subgroup $\Lambda\subseteq G$ 
such that the quotient morphism $\pi_G:\widetilde{G}\to G=\widetilde{G}/\Lambda$  is a \Q-atlas. 
(See Remark~\ref{rem:Q-gr}.)
We consider the associated manifold 
$$\widetilde{P}:=\bigsqcup\limits_{i\in I}(\{i\}\times U_i\times\widetilde{G})
=\Bigl(\bigsqcup_{i\in I}(\{i\}\times U_i)\Bigr)\times\widetilde{G}$$ 
provided with the equivalence relation \eqref{ildePsim}. 
Then $P$ can be identified with $\widetilde{P}/\sim$ and we have the projection $\pi_P:
\widetilde{P}\to P$ as defined in the 
 proof of Proposition~\ref{bundle}. 
In particular, $\pi_P\colon \widetilde{P}\to P$ is a \Q-atlas.  
We have  a smooth right action 
$$\widetilde{A}\colon 
\widetilde{P}\times\widetilde{G}\to \widetilde{P}, 
\quad ((i,n,\widetilde{g}_1),\widetilde{g}_2)\mapsto (i,n, \widetilde{g}_1\widetilde{g}_2).$$ 
 This action is free,  and  if we define the smooth manifold 
 $$\widetilde{N}:= \bigsqcup\limits_{i\in I}(\{i\}\times U_i)$$ 
 then we have the surjective submersion 
 $$\widetilde{\beta}\colon \widetilde{P}\to \widetilde{N},\quad 
 (i,n,\widetilde{g})\mapsto (i,n)$$
 whose fibres are exactly the $\widetilde{G}$-orbits in~$\widetilde{P}$. 
 Therefore $\widetilde{\beta}\colon \widetilde{P}\to \widetilde{N}$ is a (trivial) smooth principal bundle with structure group~$ \widetilde{G}$.

  On the other hand,   the projection $\pi_P\colon\widetilde{P}\to P$ 
  intertwines the  action $\widetilde{A}$ and the action $A$ of $G$ on $P$, that is, 
  cf. \eqref{actionONP}, 
$$\pi_P\circ \widetilde{A}(((i,n,\widetilde{g}_1),\widetilde{g}_2)=A(\pi_P((i,n,\widetilde{g}_1)),\pi_G(\widetilde{g}_2)).$$
 If   we consider the equivalence relation on $\widetilde{N}$
\begin{equation}\label{TildePsim}
(i_1,n_1
)\sim_N 
(i_2,n_2
)
\iff 
n_1=n_2=:n\in U_{i_1}\cap U_{i_2} 
\end{equation}
 then the corresponding quotient map $\pi_N:\widetilde{N}\to N$ is a \Q-atlas for $N$, 
 that is, an \'etale smooth map since both $\widetilde{N}$ and $N$ are smooth manifolds.

 Thus we have the following commutative diagram:
$$\xymatrix{
\widetilde{P} \ar[d]_{\widetilde{\beta}} \ar[r]^{\pi_P} &P\ar[d]^\beta\\
\widetilde{N} \ar[r]^{\pi_N}         &N}
$$
Moreover,  since $\pi_G$ is a homomorphism of \Q-groups, it follows that $\pi_P$ is a homomorphism of principal \Q-bundles over $\pi_N$. 
	\end{remark}

\begin{example}[associated bundles]
\label{bundle_assoc}
\normalfont
Assume the setting  of Remark~\ref{local}. 
Let $\Vc$ be a finite-dimensional real vector space and $\Theta\colon G\to\GL(\Vc)$ be a \Q-smooth group homomorphism. 
Then $\Theta\circ \psi_{ji}\colon U_i\cap U_j\to G$ is a family of $\GL(\Vc)$-valued transition functions with respect to the covering $\{U_i\}_{i\in I}$, 
hence they give rise via Proposition~\ref{bundle} to a principal bundle $\beta_\Theta\colon P_\Theta\to N$ with structure group~$\GL(\Vc)$. 
Since $\GL(\Vc)$ is a Lie group, it is easily seen that $\beta_\Theta\colon P_\Theta\to N$ is a smooth principal bundle. 
Using the tautological action of $\GL(\Vc)$ on $\Vc$, 
we can define the associated bundle $ P_\Theta\times_{\GL(\Vc)}\Vc\to N$ 
as in \cite[pages 54--55]{KN63}. 
This is a smooth vector bundle over $N$ and we will use the notation $P\times_\Theta\Vc:=P_\Theta\times_{\GL(\Vc)}\Vc$. 
The projection of this vector bundle will be denoted 
$$\beta_{\Theta,\Vc}\colon P\times_\Theta\Vc\to N.$$ 
It follows by the construction of an associated bundle that for every $i\in I$ we have 
a trivialization 
$$\tau_i\colon U_i\times \Vc\to \beta_{\Theta,\Vc}^{-1}(U_i)$$ 
and, if $j\in I$ with $U_i\cap U_j\ne\emptyset$, then we have the change of coordinates 
\begin{align*}
\tau_j\circ\tau_i^{-1}\vert_{(U_i\cap U_j)\times\Vc}\colon 
 (U_i\cap U_j)\times & \Vc\to (U_i\cap U_j)\times\Vc,\\
(\tau_j\circ\tau_i^{-1}) (n, & v)=(n,\Theta(\psi_{ji}(n))v).
\end{align*}
\end{example}

\section{The Atiyah functor for principal \Q-bundles}
\label{Sect4}

In this section we construct the Atiyah sequence of a principal \Q-bundle 
and we explore its functorial properties. 
These will play an important role in Section~\ref{Sect8} in 
the description of the Lie algebroid associated to the gauge groupoid 
of a principal \Q-bundle. 

\begin{definition}
	\label{AP_def}
	\normalfont 
Assume the setting of Definition~\ref{Q-bundle_def}. 
By differentiation with respect to $p\in P$, the \Q-smooth group action \eqref{Q-bundle_def_eq1} leads to a (right) \Q-smooth group action
\begin{equation}
\label{AP_def_eq1}
TP\times G\to TP. 
\end{equation}
We consider the corresponding quotient map 
$$q_{TP}\colon TP\to (TP)/G=:\Atiyah(P)$$
and the mapping 
$$\Ag(\beta)\colon \Atiyah(P)\to N, \quad vG\mapsto (\beta\circ\tau_P)(v)$$
for which the diagram 
\begin{equation*}
	\xymatrix{
		& TP \ar[d]_{q_{TP}} \ar[rr]^{\tau_P}& & P \ar[d]^{\beta} &\\
		\Atiyah(P) =\hskip-25pt & (TP)/G \ar@{.>}[rr]^{\ \ \Ag(\beta)} & & N  &\hskip-25pt \simeq P/G
	}
\end{equation*}
is commutative, where $\tau_P\colon TP\to P$ is the tangent bundle of the \Q-manifold $P$.

We also define 
\begin{equation}
	\label{AP_def_eq2}
a_P\colon \Atiyah(P)\to TN, \quad a_P(vG):=(T\beta)(v),
\end{equation}
which is well defined because of the property $\beta(pg)=\beta(p)$ for all $p\in P$ and $g\in G$, which implies $(T\beta)(vg)=(T\beta)(v)$ for all $v\in TP$ and $g\in G$, 
where we use the group action~\eqref{AP_def_eq1}. 
Finally, we define 
$$\iota_P\colon P\times_{\Ad_G}\gg\to\Atiyah(P)$$
where $\iota_P((p,w)G)$ is the image of the tangent vector $(0,w)\in T_{p,\1}(P\times G)$ through the tangent map of~\eqref{Q-bundle_def_eq1}. 
Again, it is easily seen that the mapping~$\iota_P$ is well defined. 
\end{definition}

\begin{definition}
\label{AP_def_morph_algbd}
\normalfont
As in Definition~\ref{AP_def_morph}, 
let $\beta_j\colon P_j\to N_j$ be a principal \Q-bundle with structure group~$G_j$ for $j=1,2$, 
and assume that we have a morphism of principal \Q-bundles 
defined by a \Q-smooth map $F\colon P_1\to P_2$ over a smooth map $f\colon N_1\to N_2$ 
with respect to a morphism of \Q-groups 
$\Theta\colon G_1\to G_2$. 
Then it is easily seen that the tangent map $TF\colon TP_1\to TP_2$ is equivariant 
with respect to the actions of $G_1$ on $TP_1$ and $G_2$ on $TP_2$ as in \eqref{AP_def_eq1}, 
hence there exists a mapping $\Atiyah(F)\colon \Atiyah(P_1)\to \Atiyah(P_2)$ for which the diagram 
\begin{equation*}
	\xymatrix{
		& TP_1 \ar[d]_{q_{TP_1}} \ar[rr]^{TF}& & TP_2 \ar[d]^{q_{TP_2}} &\\
		\Atiyah(P_1) =\hskip-25pt & (TP_1)/G_1 \ar@{.>}[rr]^{\ \Atiyah(F)} & & (TP_2)/G_2  &\hskip-25pt = \Atiyah(P_2)
	}
\end{equation*}
is commutative. 
\end{definition}

\begin{example}
\label{triv}
\normalfont
Consider the trivial principal \Q-bundle $\beta=\pr_1\colon N\times G\to N$, 
where $N$ is a smooth manifold and $G$ is a \Q-group with its Lie algebra~$\gg$ 
and its group operation $\bm\colon G\times G\to G$. 
We have the $G$-equivariant \Q-diffeomorphism 
$$\id_{TN}\times\rho_G\colon T(N\times G)=TN\times TG\to TN\times (\gg\rtimes_{\Ad_G} G)$$
by Remark~\ref{TG_struct} hence we obtain the bijective map 
\begin{align*}
(\id_{TN}\times \rho_G)/G\colon & \Atiyah(N\times G)=T(N\times G)/G\to TN\oplus(N\times\gg), \\ 
& 
(X,v)G\mapsto X\oplus \omega_G(v)
\end{align*}
with its inverse map 
\begin{equation}
	\label{triv_eq1}
\chi\colon TN\oplus(N\times\gg)\to T(N\times G)/G=\Atiyah(N\times G)
\end{equation} 
given by 
\begin{equation}
	\label{triv_eq2}
	\chi(X\oplus v):=(X,v)G=\{X\}\times 	vG
\end{equation}
Now assume that we have, for $j=1,2$, a trivial bundle $\beta_j=\pr_1\colon N_j\times G_j\to N_j$, 
where $N_j$ is a smooth manifold and $G_j$ is a \Q-group with its Lie algebra~$\gg_j$ 
and its group operation $\bm_j\colon G_j\times G_j\to G_j$. 
If $\Theta\colon G_1\to G_2$ is a morphism of \Q-groups, $\psi\colon N_1\to G_2$ is a \Q-smooth map, and $f\colon N_1\to N_2$ is a smooth map, then 
it is easily checked that the mapping 
$$F\colon N_1\times G_1\to N_2\times G_2,\quad F(n_1,g_1):=(f(n_1),\psi(n_1)\Theta(g_1))$$
is a morphism of principal \Q-bundles over the map $f\colon N_1\to N_2$ 
with respect to the morphism of \Q-groups 
$\Theta\colon G_1\to G_2$, and conversely, every morphism $F$ of principal \Q-bundles over the map $f$ with respect to the morphism of \Q-groups~$\Theta$ looks as above for a suitable, uniquely determined, \Q-smooth map  $\psi\colon N_1\to G_2$. 

In order to compute $\Atiyah(F)$ we first note that 
the mapping 
\begin{align*}
\widetilde{TF}
 :=(\id_{TN_2}\times\rho_{G_2}) & \circ TF\circ(\id_{TN_1}\times\rho_{G_1})^{-1} \\
& \colon TN_1\times (\gg_1\rtimes_{\Ad_{G_1}} G_1)\to TN_2\times (\gg_2\rtimes_{\Ad_{G_2}} G_2)
\end{align*}
is given by 
\begin{align*} 
	\widetilde{TF} & (X_1,v_1g_1) \\
	&=
	((Tf)(X_1),T\bm_2((Tf)(X_1),(T_\1\Theta)(v_1)\Theta(g_1))) \\
	&\mathop{=}\limits^{\eqref{TG_struct_eq1}}
	((Tf)(X_1),\underbrace{((T\psi)(X_1)\psi(n_1)^{-1}+ \Ad_{G_2}(\psi(n_1))(T_\1\Theta)(v_1))}_{\in\gg_2}\underbrace{(\psi(n_1)\Theta(g_1))}_{\in G_2})
\end{align*}
for arbitrary $n_1\in N_1$, $X_1\in T_{n_1}N_1$, $g_1\in G$, and $v_1\in T_\1G_1=\gg_1$. 

It then follows by \eqref{triv_eq2} that the mapping 
$$\Atiyah(F)\colon TN_1\oplus(N_1\times\gg_1) \to TN_2\oplus(N_2\times\gg_2)$$ 
is given by the formula 
\begin{equation}
	\label{triv_eq3}
	(\Atiyah(F))(X_1,v_1)=((Tf)(X_1),(T\psi)(X_1)\psi(n_1)^{-1}+ \Ad_{G_2}(\psi(n_1))(T_\1\Theta)(v_1))
\end{equation}
for $n_1\in N_1$, $X_1\in T_{n_1}N_1$, $g_1\in G$, and $v_1\in T_\1G_1=\gg_1$.

Using Remark~\ref{trivLiealg_morph}
we now check that the mapping $\Atiyah(F)$ is a morphism of trivial Lie algebroids 
over the map $f\colon N_1\to N_2$. 
To this end we first note that 
\begin{align}
	f^{!!} & (TN_2\oplus(N_2\times\gg_2)) \nonumber \\
	&=\{(X_1,(X_2,v_2))\in TN_1\times (TN_2\oplus(N_2\times\gg_2))\mid (Tf)(X_1)=X_2 \} 
	\nonumber \\
	\label{triv_eq4}
	&\simeq\{(X_1,v_2)\mid X_1\in TN_1, v_2\in\gg_2\} \\
	&= TN_1\oplus(N_1\times\gg_2). \nonumber
\end{align}
Then we have 
\begin{align}
\Atiyah(F)^{!!}  \colon TN_1\oplus & (N_1\times \gg_1) \to TN_1\oplus(N_1\times\gg_2), 
\nonumber \\
\Atiyah(F)^{!!}(X_1,v_1)& =(X_1,\Atiyah(F)(X_1,v_1)) \nonumber \\
&\simeq (X_1,(T\psi)(X_1)\psi(n_1)^{-1}+ \Ad_{G_2}(\psi(n_1))(T_\1\Theta)(v_1))
\nonumber
\end{align}
Here $\simeq$ denotes the vector bundle isomorphism over $\id_{N_1}$ that was used in \eqref{triv_eq4}, which allows us to disregard $X_2$ in the triple $(X_1,(X_2,v_2))$ if $(Tf)(X_1)=X_2$. 
By Remark~\ref{trivLiealg_morph}, in order to prove that $\Atiyah(F)$ is a morphism of Lie algebroids over the map $f\colon N_1\to N_2$, it is necessary and sufficient to check that the above map $\Atiyah(F)^{!!}$ is a morphism of Lie algebroids 
over the identity map~$\id_{N_1}$. 
This last condition can be checked using 
Remark~\ref{trivLiealg_end}. 
(See also Equation~\eqref{deriv_eq1} in Remark~\ref{deriv}.)  
\end{example}

\begin{definition}
	\label{abstr_def}
	\normalfont 
	An \emph{abstract Atiyah sequence} over a smooth manifold~$N$ is a short exact sequence of Lie algebroids  
	\begin{equation}
		\label{abstr_def_eq1}
		\xymatrix{
			L\ \ar@{>->}[r]^{\inc}  & A\ar@{->>}[r]^{\anc\ \ }  & TN  
		}
	\end{equation}
	where $A$ is a transitive Lie algebroid with its anchor $\anc\colon A\to TN$ 
	and $L$ is a Lie algebra bundle over the manifold~$N$.  
	If \begin{equation*}
		\xymatrix{
			L'\ \ar@{>->}[r]^{\inc'}  & A'\ar@{->>}[r]^{\anc'\ } & TN 
		}
	\end{equation*}
	is another abstract Atiyah sequence over $N$, then a \emph{morphism of these abstract Atiyah sequences} consists of two morphisms of Lie algebroids $\psi\colon L\to L'$ and $\varphi\colon A\to A'$ satisfying $\varphi\circ\inc=\inc'\circ\psi$. 
\end{definition}

\begin{definition}
	\label{AP_categ}
	\normalfont
	We denote by 
	$\PB$ 
	the category whose objects are the principal \Q-bundles over smooth manifolds 
	having a \Q-group as the structure group, 
	while the morphisms of principal \Q-bundles are defined as in Definition~\ref{AP_def_morph}. 
	
	We also denote by 
	$\AB$ 
	the category whose objects are the abstract Atiyah sequences, while its morphisms are defined as in Definition~\ref{abstr_def}. 
\end{definition}

Now we can prove the following version of \cite[Th. 1]{At57} in the framework of \Q-manifolds. 

\begin{proposition}
\label{AP_prop}
In Definition~\ref{AP_def}, the set $\Atiyah(P)$ has the structure of a smooth manifold with the following properties: 
\begin{enumerate}[{\rm(i)}]
	\item\label{AP_prop_item1} 
	The quotient map $q_{TP}\colon TP\to \Atiyah(P)$ is a \Q-submersion. 
	\item\label{AP_prop_item2} 
	The map $\Ag(\beta)\colon \Atiyah(P)\to N$ is a smooth vector bundle. 
	\item\label{AP_prop_item3} 
	We have the short exact sequence of smooth vector bundles 
	\begin{equation}
	\label{AP_prop_eq1}
		\xymatrix{
			0 \ar[r]& P\times_{\Ad_G}\gg\ar[r]^{\iota_P}& \Atiyah(P)\ar[r]^{a_P} & TN \ar[r]& 0. 
		}
	\end{equation}
\item\label{AP_prop_item4} 
The smooth vector bundle $\Atiyah(P)$ has the structure of a transitive Lie algebroid with its anchor map $a_P\colon\Atiyah(P)\to TN$. 
\end{enumerate}
\end{proposition}

\begin{proof}
We resume the notation of Remark~\ref{local} for defining 
the smooth manifold structure of~$\Atiyah(P)$. 
For every $i\in I$ we have the $G$-equivariant \Q-diffeomorphism $\psi_i\colon\beta^{-1}(U_i)\to U_i\times G$, whose tangent map 
$$T\psi_i\colon T(\beta^{-1}(U_i))\to T(U_i\times G)=T(U_i)\times TG$$ 
is a $TG$-equivariant \Q-diffeomorphism. 

We define 
\begin{equation}
	\label{AP_prop_proof_eq1}
\chi_i\colon TU_i\oplus(U_i\times \gg)\to (TP)/G=\Atiyah(P)
\end{equation}
by 
\begin{equation}
\label{AP_prop_proof_eq2}
\chi_i(w,v):=(T\psi_i)^{-1}(\{w\}\times 
\omega_G^{-1}(v))
\mathop{=}\limits^{\eqref{TG_struct_eq2}}
(T\psi_i)^{-1}(\{w\}\times 
vG)
\end{equation}
It easily follows by \eqref{TG_struct_eq1} that $\chi_i(w,v)\in (TP)/G$ for every $(w,v)\in T(U_i)\oplus(U_i\times \gg)$. 
Moreover, it is easily checked that $\chi_i$ is injective and 
$$\chi_i(T(U_i)\oplus(U_i\times \gg))=(T(U_i\times G))/G=\Ag(\beta)^{-1}(U_i).$$
Now, if $i,j\in I$ with $U_i\cap U_j\ne\emptyset$, then 
$$T\psi_j\circ (T\psi_i)^{-1}=T(\psi_j\circ\psi_i^{-1})\colon T(U_i\cap U_j)\times TG
\to T(U_i\cap U_j)\times TG$$
where the mapping $\psi_j\circ\psi_i^{-1}\colon (U_i\cap U_j)\times G\to (U_i\cap U_j)\times G$ is given by \eqref{local_eq3}. 
 
It then follows by 
\eqref{triv_eq3} in Example~\ref{triv}
that the mapping 
$$\chi_j^{-1}\circ\chi_i\vert_{T(U_i\cap U_j)\times\gg}\colon T(U_i\cap U_j)\times\gg \to T(U_i\cap U_j)\times\gg$$ 
is given by the formula 
\begin{equation}
\label{AP_prop_proof_eq3}
(\chi_j^{-1}\circ\chi_i)(w,v)=(w,T\psi_{ji}(w)\psi_{ji}(n)^{-1}+ \Ad_G(\psi_{ji}(n))v)
\end{equation}
for $n\in U_i\cap U_j$, $w\in T_n(U_i\cap U_j)$, and $v\in \gg$.  

We now check the remaining assertions in the statement. 

\eqref{AP_prop_item1} 
For every $i\in I$, we note the following diagram which is commutative 
by~\eqref{AP_prop_proof_eq2},
$$\xymatrix{
	T(U_i)\times TG \ar[rr]^{\id_{T(U_i)}\times\omega_G}& &T(U_i)\oplus(U_i\times \gg) \ar[d]^{\chi_i}\\
	T(\beta^{-1}(U_i))  \ar[u]^{T\psi_i} \ar[rr]^{q_{TP}\vert_{T(\beta^{-1}(U_i)) }}& &\Ag(\beta)^{-1}(U_i)
}$$
and whose vertical arrows are \Q-diffeomorphisms while the upper horizontal arrow is a \Q-submersion.  
It follows that the lower horizontal arrow in the above diagram is a \Q-submersion as well, and this completes the proof of the fact that $q_{TP}$ is a \Q-submersion. 

\eqref{AP_prop_item2}
It follows by the above construction that the mappings \eqref{AP_prop_proof_eq1}--\eqref{AP_prop_proof_eq2} constitute a smooth vector bundle atlas for $\Ag(\beta)\colon \Atiyah(P)\to N$ and its fibre $\Ag(\beta)^{-1}(n)$ is isomorphic to $T_nN\times\gg$ for arbitrary $n\in N$. 

\eqref{AP_prop_item3}
For every $i\in I$ we have the commutative diagram of vector bundles over $U_i$, 
in which we use notation for associated vector bundles from Example~\ref{bundle_assoc}
\begin{equation*}
	\xymatrix{
		0 \ar[r]& U_i\times\gg\ar[r] \ar[d]_{\tau_i}& T(U_i)\oplus(U_i\times \gg) \ar[d]_{\chi_i} \ar[r] & T(U_i) \ar[d]^{\id_{T(U_i)}} \ar[r]& 0\\
		0 \ar[r]& \beta_{\Ad_G,\gg}^{-1}(U_i)\ar[r]^{\iota_P} & \Ag(\beta)^{-1}(U_i)\ar[r]^{\quad a_P} & T(U_i) \ar[r]& 0
	}
\end{equation*}
whose upper row consists of the obvious maps and is a short exact sequence, 
hence the lower row is a short exact sequence as well. 
This shows that the sequence of vector bundles~\eqref{AP_prop_eq1} is exact. 

\eqref{AP_prop_item4} 
It follows by Example~\ref{triv} that the mapping in \eqref{AP_prop_proof_eq3} is a morphism of trivial Lie algebroids, 
that is, the family $(\chi_i\colon TU_i\oplus(U_i\times\gg)\to\Atiyah(P))_{i\in I}$ is a Lie algebroid atlas of~$\Atiyah(P)$. 
Thus the Lie algebroid structure of $\Atiyah(P)$ is obtained by an application of Lemma~\ref{atlas_lemma}.
\end{proof}

\begin{remark}[localization]
	\label{AP_loc}
	\normalfont 
	Assume the setting of Definition~\ref{AP_def} and fix an arbitrary open subset $U\subseteq N$. 
	Denote $P_U:=\beta^{-1}(U)\subseteq P$ and $\beta_U:=\beta\vert_{P_U}\colon P_U\to U$. 
	Also consider the inclusion maps $F_U\colon P_U\hookrightarrow P$ and $f_U\colon U\hookrightarrow N$. 
	
	For the vector bundle $\Ag(\beta)\colon \Atiyah(P)\to N$ in  Proposition~\ref{AP_prop}\eqref{AP_prop_item2}, consider its restricted vector bundle 
	$\Atiyah(P)_U:=\Ag(\beta)^{-1}(U)$ with the vector bundle projection 
	$\Ag(\beta)_U:=\Ag(\beta)\vert_{\Atiyah(P)_U}\colon \Atiyah(P)_U\to U$. 
	Then it is easily seen that the mapping $\Atiyah(F_U)\colon \Atiyah(P_U)\to \Atiyah(P)$ 
	given by Definition~\ref{AP_def_morph_algbd} (for $\Theta:=\id_G$) 
	gives an isomorphism of vector bundles 
	\begin{equation*}
	\xymatrix{
	\Atiyah(P_U) \ar[d]_{\Ag(\beta)_U} \ar[r]^{\Ag(F_U)}& \Atiyah(P)_U \ar[d]^{\Ag(\beta)\vert_{\Atiyah(P)_U}} \\
	U \ar[r]^{\id_U}& U
}
	\end{equation*}
	
\end{remark}

\begin{proposition}
\label{AP_morph_algbd_prop}
In Definition~\ref{AP_def_morph_algbd}, the mapping $\Ag(F)\colon \Atiyah(P_1)\to\Atiyah(P_2)$ is a morphism of Lie algebroids over the base map $f\colon N_1\to N_2$. 
\end{proposition}

\begin{proof}
We first note that the diagram 
\begin{equation*}
\xymatrix{
\Atiyah(P_1) \ar[d]_{a_{P_1}} \ar[r]^{\Ag(F)} & \Atiyah(P_2) \ar[d]^{a_{P_2}}  \\
TN_1 \ar[r]_{Tf} & TN_2
}
\end{equation*}
is commutative since for every $v\in TP_1$ we have 
\begin{align*}
(Tf\circ a_{P_1})(vG_1)
&\mathop{=}\limits^{\eqref{AP_def_eq2}}
Tf(T\beta_1(v))=(T(f\circ\beta_1))(v) =(T(\beta_2\circ F))(v) \\
& =T\beta_2(TF(v))
\mathop{=}\limits^{\eqref{AP_def_eq2}}a_{P_2}(TF(v)G_2)
=a_{P_2}(\Ag(F)(vG_1))
\end{align*}
where the last equality follows by Definition~\ref{AP_def_morph_algbd}. 

In order to complete the proof we use the localization property of Lie algebroid morphisms as shown in Remark~\ref{trivLiealg_morph_loc}. 
It thus follows that it suffices to show that the restriction of $\Ag(F)$ from $\Atiyah(P_1)_U$ to $\Atiyah(P_2)_V$ is a morphism of Lie algebroids for every $U$ in an open covering $\Uc$ of $N_1$ and $V$ in an open covering $\Vc$ of $N_2$ with $f(U)\subseteq V$. 
We select an open covering $\Uc$ of $N_1$ such that $(P_1)_U$ is a trivial principal \Q-bundle for every $U\in \Uc$ and an open covering $\Vc$ of $N_2$ such that $(P_2)_V$ is a trivial principal \Q-bundle for every $V\in \Vc$. 
Using Remark~\ref{AP_loc}, it follows that it suffices to show that $\Ag(F)\colon\Atiyah(P_1)\to\Atiyah(P_2)$ is a Lie algebroid morphism if both principal \Q-bundles $P_1$ and $P_2$ are trivial. 
In this special case the assertion follows by Example~\ref{triv}, and we are done. 
\end{proof}

\begin{definition}
\label{AP_functor}
\normalfont 
For every principal \Q-bundle $\beta\colon P\to N$ with structure group~$G$, 
its \emph{Atiyah sequence} is the short exact sequence~\eqref{AP_prop_eq1} 
consisting of Lie algebroids over~$N$, where $P\times_{\Ad_G}\gg$ is a Lie algebra bundle with its fibre~$\gg$. 

The \emph{Atiyah functor} is the correspondence 
\begin{equation*}
\Atiyah\colon \PB\to\AB
\end{equation*}
defined on objects as in Definition~\ref{AP_def} and on morphisms 
as in Definition~\ref{AP_def_morph_algbd} and Proposition~\ref{AP_morph_algbd_prop}. 
\end{definition}


\section{Integration of abstract Atiyah sequences}
\label{Sect5}

The main result of this section is an illustration of 
the role that the \Q-manifolds play in integration problems on Lie algebroids. 
Specifically, we prove that every abstract Atiyah sequence arises from some principal \Q-bundle  (Theorem~\ref{AP_th}).

\begin{definition}
\label{cover_def}
\normalfont 
A \emph{simple open covering} of a smooth manifold is an open covering $(U_i)_{i\in I}$ with the property that for every finite subset $I_0\subseteq I$ the open set $\bigcap\limits_{i\in I_0}U_i$ is either empty or contractible. 
\end{definition}

\begin{remark}
\label{cover_ref}
\normalfont
As noted e.g. in \cite[\S 1.2, page 91]{Ko70}, for every smooth manifold $N$ one can construct simple open coverings using convex neighborhoods with respect to any Riemannian metric. 
This argument also shows that if $N$ is second countable then it has a simple open covering which is countable. 
\end{remark}

\begin{theorem}
\label{AP_th}
Every abstract Atiyah sequence over a second countable, smooth  manifold is isomorphic to the Atiyah sequence of some principal \Q-bundle. 
\end{theorem}

\begin{proof}
We consider an abstract Atiyah sequence denoted as in \eqref{abstr_def_eq1}.  
Using Remark~\ref{cover_ref}, we select a countable simple open covering $(U_i)_{i\in\NN}$ of $N$. 
Then there exists a finite-dimensional real Lie algebra $\gg$ 
such that for every $i\in\NN$ there exists a Lie algebroid isomorphism 
\begin{equation*}
S_i\colon TU_i\oplus(U_i\times\gg)\to A_{U_i}
\end{equation*}
with its corresponding flat connection $\Theta_i:=S\vert_{TU_i}\colon TU_i\to A_{U_i}$ and 
\LAB\ chart $\psi_i\colon U_i\times\gg\to L_{U_i}$ 
satisfying 
\begin{equation*}
\nabla^{\Theta_i}_X(\psi_i(v))=\psi_i(X(v))
\end{equation*}
for every smooth vector field $X\colon U_i\to TU_i$ and $v\in\Ci(U_i,\gg)$, 
as in \cite[\S 8.2, page 317]{Ma05}. 
If $i,j\in\NN$ and $U_{ij}:=U_i\cap U_j\ne\emptyset$, then there exist a $\gg$-valued 1-form 
$$\chi_{ij}\in\Omega^1(U_{ij},\gg)$$ 
and a smooth function 
$$a_{ij}\colon U_{ij}\to\Aut(\gg)$$ 
for which the overlap isomorphism 
$S_i^{-1}\circ S_j\colon TU_{ij}\oplus(U_{ij}\times\gg)\to TU_{ij}\oplus(U_{ij}\times\gg)$ is given by 
\begin{equation}
\label{AP_th_proof_eq1}
(S_i^{-1}\circ S_j)_x(X,v)=(X,\chi_{ij}(x)X+a_{ij}(x)v)\in T_x(U_{ij})\oplus\gg
\end{equation}
for all $x\in U_{ij}$, $X\in T_x(U_{ij})$, and $v\in\gg$. 
(See \cite[proof of Th. 8.2.4]{Ma05}.)

Let $\widetilde{G}$ be a connected, simply connected Lie group whose Lie algebra is (isomorphic to)~$\gg$, and denote by $Z\widetilde{G}$ the center of~$\widetilde{G}$ 
and by $\omega_{\widetilde{G}}\in\Omega^1(\widetilde{G},\gg)$ the right-invariant Maurer-Cartan form on~$\widetilde{G}$. 
Then there exist smooth functions 
$$s_{ij}\colon U_{ij}\to\widetilde{G}$$
satisfying
\begin{equation}
	\label{AP_th_proof_eq2}
s_{ij}^*(\omega_{\widetilde{G}})=\chi_{ij}
\quad\text{ and }\quad \Ad_{\widetilde{G}}\circ s_{ij}=a_{ij},
\end{equation} 
for which there exist elements $e_{ijk}\in Z\widetilde{G}$ satisfying 
$$s_{jk}(\cdot)s_{ik}(\cdot)^{-1}s_{ij}(\cdot)=e_{ijk}\text{ on }U_{ijk}.$$
(See \cite[page 325]{Ma05}.)

We now adapt the method of proof of \cite[Th. 8.3.2]{Ma05}.
Specifically, we denote by $D\subseteq Z\widetilde{G}$ the subgroup generated by the set $\{e_{ijk}\mid i,j,k\in\NN, U_{ijk}\ne\emptyset\}$. 
Then the subgroup $D$ is at most countable, 
hence pseudo-discrete in the sense of Remark~\ref{rem:Q-gr}, 
and then the quotient map 
$$p\colon \widetilde{G}\to \widetilde{G}/D=:G$$
is a \Q-atlas and $G$ is a \Q-group. 
Moreover, if we define 
$$\bar{s}_{ij}:=p\circ s_{ij}\colon U_{ij}\to G$$
then we have 
$$\bar{s}_{jk}(\cdot)\bar{s}_{ik}(\cdot)^{-1}\bar{s}_{ij}(\cdot)=\1\in G\text{ on }U_{ijk}.$$
Then Proposition~\ref{bundle} is applicable, and we obtain a principal \Q-bundle $\beta\colon P\to N$ with structure group~$G$, whose transition functions are $\bar{s}_{ij}\colon U_{ij}\to G$. 

We now check that the Atiyah sequence of the principal \Q-bundle $\beta\colon P\to N$ is isomorphic to the sequence~\eqref{abstr_def_eq1} which was our starting point. 
To this end we note that, by \eqref{AP_th_proof_eq1}--\eqref{AP_th_proof_eq2}, 
the overlap isomorphisms of the Lie algebroid atlas $(S_i\colon TU_i\oplus(U_i\times\gg)\to A_{U_i})_{i\in I}$ coincide with the overlap isomorphisms of the Lie algebroid atlas that defines the Lie algebroid $\Atiyah(P)$ (cf. \eqref{AP_prop_proof_eq3} in the proof of Proposition~\ref{AP_prop} for $\psi_{ji}=\bar{s}_{ji}$). 
Then, by  Lemma~\ref{atlas_lemma}, we see that the Lie algebroids $A$ and $\A(P)$ are isomorphic. 
\end{proof}

\section{\Q-groupoids and their corresponding Lie algebroids}
\label{Sect6}

In this section we introduce the \Q-groupoids (Definition~\ref{BQG}) and we develop their basic Lie theory to the extent that is needed for the purposes of integrating the transitive Lie algebroids. 
We recall that the action of the Lie functor on morphisms of Lie groupoids is nontrivial,  
due to algebraic aspects related to the compatibility with Lie brackets, 
cf. \cite[Ch. 4]{Ma05}. 
This construction encounters an additional problem of differential geometric nature in the case of \Q-groupoids, 
due to the fact that the space of units is not a submanifold as in the classical case, 
as discussed in Appendix~\ref{AppA}. 
This problem is avoided using some auxiliary properties of the vector distributions on \Q-manifolds which are established in Appendix~\ref{AppB}. 

\subsection{Groupoids}${}$\\
 Recall that a \emph{groupoid} is a small category in which every morphism is invertible. 
 Thus, a groupoid $ \Gc\tto M$ is a pair of sets $(\Gc,M)$ 
 along with several structure maps:
\begin{itemize}
\item[(G1)]  two   maps
$ \bs\colon  \Gc\to M$ and $ \bt\colon  \Gc\to M$ 
called \emph{source} and  \emph{target} maps, respectively; 

\item[(G2)] a map
 $\bm\colon \Gc^{(2)}\to \Gc$, $(g,h)\mapsto \bm(g,h)=:gh$, 
 called \emph{multiplication}, defined on $ \Gc^{(2)}:=\{(g,h)\in \Gc\times \Gc\mid  \bs(g) =  \bt(h)\}$, satisfying the associativity in the sense that the product $(gh)k$ is
defined if and only if $ g(hk)$  is defined and in this case  $(gh)k = g(hk)$; 

\item[(G3)] a map $\1 \colon M \to \Gc$ called  \emph{identity section} satisfying~$g\1_x=g$  if $g\in \bs^{-1}(x)$
and 
$g \1_x = g$ for all $g\in \bt^{-1}(x)$
(hence $ \bs\circ \1=\id_{M}= \bt\circ \1$, which shows that $\bs$ and $\bt$ are surjective, while $\1$ is injective);

\item[(G4)]  a map  $ \bi \colon  \Gc\to  \Gc$, $g\mapsto \bi(g)=:g^{-1}$, called  \emph{inversion}, satisfying 
$\bi\circ\bi=\id_\Gc$, 
$ \bs\circ \bi= \bt$, 
and  $gg^{-1}= \1_{\bt(g)}$, $g^{-1}g= \1_{\bs(g)}$ for all $g\in \Gc$.
\end{itemize}
The set $ M$ is  called  the \emph{base} of the groupoid, 
and $ \Gc$ 
is called the  \emph{total space} of the groupoid. 

For  $x,y\in M$, we denote 
$\Gc(x,-):=\bs^{-1}(x)$, $\Gc^y:=\Gc(-,y):=\bt^{-1}(y)$ 
and 
$$ \Gc(x,y):=\{g\in \Gc\mid  \bs(g)=x,  \bt(g)=y\}=\Gc(x,-)\cap \Gc(-,y).$$
For arbitrary $x\in M$ we also denote 
$$ \bt_x:=\bt\vert_{ \Gc(x,-)}\colon \Gc(x,-)\to M.$$
The \emph{isotropy group} of $x\in M$ is  
$$ \Gc(x):=\Gc(x,x)=\{ g\in  \Gc\mid  \bs(g)=x= \bt(g)\}= \bt_x^{-1}(x)\subseteq\Gc.$$
The   \emph{orbit} of $x\in M$ is 
$$ \Gc.x=\{ \bt(g)\mid g\in  \bs^{-1}(x)\}= \bt( \Gc(x,-))= \bt_x( \Gc(x,-))\subseteq M.$$
If $g\in  \Gc$, $ \bs(g)=x$, and $ \bt(g)=y$, 
one has
the left translation 
$L_g\colon  \Gc(-,x)\to  \Gc(-,y)$, $h\mapsto gh$, 
and the right translation $R_g\colon \Gc(x,-)\to  \Gc(y,-)$, $h\mapsto hg$. 
Both maps $L_g$ and $R_g$ are bijective.

Since every groupoid is in particular a category, a  \emph{groupoid morphism} is simply a functor. 
Equivalently, a groupoid morphism from the groupoid $  \Gc\tto M$ to another groupoid $\Hc\tto N$ is given by two maps $\Phi: \Gc\to \Hc$ and $\phi:M\to N$ which are compatible with the structure maps, that is, for all $(g,g')\in\Gc^{(2)}$ and $x\in M$ we have 
$\Phi(\bs(g))= \bs(\Phi(g))$, $\Phi(\bt(g))= \bt(\Phi(g))$, 
$\Phi(g)\Phi(g')=\Phi(gg')$, 
and 
$\Phi( \1_x)= \1_{\phi(x)}$. 
(In particular, $\phi$ is uniquely determined by $\Phi$.)

\subsection{\Q-groupoids}${}$\\
For the notion of \Q-groupoid, we need some auxiliary facts on \Q-manifolds. 
The following lemma is suggested by the fact that smooth retracts of smooth manifolds are embedded submanifolds, cf. \cite[Th. I.1.13]{KlMiSl93}. 

\begin{lemma}
\label{retr}
Let $S$ be a \Q-manifold with 
a subset $S_0\subseteq S$ that has the structure of a \Q-manifold for which the inclusion map $\iota\colon S_0\to S$ is a \Q-immersion. 
If there exists a \Q-smooth map $F\colon S\to S_0$ satisfying $F(s)=s$ for every $s\in S_0$, then $S_0$ is an immersed \Q-submanifold of $S$. 
\end{lemma}

\begin{proof}
Let $T$ be a smooth manifold and $f\circ T\to S$ be a \Q-smooth map 
with $f(T)\subseteq S_0$. 
For every $t\in T$ we have $f(t)\in S_0$, hence $F(f(t))=f(t)=\iota(f(t))$. 
Thus $F\circ f=\iota\circ f$. 
Both maps $f\circ T\to S$ and $F\colon S\to S_0$ are \Q-smooth, hence $F\circ f\colon T\to S_0$ is \Q-smooth, and then $\iota\circ f$ is \Q-smooth, too. 
This completes the proof of the fact that $S_0$ is an immersed \Q-submanifold of $S$. 
(See \cite[Ch. I, \S 2, no. 2, page 240]{Br73}.)
\end{proof}

A version of the following proposition,     
for infinite-dimensional smooth manifolds, 
can be found in \cite[Prop. 3.1]{BGJP19}.  
(See also \cite{BP22}.) 

\begin{proposition}\label{preliminaries} 
Let $ \Gc\tto M $ be a groupoid  satisfying the following conditions: 
\begin{enumerate}[{\rm 1.}]
\item $ \Gc$ is a \Q-manifold and $M$ is a smooth manifold;
\item the map $ \bs\colon \Gc\to M$ is a \Q-submersion;
\item the map $ \bi \colon \Gc  \to \Gc$ is \Q-smooth;
\item  the map $ \1\colon  M\to  \Gc$ is \Q-smooth.
\end{enumerate}
  \noindent Then $\bi$ is a \Q-diffeomorphism,  both $ \bs,\bt\colon \Gc\to M$ are surjective \Q-submersions, and, for every $x\in M$, the sets $ \Gc(x,-)$ and $ \Gc(-,x) $ are  immersed \Q-submanifolds of $ \Gc$.
  Moreover, we have: 
 \begin{enumerate}[{\rm(i)}]
\item
\label{preliminaries_item1}   The set $ \Gc^{(2)}$ is an immersed  \Q-submanifold of $ \Gc\times  \Gc$.
\item\label{preliminaries_item2} The image  
$ \1_M$ of the map $ \1 \colon M  \to \Gc$ is  an immersed  \Q-submanifold of  $\Gc$. 
\end{enumerate}
 \end{proposition}

  \begin{proof} 
 The map $ \bi\colon  \Gc\to  \Gc$ is \Q-smooth and is its own inverse, hence it is a \Q-diffeomorphism. 
Since  $ \bs\circ  \1= \id_M$, 
it follows that the map $\bs\colon \Gc\to M$ is  surjective. 
 Using the fact that 
  $ \bt= \bs\circ  \bi$ where $\bs$ is a \Q-submersion and $\bi$ is a \Q-diffeomorphism, 
  it follows that  $ \bt\colon \Gc\to M$ is \Q-smooth and actually a surjective \Q-submersion, too. 
  
Since $M$ is a smooth manifold, it follows that every singleton subset $\{x\}\subseteq M$ is an embedded submanifold. 
 Then, by Remark~\ref{rem:subm}\eqref{rem:subm_1}, 
the sets $ \Gc(x,-)=\bs^{-1}(\{x\})$ and $ \Gc(-,x)=\bt^{-1}(\{x\}) $ are immersed \Q-submanifolds of~$\Gc$. 
 
 We now prove the Assertions \eqref{preliminaries_item1}--\eqref{preliminaries_item2}. 
 
\eqref{preliminaries_item1}  The map $( \bs, \bt): \Gc\times  \Gc\to M\times M$  is a   \Q-submersion since we have just seen that both $\bs$ and $\bt$ are submersions. 
  	The diagonal $ \Dc\subseteq M\times M$,  is contained in the image of  $( \bs, \bt)$ and we have 
	 $ \Gc^{(2)}=( \bs, \bt)^{-1}(\Dc)$.
	  But $\Dc$ is an embedded submanifold of $M\times M$  and $( \bs, \bt)$ is a \Q-submersion hence this map is transverse to $\Dc$.
Therefore, by \cite[Ch. 1, \S 2, no. 5, Prop., page 243]{Br73},  $ \Gc^{(2)}$ is an immersed  \Q-submanifold of $ \Gc\times  \Gc$. 
  
\eqref{preliminaries_item2} 
Since $ \bs\circ  \1= \id_M$, where both maps $\bs\colon \Gc\to M$ and $\1\colon M\to \Gc$ are \Q-smooth and $\id_M\colon M\to M$ is an immersion, 
it follows by Remark~\ref{rem:subm}\eqref{rem:subm_4} 
that the map $\1\colon M\to \Gc$ is actually a \Q-immersion. 
Then, an application of Lemma~\ref{retr} for $S_0:=\1_M\subseteq\Gc=:S$ and the \Q-smooth map $F:=\1\circ \bs\colon\Gc\to\Gc$ shows that $\1_M$ is an immersed  \Q-submanifold of  $\Gc$. 
This completes the proof. 
%
%
%
	 \end{proof}
 
 \begin{remark}
 \normalfont 
 In the setting of Proposition~\ref{preliminaries}, if we assume only that $\bs\colon\Gc\to M$ is \Q-smooth, then it automatically follows that $\bs$ is a \Q-submersion at $\1_x\in \Gc$ for every $x\in M$
 by a suitable version of Remark~\ref{rem:subm}\eqref{rem:subm_2}, since  $ \bs\circ  \1= \id_M$, where both maps $\bs\colon \Gc\to M$ and $\1\colon M\to \Gc$ are \Q-smooth and $\id_M\colon M\to M$ is a submersion. 
 \end{remark}
  
  \begin{remark}
  \normalfont 
  In the setting of Proposition~\ref{preliminaries}, 
  if moreover $\Gc$ is a smooth manifold then the hypothesis that $ \bi \colon \Gc  \to \Gc$ is (\Q-)smooth follows from the other hypotheses, 
  as proved in \cite[Prop. 1.1.5]{Ma05}.
  \end{remark}

 \begin{definition}\label{BQG}
\normalfont 
A \emph{\Q-groupoid}  is a groupoid  $ \Gc\tto M$  
satisfying the following conditions:
\begin{itemize}
\item[(QG1)] $  \Gc$ is a \Q-manifold and $M$ is a smooth manifold.
 
\item[(QG2)] The structure maps $ \bs: \Gc\to M$, $ \bi \colon \Gc  \to \Gc$, 
and $ \1\colon  M\to  \Gc$ are \Q-smooth.

\item[(QG3)]  
The map $ \bs:  \Gc\to  M$ is  a \Q-submersion. 

 \item[(QG4)]  
The multiplication $ \bm: \Gc^{2}\to \Gc$  is smooth.
\end{itemize}

A   \emph{\Q-groupoid  morphism} between the \Q-Lie groupoids $  \Gc\tto M$ and $\Hc\tto N$ is a  groupoid morphism $\Phi\colon \Gc\to \mathcal{H}$ over $\phi\colon M\to N$
for which $\Phi$ is a \Q-smooth map. 
(Since $\phi$ is the composition of the maps $M\mathop{\to}\limits^{\1}\Gc\mathop{\to}\limits^{\Phi}\Hc\mathop{\to}\limits^{\bs}N$, 
it follows that $\phi$ is a smooth map.) 
\end{definition}

\begin{remark}\label{GroupoidvdS} 
\normalfont 
	A \Q-groupoid is a special case of a diffeological  groupoid  in the sense of \cite[\S 3]{vdS21}.
\end{remark}

\begin{remark}\label{OtherProperties}
\normalfont 
By  Proposition \ref{preliminaries},  the multiplication in Definition \ref{BQG} is defined on a \Q-manifold and a  \Q-groupoid $ \Gc\tto M$ has the properties in 
Proposition~\ref{preliminaries}. 
\end{remark}

 As in  the classical case of the category of smooth manifolds, we have (cf. e.g., \cite[Th. 5.4]{MoMr03}):

  \begin{proposition}\label{generalproperties} 
If  $ \Gc\tto M$ be  \Q-groupoid,  
then the following assertions hold.
 \begin{enumerate}[{\rm(i)}]
 \item  \label{generalproperties_item1} 
  If $g\in \Gc$, $ \bs(g)=x$ and $ \bt(g)=y$,  then both maps $L_g\colon  \Gc(-,x)\to  \Gc(-,y)$  and  $R_g\colon  \Gc(y,-)\to \Gc(x,-)$ are \Q-diffeomorphisms.

  \item \label{generalproperties_item2}  For any $x,y\in M$, the set $\Gc(x,y)=\{h\in  \Gc(x,-)\mid  \bt_x(h)=y\}$ is an immersed  \Q-submanifold of $ \Gc(x,-)$. 
  Moreover  the  isotropy group   $ \Gc(x)$ is a \Q-group  
  and 
  $T_{\1_x}(\Gc(x))=\Ker (T_{ \1_x} \bs)\cap \Ker (T_{ \1_x} \bt) $. 
\end{enumerate}
\end{proposition}
  
   \begin{proof} 
\eqref{generalproperties_item1}   	
 We have already seen in Proposition~\ref{preliminaries} that   
 $\Gc(x,-)=\bs^{-1}(x)$ and $\Gc(-,x)=\bt^{-1}(x)$ are  immersed \Q-submanifolds of 
 $\Gc$. 

It is easily checked that the mapping 
$$\eta_g\colon \Gc\to \Gc\times\Gc,\quad  h\mapsto (g,h)$$ 
is \Q-smooth 
since $\Gc\times\Gc$ is a \Q-manifold with respect to the product of \Q-smooth structures of $\Gc$ and itself. 
(See also \cite[Ch. 1, \S 1, no. 6, Prop., page 238]{Br73}.) 
Since the inclusion map $\Gc(-,x)\hookrightarrow \Gc$ is \Q-smooth, it then follows that 
$\eta_g\vert_{ \Gc(-,x)}\colon\Gc(-,x)\to \Gc\times\Gc$ is \Q-smooth. 
But the image of this last map is contained in $\Gc^{(2)}$, which is an immersed \Q-submanifold of $\Gc\times\Gc$ by Proposition~\ref{preliminaries}\eqref{preliminaries_item1}. 
It then follows by the universality property of immersed \Q-submanifolds 
(cf. \cite[Ch. 1, \S 2, no. 2, (ii), page 240]{Br73})
that 
$\eta_g\vert_{ \Gc(-,x)}\colon\Gc(-,x)\to \Gc^{(2)}$ is \Q-smooth. 
On the other hand, the multiplication map $\bm\colon  \Gc^{(2)}\to \Gc$ is \Q-smooth, 
hence the composition $\bm\circ\eta_g\vert_{ \Gc(-,x)}\colon\Gc(-,x)\to \Gc$ is \Q-smooth. 
We now note that $L_g=\bm\circ\eta_g\vert_{ \Gc(-,x)}$, hence 
$L_g\colon\Gc(-,x)\to \Gc$ is \Q-smooth. 
Moreover, we have $L_g(\Gc(-,x))=\Gc(-,y)$ and $\Gc(-,y)$ is an immersed \Q-submanifold of $\Gc$ by Proposition~\ref{preliminaries}. 
Then, by  the universality property of immersed \Q-submanifolds again, 
we obtain that the map $L_g\colon\Gc(-,x)\to\Gc(-,y)$ is \Q-smooth. 
By the same argument, the map $L_{g^{-1}}\colon\Gc(-,y)\to\Gc(-,x)$ is \Q-smooth. 
Since $L_{g^{-1}}=(L_g)^{-1}$, we thus see that $L_g\colon\Gc(-,x)\to\Gc(-,y)$ is actually a \Q-diffeomorphism. 

Using a similar reasoning, or the above result and the straightforward equality 
$$\bi\circ L_{g^{-1}}\circ\bi\vert_{\Gc(y,-)}=R_g\colon \Gc(y,-)\to \Gc(x,-)$$
along with the fact that $\bi\colon\Gc\to\Gc$ is a \Q-diffeomorphism, 
we see that the right translation 
$R_g\colon  \Gc(y,-)\to \Gc(x,-)$ is a \Q-diffeomorphism, too.

  \eqref{generalproperties_item2} We adapt the method of proof of \cite[Th. 5.4]{MoMr03}. 
To this end we will first  show that, defining 
  $$\Delta_g:=\Ker T_g \bs\cap \Ker T_g \bt 
  \quad \text{for all }g\in  \Gc(x,-),$$
one obtains a distribution  $\Delta\subseteq T \Gc(x,-)$ which is integrable (in the sense defined in  \cite[Ch. 2, \S 5, no. 0]{Br73}). 
 Indeed  given any $g\in \Gc(x,-)$, we have the left translation 
 \begin{equation}\label{Lgcolon}
 L_g\colon  \Gc(-,x)\to  \Gc(-, \bt(g)), 
\end{equation}
 and moreover $ \bs\circ L_g = \bs\vert_{ \Gc(-,x)}$, hence 
  $$TL_g(\Delta_{ \1_x})=\Delta_g.$$  
 Therefore the map  
 $$\Phi\colon \Gc(x,-)\times \Delta_{ \1_x}\to T \Gc(x,-),\quad \Phi(g,v)=TL_g(v)$$  
 is an injective  morphism over $\id_{ \Gc(x,-)}$ from  the trivial distribution  $ \Gc(x,-)\times \Delta_{ \1}$ on  $ \Gc(x,-)$  whose range is the distribution $\Delta$. 
 
  For any $v\in  \Delta_{\1_x}$ denote by $X_v$ the vector field on $   \Gc(x,-)$ defined by $X_v(g)=TL_g(v)$. 
 Then the set $$\Gamma(\Delta)=\{X_v\mid v\in  \Delta_{ \1_x}\}$$ generates the distribution $\Delta$. 
 Now any integral curve $\gamma: {]-\varepsilon,\varepsilon[} \to  \Gc(x,-)$  of $X\in \Gamma(\Delta)$ with $\gamma(0)=g$ is  also contained in $ \Gc(-,  \bt(g))$. It follows that the Lie bracket $[X,Y](g)$  of vector fields $X$ and $Y$ in $\Gamma(\Delta)$ is tangent to $ \Gc(-, \bt(g))$ and so $\Gamma(\Delta)$ is stable under Lie bracket 
 and so $\Delta$ is involutive. 
 Thus, by  \cite[Ch. 2, \S 5, no. 0, Th., page 251]{Br73}, the distribution $\Delta$ is integrable, i.e., 
 there exists a partition of $ \Gc(x,-)$ into immersed \Q-manifolds  and each one  is modeled on the vector space $\Delta_{ \1_x}$. 
 Moreover, for $g\in  \Gc(x,-)$ , since $\Delta_g=\ker T_g \bt_x$, the maximal leaf through~ $g$ is a connected component of $ \bt^{-1}( \bt(g))$ and so is an immersed   \Q-submanifold  of  $  \Gc(x,-)$. 
 In particular  the isotropy group $ \Gc(x)$ 
 is a union of such leaves and so $ \Gc(x)$ is an immersed \Q-submanifold of $ \Gc(x,-)$. 
 Since $ \Gc(x)=\Gc(x,x)$,  
 we have  $\bm(\Gc(x)\times\Gc(x))\subseteq\Gc(x)$.  
 Since $\bm$ is \Q-smooth and  $ \Gc(x)$  is an immersed \Q-submanifold of $ \Gc$, this easily implies that the restriction of $ \bm$ to  $ \Gc(x)\times  \Gc(x)\subseteq  \Gc^{(2)}$  is smooth as a mapping $\Gc(x)\times  \Gc(x)\to\Gc(x)$.  
 Analogous  arguments work for the inversion $ \bi$ restricted to $ \Gc(x)$. 
 It follows that $ \Gc(x )$ is a \Q-group.
 
  Finally, since    $L_g\colon  \Gc(-,x)\to  \Gc(-,  \bt(g))$,  is a \Q-diffeomorphism,  
  $ \bs\circ L_g = \bs\vert_{ \Gc(-,x)}$,  and  $TL_g(\Delta_{ \1_x})=\Delta_g$,   it follows that the mapping $L_g\vert_{\Gc(x)}\colon\Gc(x)\to   
  \Gc(x,  \bt(g))$ is a \Q-diffeomorphism.
   We conclude that each fibre of $ \bt_x$  is  \Q-diffeomorphic to $ \Gc(x)$.  This proves the first part of the assertion and the Lie algebra of $ \Gc(x)$ is isomorphic to $\ker T_{ \1_x} \bs\cap \ker T_{ \1_x} \bt $ (cf. the construction of the Lie algebroid $\mathcal{AG}$).
\end{proof}

\begin{example}\label{Lie groupoid} 
\normalfont
Since any smooth manifold is a \Q-manifold, any Lie groupoid $ \Gc\tto M$ is also a \Q-groupoid.
\end{example}

\begin{example}\label{Q-group}
\normalfont
 Let $G$ be a \Q-group with its unit element $\1_G\in G$. 
 If $M$ is a singleton $\{\1_G\}$ then $G\tto \{\1\}$ can be considered a \Q-groupoid  whose multiplication, inversion, and unit maps are given by their corresponding  the multiplication, inversion, and the unit element in $G$.
\end{example}

\begin{example}\label{action} 
\normalfont
  Let $G$ be a \Q-group with its unit element $\1_G\in G$ and  $A\colon M\times G\to M$, $(x,g)\mapsto xg$ be a \Q-smooth right action on a  manifold $M$. As in the classical case of Lie groups,   to such action  we can associate a \Q-groupoid $ \Gc\tto M$ defined in the following way:
\begin{itemize}
\item $ \Gc:=M\times G$ 
\item $ \bs(x,g):=x$ and $ \bt(x,g):= A(x,g)=xg$  
\item if $y=xg$ then $ \bm((y,h),(x,g)):=(x,hg)$. 
\item $ \bi(x,g):=(xg,g^{-1})$; 
\item $ \1_x=(x, \1_G)$.
 \end{itemize}
For any $x_0,y_0\in M$ one has 
$\Gc(x_0,y_0)=\{x_0\}\times\{g\in G\mid x_0g=y_0\}$
and in particular the isotropy group at $x_0$ is 
$\Gc(x_0)=\{x_0\}\times G(x_0)$,
where $G(x_0):=\{g\in G\mid x_0g=x_0\}$. 

These assertions can be easily proved by an  application of 
Proposition~\ref {preliminaries}, taking into account the following facts:
\begin{enumerate}[{\rm(a)}]
\item $ \Gc$ is a \Q-manifold as a product of \Q-manifolds;
\item $ \bs$ is a \Q-submersion since it is the projection of $ \Gc$ on $M$;
\item $ \bi$ is a \Q-diffeomorphism  as  it is the map  $(x,g)\mapsto (A(x,g), g^{-1})$ whose components are \Q-smooth maps;
\item $ \bm$ is smooth since each component of the map $((y,h),(x,g))\mapsto (x,hg)$ is smooth and as  $ \Gc^{(2)}$ is an immersed  \Q-submanifold of $ \Gc\times  \Gc$.
\end{enumerate}
\end{example}

\subsection{The Lie algebroid associated to a \Q-groupoid}
\label{LieAlgebroidQQtLieGroupoid}${}$\\
In the same way as in the  classical context of Lie groupoids (cf. e.g., \cite[\S 3.5]{Ma05} or \cite[\S 6.1]{MoMr03}), to every \Q-groupoid there corresponds a  Lie algebroid on  $M$.  
We now sketch its construction.

By \cite[Ch. 1, \S 1, no. 4]{Br73},  each \Q-manifold $S$ has a tangent bundle $\tau_S\colon TS\to S$, where  $TS$ is  a \Q-manifold  and $\tau_S$ is a \Q-submersion.  
The notion of subbundle  of $TS$ and of pullback bundle are discussed in Appendix~\ref{AppB}. 

Every \Q-smooth map $f\colon S\to S'$ between two  \Q-manifolds has a tangent map $Tf\colon TS\to TS'$ which is \Q-smooth and fibrewise linear.  
A  vector field is a \Q-smooth map  $X\colon S\to TS$ such that $\tau_S\circ X=\id_S$, 
and the set  $\Gamma(TS)$ of all vector fields on $S$  has the natural structure of a Lie algebra with respect to the Lie bracket of vector fields. 

Now let $\Gc$ be a \Q-groupoid and 
denote by  $T^\bs \Gc:=\Ker (T\bs)\to \Gc$   the vertical  subbundle of $T \Gc\to \Gc$ associated to the \Q-submersion $ \bs\colon  \Gc\to M$. 
The elements of $T^\bs \Gc$ are called \emph{vertical} tangent vectors on $\Gc$. 
This bundle is involutive (cf. \cite[Ch. 2, \S5, no. 0]{Br73}). 
A vector field $X\colon \Gc\to T\Gc$ is called  \emph{right invariant} 
if it is vertical and $T_h R_g(X(h))= X(hg)$ whenever $(h,g)\in\Gc^{(2)}$.  
Let $\Gamma_{\rm inv}^r( \Gc)$  be the vector space  of  right invariant vector fields. 
Then $\Gamma_{\rm inv}^r(\Gc)$ 
is a subalgebra of the Lie algebra $\Xc(\Gc)$. 
(See \cite[Lemma 3.5.5]{Ma05} or \cite[Prop. 6.1(i)]{MoMr03} for the classical case.)
Proposition \ref{generalproperties}\eqref{generalproperties_item1}  implies that the tangent map $T \bt$ induces a morphism of Lie algebras from $\Gamma_{\rm inv}^r( \Gc)$ to   the Lie algebra $\Gamma(M)$ of vector fields on $M$.

We define $\mathcal{AG}\to M $ as the pull-back  of  the vector bundle $T^\bs \Gc\to  \Gc$ via the unit map $ \1\colon M\to \Gc$. 
That pull-back can be regarded as the restriction of  $T^\bs \Gc$ to the  immersed \Q-submanifold $\1_M\subseteq \Gc$ 
(cf. Proposition~\ref{preliminaries}\eqref{preliminaries_item2}). 

On the one hand, as e.g., in \cite[\S 2.1]{Ma05}, we have a canonical \Q-morphism $\1^!\colon\mathcal{AG}\to T^\bs \Gc $ from the bundle  $\mathcal{AG}\to M $ to the bundle $T^ \bs \Gc\to  \Gc$ and the commutative diagram 
\begin{equation*}
\xymatrix{
\mathcal{AG} \ar[r]^{\1^!} \ar[d] & T^\bs\Gc \ar[d] \\
M \ar[r]^{\1} & \Gc
}
\end{equation*}
Here $\1^!\colon\mathcal{AG}\to T^\bs \Gc $ is a fibrewise isomorphism  
and gives rise to an isomorphism of vector spaces 
$\Gamma(\mathcal{AG})\to \Gamma_{\rm inv}^r( \Gc)$
which 
 can be used to transport the Lie bracket from $\Gamma_{\rm inv}^r( \Gc)$ to 
 the vector space $\Gamma(\mathcal{AG})$ of smooth sections of  $\mathcal{AG}\to M $. 
The Lie bracket on  $\Gamma(\mathcal{AG})$ 
will also be denoted by $[\cdot,\cdot]$, 
and we have an isomorphism of Lie algebras 
$\Gamma(\mathcal{AG})\to\Gamma_{\rm inv}^r( \Gc)$. 
Additionally, $\mathcal{AG}$ is provided with the anchor $\rho=T \bt\vert_{\mathcal{AG}}\colon\mathcal{AG}\to TM$, 
just as in the classical case discussed in  \cite[Lemma 3.5.7]{Ma05} or \cite[Prop. 6.1(iii)]{MoMr03}. 
It follows that  $(\mathcal{AG}, M,\rho,[\cdot,\cdot])$ is a  Lie algebroid.

\subsection*{Towards a Lie functor for \Q-groupoids}
We now discuss the infinitesimal version of a morphism of \Q-groupoids, 
which should be expected to be a morphism of Lie algebroids. 
However, we recall from \cite[Ch. 4]{Ma05} that the corresponding discussion is nontrivial even in the classical case of Lie groupoids. 
Consider a \Q-groupoid morphism 
\begin{equation}
\label{Liefunctor_eq1}
\xymatrix{
	\Gc \ar@<-2pt>[d]_{\bt} \ar@<2pt>[d]^{\bs} \ar[r]^{\Phi} 
	 & \Hc \ar@<-2pt>[d]_{\bt} \ar@<2pt>[d]^{\bs} \\
M \ar[r]^{\phi} & N
}
\end{equation}
Since $\bs\circ\Phi=\phi\circ\bs$, it follows that for every $x\in M$ we have $\Phi(\Gc(x,-))\subseteq\Hc(\phi(x),-)$ and $\Phi_{\1_x}=\1_{\phi(x)}$. 
This further implies $(T\Phi)(T^\bs\Hc)\subseteq\Hc^\bs$. 
Moreover, 
\begin{equation}
\label{Liefunctor_eq2}
(\forall x\in M)\quad (T\Phi)(T_{\1_x}(\Gc(x,-)))\subseteq 
T_{\1_{\phi(x)}} (\Hc(\phi(x),-))
\end{equation}
This leads to the following commutative diagram: 
\begin{equation*}
\xymatrix{\A\Gc \ar[dr]^{\1_M^{!}}  \ar@{.>}[rrr]^{\A\Phi} \ar[dddd] &	 &  & \A\Hc \ar[dl]_{\1_N^{!}} \ar[dddd]\\
 & T^\bs\Gc \ar@{^{(}->}[d] \ar@{.>}[r]^{T\Phi} & T^\bs\Hc \ar@{^{(}->}[d] & \\
& T\Gc \ar[d]_{\tau_\Gc} \ar[r]^{T\Phi}  & T\Hc \ar[d]^{\tau_\Hc}  & \\
& \Gc \ar[r]^{\Phi}  & \Hc & \\
M \ar[ur]^{\1_M} \ar[rrr]^{\phi} & & & N \ar[ul]_{\1_N} 
}
\end{equation*}
Here the unit mapping $\1_M\colon M\to\Gc$ is an immersed submanifold, 
hence $\1_M^{!}\colon\A\Gc\to T^\bs\Gc$ is in turn an immersed submanifold, by Proposition~\ref{distrib4}. 
This further implies that the mapping $\A\Phi\colon\A\Gc\to\A\Hc$, 
well-defined by \eqref{Liefunctor_eq2}, 
is also smooth. 
Thus $\A\Phi\colon\A\Gc\to\A\Hc$ is a (smooth) morphism of locally trivial, smooth, vector bundles 
over the base map $\phi\colon M\to N$. 

We now prove in Proposition~\ref{Liefunctor_prop} that $\A\Phi\colon\A\Gc\to\A\Hc$ is a morphism of Lie algebroids 
in the case when both $\Gc\tto M$ and $\Hc\tto N$ are locally trivial  \Q-groupoids 
(cf. Section~\ref{Sect7}), 
which suffices for the purposes of the present paper. 
The development of the general Lie functor for \Q-groupoids is postponed for a later paper. 

\begin{proposition}
\label{Liefunctor_prop}
Let $\Gc\tto M$ and $\Hc\tto N$ be  \Q-groupoids 
with a \Q-groupoid morphism $\Phi\colon \Gc\to\Hc$ as in  \eqref{Liefunctor_eq1}. 
If 
	both $\Gc\tto M$ and $\Hc\tto N$ are locally trivial  \Q-groupoids, 
then $\A\Phi\colon\A\Gc\to\A\Hc$ is a morphism of Lie algebroids over the base map $\phi\colon M\to N$. 
\end{proposition}

\begin{proof}
Using the local character of morphisms of Lie algebroids (Remark~\ref{trivLiealg_morph_loc} and the fact that both  $\Gc\tto M$ and $\Hc\tto N$ are locally trivial  \Q-groupoids it follows by Remark~\ref{decomposing} that we may assume that both these groupoids are actually trivial, that is $\Gc=M\times G\times M$ and $\Hc=N\times H\times N$ for certain \Q-groups $G$ and $H$ with their Lie algebras $\gg=T_\1 G$ and $\hg=T_\1 H$. 
With the notation from \eqref{Liefunctor_eq1}, since $\Phi$ is a morphism of \Q-groupoids, it easily follows that there exists a morphism of \Q-groups $\Theta\colon G\to H$ such that 
\begin{equation*}
(\forall x,y\in M)(\forall g\in G)\quad \Phi(y,g,x)=(\phi(y),\Theta(g),\phi(x)).
\end{equation*}
We further obtain that $\A\Gc=(TM)\oplus(M\times \gg)$ and $\A\Hc=TN\oplus(N\times \hg)$ are trivial Lie algebroids, and $\A\Phi\colon\A\Gc\to\A\Hc$ is given by 
\begin{equation}
\label{Liefunctor_prop_proof_eq1}
 (\A\Phi)(X,(x,v))=((T_x\phi)(X),(\phi(x),(T_\1\Theta)(v)))\in T_{\phi(x)}N\oplus (\{\phi(x)\}\times\hg)
\end{equation}
if $x\in M$, $X\in T_xM$, and $v\in\gg$.

We then have, cf. \eqref{trivLiealg_morph_eq1}, 
\begin{equation*}
\phi^{!!}(\A\Hc)=\{(X,(Y,(y,w)))\in TM\times\A\Hc\mid (T\phi)(X)=Y\}
\end{equation*}
hence the mapping 
\begin{equation}
	\label{Liefunctor_prop_proof_eq2}
\phi^{!!}(\A\Hc)\to TM\oplus (M\times\hg),\quad (X,(Y,(y,w)))\mapsto (X,(\tau_M(X),w))
\end{equation} 
is  bijective (where $\tau_M\colon TM\to M$ is the tangent bundle of $M$) with its inverse 
$$TM\oplus (M\times\hg)\to \phi^{!!}(\A\Hc),\quad (X,(x,w))\mapsto (X,((T\phi)(X),(\phi(x),w))).$$
In this way we may identify the Lie algebroid $\phi^{!!}(\A\Hc)$ with the trivial Lie algebroid $TM\oplus (M\times\hg)$ over $M$. 

Moreover, cf. \eqref{trivLiealg_morph_eq2}, 
the mapping $(\A\Phi)^{!!}\colon \A\Gc\to \phi^{!!}(\A\Hc)$ is given by 
\begin{equation*}
(\A\Phi)^{!!}(X,(x,v))=(X,(\A\Phi)(X,(x,v)))
\mathop{=}\limits^{\eqref{Liefunctor_prop_proof_eq1}}
(X,((T_x\phi)(X),(\phi(x),(T_\1\Theta)(v))))
\end{equation*}
hence, using the identification~\eqref{Liefunctor_prop_proof_eq2}, 
we may write 
\begin{align}
&(\A\Phi)^{!!}\colon \A\Gc=TM\oplus(M\times\gg)\to TM\oplus(M\times\hg),\nonumber\\
\label{Liefunctor_prop_proof_eq3}
&(\A\Phi)^{!!}(X,(x,v))=(X,(x,(T_\1\Theta)(v))).
\end{align}
Since $T_\1\Theta\colon\gg\to\hg$ is a morphism of Lie algebras, 
it directly follows by the definition of the Lie bracket of trivial Lie algebroids, cf. \eqref{trivLiealg_def_eq1}, that the above mapping $(\A\Phi)^{!!}$ is a morphism of trivial Lie algebroids over the base map $\id_M$, and this completes the proof. 
\end{proof}

\section{Locally trivial  \Q-groupoids}
\label{Sect7}

In this section we briefly discuss locally trivial \Q-groupoids, 
which constitute the main class of \Q-groupoids that are needed for the integration of general transitive Lie algebroids in Corollary~\ref{integration}. 

A \Q-groupoid $ \Gc\tto M$ is called  \emph{transitive} if for any $x,y\in M$, there exists $g\in  \Gc$ such that $\bs(g)=x$ and $\bt(g)=y$. 
A \Q-groupoid $ \Gc\tto M$ is called \emph{locally trivial} if the map $( \bt, \bs): \Gc\to M\times M$ is a surjective \Q-submersion. 
Note that if $\Gc$ is locally trivial then it is transitive since for any pair $(x,y)\in M\times M$, 
we have  $( \bt, \bs)^{-1}(x,y)\not=\emptyset$. The converse is not true in general.

  As for Lie groupoids, cf. \cite[Prop. 1.3.3]{Ma05}, 
  we have the following characterization of  locally trivial \Q-groupoids: 
 
 \begin{proposition}
 	\label{txSubmersion} 
 A \Q-groupoid $ \Gc\tto M$  is locally trivial if and only if for some $x\in M$ the map $ \bt_x\colon \Gc(x,-)\to M$ is a surjective \Q-submersion. 
 Moreover, if this last property is true for one point in $M$, it is then true for all points in $M$. 
 \end{proposition}

\begin{proof}  
	Since $\Gc\tto M$ is a groupoid, it is clear that the surjectivity of the mapping~$(\bt,\bs)$ is  equivalent to surjectivity of the groupoid orbit map~$\bt_x$. 
	In order to prove the equivalence of \Q-submersion properties of these maps, we note that they are related by the commutative diagram
	\begin{equation*}
		\xymatrix{
		\Gc(x,-)\times\Gc(x,-) \ar[r]^{\qquad\delta_x} \ar[dr]_{\bt_x\times\bt_x}& \Gc \ar[d]^{(\bt,\bs)} \\
		& M\times M
		}
	\end{equation*}
	where the division map $\delta_x\colon \Gc(x,-)\times\Gc(x,-) \to\Gc$, $\delta_x(h,k):=kh^{-1}$, is a surjective \Q-smooth map by the hypothesis that $\Gc\tto M$ is a transitive \Q-groupoid.  
	(See also Remark~\ref{OtherProperties}.)
	If $\bt_x$ is a \Q-submersion, then so is $\bt_x\times\bt_x$, hence the above commutative diagram shows that $(\bt,\bs)$ is a \Q-submersion as well, by Remark~\ref{rem:subm}\eqref{rem:subm_2}. 
	
	Conversely, assume the groupoid $ \Gc\tto M$  is locally trivial, i.e.,  $( \bs, \bt): \Gc\to M\times M$ is a surjective \Q-submersion. 
	Since $\{x\}\times M$ is an embedded submanifold of $M\times M$  and $( \bs, \bt)^{-1}(\{x\}\times M)= \Gc(x,-)$,   the \Q-submersion
$( \bs, \bt)\colon \Gc\to M\times M$ gives by restriction 
a \Q-submersion $\Gc(x,-)\to \{x\}\times M$. 
This last \Q-submersion can be identified wuth $\bt_x$, 
which is thus a \Q-submersion. 
In order to check that $\bt_y$ is a \Q-submersion for arbitrary $y\in M$, 
we use the fact that 
$ \Gc$ is transitive, 
hence there exists $h\in  \Gc$ such that  $\bs(h)=y$ and $\bt(h)=x$. 
By Proposition~\ref{generalproperties}\eqref{generalproperties_item1}, 
the mapping $R_h\colon\Gc(x,-)\to\Gc(y,-)$,  $k\mapsto kh$, is a \Q-diffeomorphism 
such that  
$\bt_y\circ R_h=\bt_x$.  
Since  $\bt_x$ is a \Q-submersion, 
it then follows 
that $ \bt_y:  \Gc(y,-)\to M$ is also a surjective \Q-submersion, 
which ends 
the proof.
\end{proof}

  As in  the classical case of the category of Lie groupoids, we have (cf. the comments preceding  \cite[Prop. 1.3.5]{Ma05}):

  \begin{theorem}\label{G-principalbundle} 
If  $ \Gc\tto M$ is a locally trivial  \Q-groupoid, then
   $ \bt_x\colon  \Gc(x,-)\to M$
    is  a principal \Q-bundle with its structure group $\Gc(x)$ for every $x\in M$. 
 \end{theorem}

  \begin{proof}  We consider the map 
$$\Psi\colon \Gc(x,-)/\Gc(x)\to M,\quad g\cdot\Gc(x)\mapsto  \bt(g).$$
It is easy to prove that  $\Psi$ is bijective since  it is injective by construction and  surjective  by Proposition \ref{txSubmersion}. 
 We must show that $ \bt_x: \Gc(x,-)\to M$ satisfies the conditions of Definition~\ref{Q-bundle_def}.

Consider the right group action 
$\Rc\colon \Gc(x,-)\times\Gc(x)\to\Gc(x,-)$, 
  $(h,g)\mapsto \mathcal{R}(h,g)=hg$. 
 As   the groupoid multiplication and the inversion are \Q-smooth, it follows that  this  action  is \Q-smooth and free. 
The transitivity implies that  for any $y\in M$, there exists $h\in  \Gc(x)$,  such that $ \bt_x(h)=y$.  
Consider  the map $r_h\colon\Gc(x)\to\Gc(x,-)$, 
$r_h(g):= hg=L_h(g)$,   
which is the restriction of $L_h$ to $ \Gc(x)$ and so
defines a \Q-diffeomorphism $\Gc(x)\to\Gc(x,y)$. 
 As  $\Gc(x,y)\subseteq\Gc(x,-)$ this implies that
 $r_h\colon\Gc(x)\to\Gc(x,-)$ is a \Q-immersion whose range is $ \Gc(x,y)= \bt_x^{-1}(y)$.  
 
 For any $y\in M$,  since $ \bt_x\colon\Gc(x,-)\to M$ is a surjective submersion, there exist an open subset $V\subseteq M$ with $y\in V$ and  a smooth cross-section $\sigma\colon V\to   \bt_x^{-1}(V)$.  
 We define the map $\Phi\colon V\times  \Gc(x)\to  \bt_x^{-1}(V)$,  $\Phi(h,g)=r_{\sigma(z)}(g)=\sigma(z)g$. 
 As   the multiplication in $ \Gc$ is \Q-smooth and each  $r_{\sigma(z)}$ is a diffeomorphism on the fibre over $z$, this map is smooth and bijective.  
 Thus $\Phi$ is a \Q-diffeomorphism. 
 Note that for any $h\in  \Gc(x,z)$ we have  $(\sigma(z))^{-1}h\in \Gc(x)$. 
 Therefore the map $\Phi^{-1}\colon \bt_x^{-1}(V) \to V\times  \Gc(x)$ can be written
 $$\Phi^{-1}(h)=( \bt(h), \phi(h)) $$
 with $\phi(h)=(\sigma( \bt(h))^{-1}h$.  
 According to Remark~\ref{local} 
 or 
 Definition~\ref{Q-bundle_def},  
 the proof will be complete if we show that
 $$\phi(\mathcal{R}(h,g))=\mathcal{R}(\phi(h),g)$$
 But we have
 $$\phi(\mathcal{R}(h,g))=\phi(hg)=(\sigma( \bt(g^{-1}h))^{-1}hg=(\sigma( \bt(h))^{-1}hg=\phi(h)g=\mathcal{R}(\phi(h),g)$$
 which ends the proof.
 \end{proof}
 
 \begin{remark}
 \label{Qcross}
 \normalfont
 Let $S$ be a \Q-manifold with a \Q-atlas $\pi\colon M\to S$. 
 If $N$ is a smooth manifold and $f\colon S\to N$ is a \Q-submersion, 
 then $f\circ\pi\colon M\to N$ is a smooth submersion, hence for every $s_0\in S_0$, $m_0\in \pi^{-1}(s_0)\subseteq M$ and $n_0\in N$ there exist an open subset $N_0\subseteq N$ and a smooth mapping $\sigma_0\colon N_0\to M$ satisfying $n_0\in N_0$, $\sigma_0(n_0)=m_0$, and $(f\circ\pi)\circ\sigma_0=\id_{N_0}$. 
 Denoting $\sigma:=\pi\circ\sigma_0\colon N_0\to S$, we then obtain $f\circ\sigma=\id_{N_0}$ and $\sigma(n_0)=\pi(m_0)=s_0$. 
 If these conditions are satisfied, we call $\sigma$ a \emph{\Q-smooth local cross-section} of $f$. 
 \end{remark}

\begin{remark}
\label{decomposing}
\normalfont 
Let $\Gc\tto M$ be a locally trivial \Q-groupoid and fix any $x_0\in M$. 
By Proposition~\ref{txSubmersion}, the mapping $ \bt_{x_0}\colon \Gc(x_0,-)\to M$ is a surjective \Q-submersion. 
Therefore, by Remark~\ref{Qcross}, one can find an open cover $\{U_i\}_{i\in I}$ of $M$ 
and for every $i\in I$ a \Q-smooth map $\sigma_i\colon U_i\to \Gc$ with $\bt_{x_0}\circ\sigma_i=\id_{U_i}$ and $\sigma_i(x_0)=\1_{x_0}\in \Gc(x_0,-)$. 
Now let us consider the \Q-group $G_0:=\Gc(x_0)$. 
For arbitrary $i\in I$, it is easily checked that the restricted groupoid $\Gc(U_i,U_i):=\bt^{-1}(U_i)\cap\bs^{-1}(U_i)\tto U_i$ is a \Q-groupoid. 
Moreover, the mapping 
$$\Sigma_i\colon U_i\times G_0\times U_i\to \Gc(U_i,U_i),\quad 
(y,g,x)\mapsto \sigma_i(y)g\sigma_i(x)^{-1}$$
is a base-preserving isomorphism of \Q-groupoids just as in \cite[\S 1.3, Eq. (2)]{Ma05}. 
\end{remark}

 \section{The gauge groupoid of a principal \Q-bundle}
 \label{Sect8}
 
 In this section we obtain the main result of the present paper 
 on integrating  general transitive Lie algebroids on second countable, smooth manifolds, 
 to locally trivial \Q-groupoids (Corollary~\ref{integration}).  
 The key to this result is the gauge groupoid of a principal \Q-bundle, 
 which allows us to  show that the classical relation between smooth principal bundles and transitive Lie groupoids carries over to the setting of \Q-manifolds 
 (Theorems \ref{GaugeQGroupoid} and \ref{isoloctricQgroupoid}). 
 More specifically, we show that the gauge groupoid of a principal \Q-bundle admits a \Q-atlas whose domain of definition is a smooth principal bundle. 
 This additional piece of information plays a crucial role in the proof of Theorem~\ref{Ehr_th}, which gives the description of the Lie algebroid associated to a gauge groupoid.

 Let $G$ be a \Q-group  and  $\beta:P\to M$ be a principal \Q-bundle with structure group~$G$. 
 We  will  define the gauge \Q-groupoid of $P$, following the steps of  the classical construction of the gauge groupoid of  a principal bundle. 
  
  From the right action 
  $$P\times G\to P,\qquad (p,g)\mapsto 
  pg$$  
 we get a right diagonal  action of $G$ on $P\times P$
 $$(P\times P)\times G,\quad  
 (p,q,g)\mapsto 
 (pg,qg)
 $$ 
with its
corresponding quotient map 
$$P\times P\to (P\times P)/G,\quad (p,q)\mapsto[p,q]:=(p,q)G=\{(pg,qg)\mid g\in G\}.$$
The quotient set
$$ \Gc:=(P\times P)/G$$
is endowed with the \emph{gauge groupoid} structure $ \Gc\tto M\equiv P/G$ for which the above quotient map is a groupoid morphism defined on the pair groupoid $P\times P\tto P$. 
More specifically, the structure maps of  $ \Gc\tto M$ are given 
 in the following way:
\begin{itemize}
	\item 
	 the source/target maps $ \bs,\bt\colon\Gc\to M$, 
	 $\bs([p,q]):=\beta(q)$, $\bt([p,q]):=\beta(p)$; 
\item the multiplication map $\bm\colon\Gc^{(2)}\to\Gc$, $\bm([p, q], [q,r])):=[p, q]\cdot [q,r]:= [p, r]$;  
\item  
the unit map $\1\colon M\to \Gc$, $\1_{\beta(p)}:=[p,p]$
\item the inverse $ \bi$ is defined by $[p, q]^{-1}$ of the class $[p,q]$   is the class $[q, p]$. 
 \end{itemize}
By the way, we have $ \Gc(x,-)=\{ [p,q]\in  \Gc: \beta(p)=x\}$ for every $x\in M$. 

In Theorem~\ref{GaugeQGroupoid} below we use the above notation as well as notation introduced in Remark~\ref{principalbundleQchart}. 
{In order to motivate the definition of the equivalence relation~\eqref{simTildeMathG} in the proof of Theorem~\ref{GaugeQGroupoid}, 
	we make the following remark on the special case of gauge groupoids associated to smooth principal bundles. 
	\begin{remark}
		\label{GaugeQGroupoid_smooth}
		\normalfont 
Let $\Gc\tto M$ be a locally trivial Lie groupoid. 
Select any point $x_0\in M$ and consider the Lie group $G:=\Gc(x_0)$. 
Assume that $M=\bigcup\limits_{i\in I}U_i$ is an open covering 
and for every $i\in I$ we have a smooth cross-section $\sigma_i\colon U_i\to \Gc(-,x_0)$ 
of the surjective submersion $\bt_{x_0}\colon\Gc(-,x_0)\to M$, 
hence $\bt\circ\sigma_i=\id_{U_i}$ and $\bs(\sigma_i(x))=x_0$ for every $x\in U_i$. 
Then the mapping 
\begin{equation}
	\label{GaugeQGroupoid_eq1}
	\tau_{ji}\colon U_i\times G\times U_j\to(\bt\times\bs)^{-1}(U_i\times U_j),\quad  (x,g,y)\mapsto\sigma_i(x)g\sigma_j(y)^{-1}, 
\end{equation}
is a diffeomorphism. 
(Compare the section atlas described after \cite[Prop. 1.3.3]{Ma05}.)
We denote its inverse by $\theta_{ji}:=\tau_{ji}^{-1}\colon (\bt\times\bs)^{-1}(U_i\times U_j)\to U_i\times G\times U_j$, $\theta_{ji}(h)=:(\bt(h),\varphi_{ij}(h),\bs(h))$,  
hence 
\begin{equation*} 
	(\forall h\in (\bt\times\bs)^{-1}(U_i\times U_j))\quad h=\sigma_i(\bt(h))\varphi_{ij}(h)\sigma_j(\bs(h))^{-1}.
\end{equation*} 
Therefore, if $h\in (\bt\times\bs)^{-1}((U_{i_1}\times U_{j_1})\cap(U_{i_2}\times U_{j_2})))$, 
then 
\begin{equation}
\label{GaugeQGroupoid_eq2}
\sigma_{i_2}(\bt(h))\varphi_{i_2j_2}(h)\sigma_{j_2}(\bs(h))^{-1}
=h=
\sigma_{i_1}(\bt(h))\varphi_{i_1j_1}(h)\sigma_{j_1}(\bs(h))^{-1}.
\end{equation}
We now apply Remark~\ref{bundle_smooth} for the principal bundle $\beta=\bt_{x_0}\colon P=\Gc(-,x_0)\to  N=M$. 
Let $\psi_{ji}\colon U_i\cap U_j\to G=\Gc(x_0)$ be the functions satisfying~\eqref{bundle_smooth_eq2}, 
that is, $\sigma_i(x)=\sigma_j(x)\psi_{ji}(x)$ for all $x\in U_i\cap U_j$. 
Then, by~\eqref{GaugeQGroupoid_eq2}, we obtain 
\begin{eqnarray*}
\varphi_{i_2j_2}(h)
=\psi_{i_2 i_1}(\bt(h))\varphi_{i_1j_1}(h)\psi_{j_2 j_1}(\bs(h)).
\end{eqnarray*}
Consequently, if $i_1,i_2,j_1,j_2\in I$, 
then we can describe the points where the maps $\tau_{j_1 i_1}$ and $\tau_{j_2 i_2}$ 
defined in \eqref{GaugeQGroupoid_eq1}
take the same value. 
Specifically,  
if $(x_1,g_1,y_1)\in U_{i_1}\times G\times U_{j_1}$, 
and $(x_2,g_2,y_2)\in U_{i_2}\times G\times U_{j_2}$, then we have 
\begin{align*}
\tau_{j_1 i_1} & (x_1,g_1,y_1)=\tau_{j_2 i_2}(x_2,g_2,y_2) \\
&\iff x_1=x_2=:x,\ y_1=y_2=:y,\  \sigma_{i_1}(x)g_1\sigma_{j_1}(y)^{-1}=\sigma_{i_2}(x)g_2\sigma_{j_2}(y)^{-1} \\
&\iff  x_1=x_2=:x,\ y_1=y_2=:y,\ 
g_2=\psi_{i_2 i_1}(x)g_1\psi_{j_2 j_1}(y)^{-1}. 
\end{align*}
This last condition should be compared with \eqref{simTildeMathG} in the proof of Theorem~\ref{GaugeQGroupoid}. 
	\end{remark}
}
  
\begin{theorem}\label{GaugeQGroupoid}  
	If $G$ is a \Q-group and $\beta\colon P\to M$ is a principal \Q-bundle with structure group~$G$, then 
	the gauge groupoid $ \Gc=(P\times P)/G\tto M$ is a locally trivial  \Q-groupoid. 
	There exists a locally trivial Lie groupoid $ \widetilde{\Gc}\tto \widetilde{M}$ 
	and a \Q-groupoid morphism  $\pi_\Gc:\widetilde{ \Gc}\to  \Gc$ 
	which is a \Q-atlas of $ \Gc$.
\end{theorem}
 
\begin{proof} 
	The method of proof is to build a \Q-atlas $\pi_{ \Gc}:\widetilde{ \Gc}\to  \Gc$ using ideas from the proof of Proposition~\ref{bundle} and $\widetilde{ \Gc}$ can be provided with a Lie groupoid structure over $M$ such that $\pi_\Gc$ becomes a \Q-Lie groupoid morphism over $\id_M$.
   
   To this end we consider  an open covering $M=\bigcup\limits_{i\in I} U_i$, a family of transition functions $\psi_{ij}\colon U_i\cap U_j\to G$, and a Lie group $\widetilde{G}$  with a pseudo-discrete subgroup $\Lambda\subseteq\widetilde{G}$ such that $G=\widetilde{G}/\Lambda$, with the quotient homomorphism $\pi_G\colon\widetilde{G}\to \widetilde{G}/\Lambda$ 
   (cf. Remark~\ref{rem:Q-gr}), 
   and  their corresponding \Q-atlas $\pi_P:\widetilde{P}\to P$. 

We then consider the disjoint unions of smooth manifolds 
$$\widetilde{\Gc}:=\bigsqcup_{(i,j)\in I^2}(\{i\}\times U_i)\times \widetilde{G}\times (\{j\}\times U_j)
\text{ and }
\widetilde{M}:=\bigsqcup_{i\in I}(\{i\}\times U_i)
$$
with the maps 
\begin{align*}
&\widetilde{\bt},\widetilde{\bs}\colon\widetilde{\Gc}\to \widetilde{M},\quad \widetilde{\bt}(i,x,\widetilde{g},j,y):=(i,x),\ \widetilde{\bs}(i,x,\widetilde{g},j,y):=(j,y),\\
&\widetilde{\bi}\colon\widetilde{\Gc}\to\widetilde{\Gc},\quad 
(i,x,\widetilde{g},j,y)\mapsto (j,y,\widetilde{g}^{-1},i,x), \\
&\1\colon \widetilde{M}\to\widetilde{\Gc},\quad \1(i,x):=(i,x,\1_{\widetilde{G}},i,x). 
\end{align*}
Then $\widetilde{\Gc}$ is a Lie groupoid with its multiplication $\widetilde{\bm}\colon\widetilde{\Gc}^{(2)}\to\widetilde{\Gc}$ given by 
$$(i_1,x_1,\widetilde{g}_1,j_1,y_1)\cdot(i_2,x_2,\widetilde{g}_2,j_2,y_2)
=(i_1,x_1,\widetilde{g}_1\widetilde{g}_2,j_2,y_2)$$
if $(j_1,y_1)=(i_2,x_2)$. 

Recalling the smooth principal bundle with its structure group $\widetilde{G}$ 
$$\widetilde{\beta}\colon \widetilde{P}:=\bigsqcup_{i\in I}(\{i\}\times U_i\times \widetilde{G})\to \widetilde{M},\quad 
(i,x,\widetilde{g})\mapsto(i,x)$$
from Remark~\ref{principalbundleQchart}, it is easily checked that  
the mapping 
\begin{equation*}
\widetilde{\chi}\colon (\widetilde{P}\times\widetilde{P})/\widetilde{G}\to \widetilde{\Gc},\quad 
[(i,x,\widetilde{g}),(j,y,\widetilde{f})]\mapsto (i,x,\widetilde{g}\widetilde{f}^{-1},j,y)
\end{equation*}
is an isomorphism of Lie groupoids with its inverse mapping 
\begin{equation*}
\widetilde{\chi}^{-1}\colon	\widetilde{\Gc}\to (\widetilde{P}\times\widetilde{P})/\widetilde{G},\quad 
(i,x,\widetilde{g},j,y)\mapsto	[(i,x,\widetilde{g}),(j,y,\1_{\widetilde{G}})]
\end{equation*}
where $\1_{\widetilde{G}}\in\widetilde{G}$ is the unit element of the Lie group~$\widetilde{G}$. 

As noted in Remark~\ref{principalbundleQchart}, the mapping $\pi_P\colon\widetilde{P}\to P$ is a surjective homomorphism of principal \Q-bundles over the \'etale smooth mapping 
$$\pi_M\colon\widetilde{M}\to M,\quad \pi_M(i,x):=x$$ 
with respect to the \Q-group morphism $\pi_G\colon\widetilde{G}\to G$, 
hence we obtain the surjective  morphism of groupoids\footnote{In fact, the construction of the gauge groupoid of a principal bundle is a functor between suitable categories of principal bundles and  transitive groupoids, respectively. 
See \cite[Ex. 1.2.9]{Ma05} for that fact in the setting of Lie groupoids.}
$$\widetilde{\pi}_\Gc\colon (\widetilde{P}\times\widetilde{P})/\widetilde{G}\to
(P\times P)/G, 
\quad 
[\widetilde{p},\widetilde{q}]\mapsto [\pi_P(\widetilde{p}),\pi_P(\widetilde{q})].$$
We then obtain the commutative diagram 
\begin{equation}
\label{GaugeQGroupoid_proof_eq0}  
\xymatrix{(\widetilde{P}\times\widetilde{P})/\widetilde{G} \ar[rr]^{\widetilde{\chi}} \ar[dr]_{\widetilde{\pi}_\Gc}  & & \widetilde{\Gc} \ar[dl]^{\pi_\Gc}\\
	& (P\times P)/G & 
}
\end{equation}
where $\pi_\Gc:=\widetilde{\pi}_\Gc\circ\widetilde{\chi}^{-1}$. 

We now define an equivalence relation on $\widetilde{\Gc}$ by 
\begin{align}
(i_1,x_1,\widetilde{g}_1, & j_1,y_1) 
 \sim_{\Gc}(i_2,x_2,,\widetilde{g}_1,j_2,y_2) \nonumber \\
\label{simTildeMathG}
& \iff 
\begin{cases}
(x_1, y_1)=(x_2,y_2)=:(x,y)\in( U_{i_1}\times U_{j_1})\cap (U_{i_2}\times U_{j_2}) & \\
\pi_G(\widetilde{g}_2)=\psi_{i_2i_1}(x)\pi_G(\widetilde{g}_1)\psi_{j_2j_1}(y)^{-1} &
\end{cases}
\end{align}
and we claim that the fibres of $\pi_\Gc$ are exactly the equivalence classes with respect to~$\sim_\Gc$. 

To this end, we first compute the fibres of $\widetilde{\pi}_\Gc$. 
Specifically, since the group morphism 
$\pi_G\colon\widetilde{G}\to\widetilde{G}/\Lambda=G$ is surjective, we have 
\begin{align*}
	\widetilde{\pi}_\Gc([(i_1,x_1,\widetilde{g}_1) & ,(j_1,y_1,\widetilde{f}_1)]) 
	=\widetilde{\pi}_\Gc([(i_2,x_2,\widetilde{g}_2) ,(j_2,y_2,\widetilde{f}_2)])
	\\
	 \iff & [\pi_P(i_1,x_1,\widetilde{g}_1)  ,\pi_P(j_1,y_1,\widetilde{f}_1)]
	=[\pi_P(i_2,x_2,\widetilde{g}_2),\pi_P(j_2,y_2,\widetilde{f}_2)]
	\\
	\iff & (\exists \widetilde{g}\in\widetilde{G}) 
	\begin{cases}
		 \pi_P(i_2,x_2,\widetilde{g}_2)=\pi_P(i_1,x_1,\widetilde{g}_1)\pi_G(\widetilde{g}) & \\
		 \pi_P(j_2,y_2,\widetilde{f}_2)=\pi_P(j_1,y_1,\widetilde{f}_1)\pi_G(\widetilde{g})
		 &
	\end{cases} 
\\
\iff & (\exists \widetilde{g}\in\widetilde{G}) 
    \begin{cases}
	\pi_P(i_2,x_2,\widetilde{g}_2)=\pi_P(i_1,x_1,\widetilde{g}_1\widetilde{g}) & \\
	\pi_P(j_2,y_2,\widetilde{f}_2)=\pi_P(j_1,y_1,\widetilde{f}_1\widetilde{g}) & 
	\end{cases} 
\\
\mathop{\iff}\limits^{\eqref{ildePsim}} & (\exists \widetilde{g}\in\widetilde{G}) 
\begin{cases}
	x_2=x_1=:x\in U_{i_2}\cap U_{i_1}, & 
	\pi_G(\widetilde{g}_2)=\psi_{i_2 i_1}(x)\pi_G(\widetilde{g}_1\widetilde{g}) \\
	y_2=y_1=:y\in U_{j_2}\cap U_{j_1}, &
	\pi_G(\widetilde{f}_2)=\psi_{j_2 j_1}(y)\pi_G(\widetilde{f}_1\widetilde{g}) 
\end{cases} 
\\
\iff & 
\begin{cases}
	x_2=x_1=:x\in U_{i_2}\cap U_{i_1}, \ 	y_2=y_1=:y\in U_{j_2}\cap U_{j_1} & \\
	\pi_G(\widetilde{g}_2\widetilde{f}_2^{-1})=\psi_{i_2 i_1}(x)\pi_G(\widetilde{g}_1\widetilde{f}_1^{-1})\psi_{j_2 j_1}(y)^{-1} & 
\end{cases}
\end{align*}
Now, since 
$\widetilde{\chi}([(i_r,x_r,\widetilde{g}_r) ,(j_r,y_r,\widetilde{f}_r)])
=(i_r,x_r,\widetilde{g}_r\widetilde{f}_r^{-1} ,j_r,y_r)$ for $r=1,2$, 
we obtain by \eqref{simTildeMathG}
\begin{align*}
	& \widetilde{\pi}_\Gc([(i_1,x_1,\widetilde{g}_1) ,(j_1,y_1,\widetilde{f}_1)]) 
=\widetilde{\pi}_\Gc([(i_2,x_2,\widetilde{g}_2) ,(j_2,y_2,\widetilde{f}_2)])
\\
\iff 
& \widetilde{\chi}([(i_1,x_1,\widetilde{g}_1)  ,(j_1,y_1,\widetilde{f}_1)])
\sim_\Gc \widetilde{\chi}([(i_2,x_2,\widetilde{g}_2) ,(j_2,y_2,\widetilde{f}_2)]).
\end{align*}
Since $\widetilde{\chi}$ is bijective, this proves our claim that the fibres of $\pi_\Gc=\widetilde{\pi}_\Gc\circ \widetilde{\chi}^{-1}$ are the equivalence classes defined by~\eqref{simTildeMathG}. 

{\bf  Assertion 1}. 
\emph{The mapping 
$\pi_\Gc\colon \widetilde{\Gc}\to (P\times P)/G=\Gc$ is a \Q-atlas.}
    
   \begin{proof}[Proof of Assertion 1]
We check the conditions \eqref{Q_ex_a}--\eqref{Q_ex_b} in the definition of a \Q-chart given in 
Definition~\ref{Q_def}. 

\eqref{Q_ex_a}
To prove the existence of local transverse diffeomorphisms, 
assume 
$$\pi_\Gc(i_1,x_1,\widetilde{g}_{10}, j_1,y_1) =\pi_\Gc(i_2,x_2,,\widetilde{g}_{20},j_2,y_2)$$ 
that is, 
$(x_1, y_1)=(x_2,y_2)=:(x_0,y_0)\in( U_{i_1}\times U_{j_1})\cap (U_{i_2}\times U_{j_2})$ and 
\begin{equation}
	\label{GaugeQGroupoid_proof_eq1}  
	\pi_G(\widetilde{g}_{20})=\psi_{i_2i_1}(x_0)\pi_G(\widetilde{g}_{10})\psi_{j_2j_1}(y_0)^{-1}  
\end{equation}
according to the above description of the fibres of~$\pi_\Gc$.
Here $\psi_{i_2i_1}\colon U_{i_2}\cap U_{i_1}\to G$ and $\psi_{j_2 j_1}\colon U_{j_2}\cap U_{j_1}\to G$ are \Q-smooth mappings hence, as in the proof of Proposition~\ref{bundle}, 
there exist open subsets $U_{x_0}\subseteq U_{i_2}\cap U_{i_1}$ and $U_{y_0}\subseteq U_{j_2}\cap U_{j_1}$ with $x_0\in U_{x_0}$ and $y_0\in U_{y_0}$, as well as smooth mappings
 $\widetilde{\psi}_{i_2i_1}\colon U_{x_0}\to \widetilde{G}$ and $\widetilde{\psi}_{j_2 j_1}\colon U_{y_0}\to \widetilde{G}$  with 
 \begin{equation}
 	\label{GaugeQGroupoid_proof_eq2} 
 \pi_G\circ \widetilde{\psi}_{i_2i_1}=\psi_{i_2 i_1}\vert_{U_{x_0}}
\text{ and }\pi_G\circ \widetilde{\psi}_{j_2j_1}=\psi_{j_2 j_1}\vert_{U_{y_0}}.
 \end{equation}
$\pi_G\circ \widetilde{\psi}_{i_2i_1}=\psi_{i_2 i_1}\vert_{U_{x_0}}$ 
and $\pi_G\circ \widetilde{\psi}_{j_2j_1}=\psi_{j_2 j_1}\vert_{U_{y_0}}$. 
Then \eqref{GaugeQGroupoid_proof_eq1} implies 
$\pi_G(\widetilde{g}_{20})=\pi_G(\widetilde{\psi}_{i_2i_1}(x_0)\widetilde{g}_{10}\widetilde{\psi}_{j_2j_1}(y_0)^{-1}) $ hence, since $\Ker\pi_G=\Lambda$, it follows that there exists $\lambda_0\in\Lambda$ satisfying
\begin{equation}
\label{GaugeQGroupoid_proof_eq3}  
\widetilde{g}_{20}=\lambda_0\widetilde{\psi}_{i_2i_1}(x_0)\widetilde{g}_{10}\widetilde{\psi}_{j_2j_1}(y_0)^{-1}.
\end{equation}
We now define 
\begin{align*}
	& h\colon  U_{x_0}\times\{i_1\}\times\widetilde{G}\times U_{y_0}\times\{j_1\}
\to U_{x_0}\times \{i_2\}\times\widetilde{G}\times U_{y_0}\times\{j_2\}, 
\\   
& h(x,i_1,\widetilde{g}_1,y,j_1):=  h(x,i_2,\lambda_0\widetilde{\psi}_{i_2i_1}(x)\widetilde{g}_1\widetilde{\psi}_{j_2j_1}(y)^{-1},y,j_2).
\end{align*}
It follows by \eqref{GaugeQGroupoid_proof_eq3} that $h(i_1,x_1,\widetilde{g}_{10}, j_1,y_1) =(i_2,x_2,,\widetilde{g}_{20},j_2,y_2)$ 
and it is clear that $h$ is a diffeomorphism 
with its inverse 
\begin{align*}
	& h^{-1}\colon  
	U_{x_0}\times \{i_2\}\times\widetilde{G}\times U_{y_0}\times\{j_2\} 
	\to U_{x_0}\times\{i_1\}\times\widetilde{G}\times U_{y_0}\times\{j_1\}, 
	\\   
	& h^{-1}(x,i_2,\widetilde{g}_2,y,j_2):=  h^{-1}(x,i_1,\lambda_0^{-1}\widetilde{\psi}_{i_2i_1}(x)^{-1}\widetilde{g}_1\widetilde{\psi}_{j_2j_1}(y),y,j_1).
\end{align*}
(Recall that $\lambda_0\in\Lambda$ and $\Lambda$ is contained in the center of $\widetilde{G}$.)
Moreover, by \eqref{simTildeMathG}, we have $\pi_\Gc\circ h=\pi_\Gc\vert_{U_{x_0}\times\{i_1\}\times\widetilde{G}\times U_{y_0}\times\{j_1\}}$, 
hence $h$ is a transverse diffeomorphism with respect to $\pi_\Gc\colon\widetilde{\Gc}\to\Gc$. 

\eqref{Q_ex_b} 
We now prove that if $T$ is a smooth manifold and $\gamma_1,\gamma_2\colon T\to \widetilde{\Gc}$ are smooth mappings satisfying $\pi_\Gc\circ\gamma_1=\pi_\Gc\circ\gamma_2$, 
and   if $t_0\in T$ has  the property $\gamma_1(t_0)=\gamma_2(t_0)$, 
then $t_0$ has an open neighborhood $T_0\subseteq T$ with $\gamma_1(t)=\gamma_2(t)$ for every $t\in T_0$.   
To this end, we may assume without loss of generality that the smooth manifold $T_0$ is connected and then there exist $i_1,i_2,j_1,j_2\in I$ for which we may write 
$$\gamma_r(t)=(x(t),i_r,\widetilde{g}_r(t),y(t),j_r)\in U_{i_r}\times\{i_r\}\times\widetilde{G}\times U_{j_r}\times\{j_r
\}\subseteq \widetilde{\Gc}$$ 
for $t\in T$ and $r=1,2$. 
Since $\pi_\Gc(\gamma_1(t))=\pi_\Gc(\gamma_2(t))$ for every $t\in T$, it follows 
by~\eqref{simTildeMathG} that for every $t\in T$ we have 
$$
\begin{cases}
	(x_1(t), y_1(t))=(x_2(t),y_2(t))=:(x(t),y(t))\in( U_{i_1}\times U_{j_1})\cap (U_{i_2}\times U_{j_2}) & \\
	\pi_G(\widetilde{g}_2(t))
	=\psi_{i_2i_1}(x(t))\pi_G(\widetilde{g}_1(t))\psi_{j_2j_1}(y(t))^{-1}. &
\end{cases}$$
As in the proof of (a) above, we now select smooth mappings $\widetilde{\psi}_{i_2i_1}\colon U_{x_0}\to \widetilde{G}$ and $\widetilde{\psi}_{j_2 j_1}\colon U_{y_0}\to \widetilde{G}$ 
satisfying~\eqref{GaugeQGroupoid_proof_eq2} for $x_0:=x(t_0)$ and $y_0:=y(t_0)$. 
Since the mappings $\gamma_1,\gamma_2\colon T\to\widetilde{G}$ are continuous, 
their components $x\colon T\to U_{i_1}\times U_{j_1}$ and $y\colon T\to U_{i_2}\times U_{j_2}$ are also continuous, hence there exists a connected open neighborhood $T_0$ of $t_0\in T$ satisfying $x(T_0)\subseteq U_{x_0}$ and $y(T_0)\subseteq U_{y_0}$. 
For every $t\in T_0$ we then obtain 
$$\pi_G(\widetilde{g}_2(t))
=\pi_G(\widetilde{\psi}_{i_2i_1}(x(t))\widetilde{g}_1(t)\widetilde{\psi}_{j_2j_1}(y(t))^{-1})$$
hence, since $\Ker\pi_G=\Lambda$, 
it follows that the smooth mapping 
$$T_0\to \widetilde{G},\quad 
t\mapsto 
\widetilde{g}_2(t)^{-1}
\widetilde{\psi}_{i_2i_1}(x(t))\widetilde{g}_1(t)\widetilde{\psi}_{j_2j_1}(y(t))^{-1}
$$
takes values in $\Lambda$. 
Since $\Lambda$ is pseudo-discrete, it follows that the above mapping is constant, 
hence there exists $\lambda_0\in\Lambda$ with 
$$(\forall t\in T_0)\quad \widetilde{g}_2(t)
=\lambda_0\widetilde{\psi}_{i_2i_1}(x(t))\widetilde{g}_1(t)\widetilde{\psi}_{j_2j_1}(y(t))^{-1}.$$
Since $\gamma_1(t_0)=\gamma_2(t_0)$, we then obtain $\lambda_0=\1_{\widetilde{G}}$, 
and then, finally, $\gamma_1(t)=\gamma_2(t)$ for all $t\in T_0$. 
\end{proof}

{\bf Assertion 2}.  \emph{The Lie groupoid $\widetilde{ \Gc}\tto \widetilde{M}$ is locally trivial and the \Q-atlas $\pi_\Gc\colon \widetilde{\Gc}\to (P\times P)/G$ is a groupoid morphism over the \'etale smooth mapping $\pi_M\colon \widetilde{M}\to M$.
} 

\begin{proof}[Proof of Assertion 2] 
In order to see that the Lie groupoid $\widetilde{ \Gc}\tto \widetilde{M}$ is locally trivial we can directly check that the mapping 
$(\widetilde{\bt},\widetilde{\bs})\colon \widetilde{ \Gc}\to \widetilde{M}\times \widetilde{M}$ is a surjective submersion. 
Alternatively, we  can use the fact that $\widetilde{\chi}\colon(\widetilde{P}\times \widetilde{P})/\widetilde{G}\to\widetilde{\Gc}$  is an isomorphism of Lie groupoids, 
and the Lie groupoid $(\widetilde{P}\times \widetilde{P})/\widetilde{G}\to\widetilde{\Gc}$ is locally trivial since it is the gauge groupoid of the trivial principal bundle $\widetilde{\beta}\colon \widetilde{P}\to \widetilde{M}$. 

The fact that $\pi_\Gc\colon \widetilde{\Gc}\to (P\times P)/G$ is a groupoid morphism over the mapping $\pi_M\colon \widetilde{M}\to M$ follows by the commutative diagram \eqref{GaugeQGroupoid_proof_eq0}, in which $\widetilde{\chi}$ is a groupoid isomorphism over the identity mapping $\id_{\widetilde{M}}\colon \widetilde{M}\to\widetilde{M}$, 
while $\widetilde{\pi}_\Gc$ is a groupoid morphism over the mapping $\pi_M\colon \widetilde{M}\to M$. 
 \end{proof}
 
 {\bf Assertion 3}. \emph{$ \Gc\tto M$ is a \Q-groupoid and $\pi_ \Gc\colon \widetilde{\Gc}\to (P\times P)/G$ is a \Q-groupoid morphism over~$\pi_M\colon \widetilde{M}\to M$.}
 
 \begin{proof} 
 	In order to prove that  $ \Gc\tto M$ is a \Q-Lie groupoid, using Assertions 1--2,  we only need to show that the assumptions of Proposition \ref{preliminaries}  are satisfied and that the multiplication $\bm\colon\Gc^{(2)}\to\Gc$ is \Q-smooth.
 
 From the definition of $\widetilde \bs$ and $ \bs$ we have 
 $$ \bs\circ \pi_ \Gc=\pi_M\circ \widetilde \bs.$$
  Since $\widetilde \bs$ is a submersion,  $\pi_ \Gc\colon\widetilde{ \Gc}\to  \Gc$ and $\pi_M:\widetilde{M}\to M$ are \Q-atlases, it follows that $ \bs$  is a \Q-submersion.  
 In the same way, from there definitions we have 
 $$ \bi\circ \pi_ \Gc=\pi_ \Gc\circ \widetilde \bi$$
 and so $ \bi$ is smooth.
 
 Also, it directly follows from there definitions, that  $(\pi_ \Gc\times\pi_\Gc)(\widetilde{\Gc}^{(2)})= \Gc^{(2)}$ and 
 $$ \bm\circ (\pi_ \Gc, \pi_ \Gc)\vert_{\widetilde{ \Gc}^{(2)}}=\pi_ \Gc\circ \widetilde \bm$$
 which implies that $ \bm$ is smooth.
  
  All these relations implies that $ \Gc\tto M$ is a \Q-groupoid and $\pi_ \Gc$ is a \Q-groupoid morphism over $\pi_M$. 
  Finally, we have 
  $$ (\pi_M,\pi_M)\circ (\widetilde{\bt}\times\widetilde{\bs})=(\bt\times\bs)\circ(\pi_\Gc,\pi_\Gc)$$ 
  where $\widetilde{\bt}\times\widetilde{\bs}\colon \widetilde{\Gc}\to\widetilde{M}\times\widetilde{M}$ is a surjective submersion, 
  hence $\widetilde{\bt}\times\widetilde{\bs}\colon \Gc\to M\times M$ is a surjective \Q-submersion. 
  Thus, the fact that $\widetilde{ \Gc}\tto\widetilde{M}$ is locally trivial implies the same property for $ \Gc\tto M$.
   \end{proof}
The proof of Theorem~\ref{GaugeQGroupoid} is complete by the preparatory remarks along with Assertions 1--3. 
\end{proof}
 
Now  we have the following generalization of a  classical result of locally trivial Lie groupoids:

\begin{theorem}\label{isoloctricQgroupoid} 
Every locally trivial \Q-groupoid is isomorphic to the gauge groupoid of a principal  \Q-bundle.
\end{theorem}

\begin{proof} 
	Let $ \Gc\tto M$ be a locally trivial \Q-groupoid. 
	Fix some $x\in M$. 
	From Theorem~\ref{G-principalbundle} 
  the mapping $ \bt_x:  \Gc(x,-)\to M$ 
    is  a principal \Q-bundle whose structure group is $ \Gc(x)$. 
    The corresponding gauge groupoid is  $( \Gc(x,-)\times \Gc(x,-))/ \Gc(x)\tto M$.
We consider the mapping 
$$\widehat{\Phi}\colon \Gc(x,-) \times  \Gc(x,-)\to  \Gc,\quad  
\widehat{\Phi}(g,h):=gh^{-1}.$$
    We will show that $\widehat{\Phi}$ is a  surjective \Q-submersion (cf. \cite[proof of Prop. 1.3.3]{Ma05}).  
    The pull-back of $ \bt: \Gc\to M$ over $ \bt_x: \Gc(x,-)\to M$ is $ \Gc(x,-)\times_M \Gc\to  \Gc(x,-)$ and we have the following diagram 
   \begin{equation}\label{pullbackG}
    \xymatrix{
 \Gc(x,-)\times_M \Gc \ar[d] \ar[r]^{\;\;\;\;\;\;\;\;\;\theta} & \Gc\ar[d]^ \bt\\
 \Gc(x,-) \ar[r]^{ \bt_x}          &M}
\end{equation} 
  Since each vertical arrows  and $ \bt_x$ are all surjective  \Q-submersions, so is $\theta$, 
  by \cite[Ch. I, \S 2, no. 6, Cor., page 244]{Br73}.   
  We now define the maps
    \begin{align*}
    &\phi\colon \Gc(x,-)\times \Gc(x,-)\to  \Gc(x,-)\times_M \Gc, \quad (g,h)\mapsto (g, gh^{-1}),\\
    &\psi\colon \Gc(x,-)\times_M \Gc\to  \Gc(x,-)\times \Gc(x,-),  \quad (g,h)\mapsto (g, h^{-1}g).
    \end{align*}
     Then $\phi^{-1}=\psi$  and, as  $ \bm$ and $ \bi$ are \Q-smooth, so are $\psi$ and $\phi$. But $\widehat{\Phi}=\theta\circ \psi $  and, since $\theta$ is a surjective \Q-submersion, 
     so is $\widehat{\Phi}$. 
    
    On the other hand, consider the quotient map
      $$\Phi: ( \Gc(x,-)\times \Gc(x,-))/ \Gc(x))\to\Gc,\qquad [g,h]\mapsto gh^{-1}.$$
    Then we have the following commutative diagram
    \begin{equation}
	\label{hatPhi}
	\xymatrix{
		 \Gc(x,-)\times  \Gc(x,-)\ar[d] \ar[r]^{\;\;\;\;\;\;\;\;\;\;\;\;\widehat{\Phi}}&   \Gc \\
	 (\Gc(x,-)\times  \Gc(x,-))/ \Gc(x) \ar[ur]_{\;\;\;\;\;\;\;\;\;\Phi}& 
	}
	\end{equation}
	It follows that $\Phi$ is a bijection and, as the horizontal and vertical arrows are \Q-smooth, so are $\Phi$ and $\Phi^{-1}$.
	
	Finally, it is straightforward to prove that $\Phi$ is a groupoid morphism.
\end{proof}

\begin{remark}\label{isoG-Gauge} 
\normalfont 
	 In Theorem \ref{isoloctricQgroupoid}, the isomorphism $\Phi$ between $\mathcal {G}$ and the gauge groupoid of the principal bundle $ \bt_x:  \Gc(x,-)\to M$ depends on the choice of the point $x\in M$. However, from the proof of Proposition \ref{txSubmersion}, if $g$ belongs to $ \Gc(x,y)$ then $R_g$ induces a diffeomorphism from the principal bundle and $ \bt_y:  \Gc(y,-)\to M$ and so the conclusion of Theorem \ref{isoloctricQgroupoid} does not depends on the choice of such a point in $M$.
\end{remark}

\subsection*{The Ehresmann isomorphism for principal \Q-bundles}
If $\beta\colon P\to M$ is a smooth principal bundle whose structure group is a Lie group~$G$, 
and with the gauge groupoid $\Gc:=(P\times P)/G\tto M$, 
then there exists a natural isomorphism of Lie algebroids 
\begin{equation*}
\Ehr(\beta)\colon (TP)/G\mathop{\longrightarrow}\limits^\sim \mathcal{AG}
\end{equation*} 
called the \emph{Ehresmann isomorphism} in \cite[\S 4, Ex. (a), page 533]{Pr83} 
because of the first occurrence of that vector bundle isomorphism in \cite[\S4, page 40]{Eh51}. 
See also \cite[\S 15, page 95]{Lb71} and  \cite[subsect. 2.2]{Lb07}. 
We recall that, in order to construct the aforementioned isomorphism, 
one starts from the quotient map $\Psi\colon P\times P\to (P\times P)/G=\Gc$, 
whose tangent map $\partial_1\Psi$ with respect to its first variable 
fits in a commutative diagram 
\begin{equation*}
	\xymatrix{
TP \ar[d] \ar@{^{(}->}[r]^{E} & (TP)\times P \ar[r]^{\quad \partial_1\Psi} \ar[d] & T^\bs\Gc \ar[d]\\
P \ar@{^{(}->}[r]^{\1} 
& P\times P \ar[r]^{\Psi} & \Gc
}
\end{equation*}
Here $\1\colon P\to P\times P$, $p\mapsto (p,p)$, 
and $E\colon TP\to (TP)\times P$, $E(v):=(v,p)$ if $v\in T_pP$. 
Since $(\Psi\circ \1)(P)=\1_\Gc$ (the image in $\Gc$ 
of the unit space of the groupoid $\Gc$), 
we obtain $\Ran((\partial_1\Psi)\circ E)=\mathcal{AG}$ 
and the pair $((\partial_1\Psi)\circ E,\Psi\circ\1)$ 
is a fibrewise isomorphism from the tangent vector bundle $TP\to P$ onto the Lie algebroid $\mathcal{AG}\to \1_\Gc\simeq P/G$. 
By the definition of the quotient map $\Psi$, we additionally have $\Psi(pg,qg)=\Psi(p,q)$ for all $p,q\in P$ and $g\in G$, 
hence the map $(\partial_1\Psi)\circ E$ is constant on the orbits of the tangent action $TP\times G\to TP$, 
and  we then obtain the commutative diagram 
\begin{equation*}
	\xymatrix{
		TP  \ar[r]^{E\circ \partial_1\Psi} \ar[d] & \mathcal{AG} \\
		(TP)/G \ar@{.>}[ur]_{\Ehr(\beta)} &  
	}
\end{equation*}
in which the vertical arrow is the quotient map cf. \cite[Prop. 3.2.1]{Ma05}, 
while $\Ehr(\beta)$ is a fibrewise isomorphism over the diffeomorphism $P/G\simeq \1_\Gc$, hence $\Ehr(\beta)$ is an isomorphism of vector bundles. 
Moreover, one can check that $\Ehr(\beta)$ is actually an isomorphism of Lie algebroids, 
cf.  e.g., \cite[subsect. 2.2]{Lb07}. 

If we denote by $\mathbb{PB}$ the category whose objects are the smooth principal bundles, 
by $\mathbb{LA}$ the category whose objects are the Lie algebroids and the morphisms are discussed in \cite[\S 4.3]{Ma05}, 
by $\mathbb{LG}$ the category whose objects are the Lie groupoids
then we have the diagram of functors 
\begin{equation*}
\xymatrix{
\mathbb{PB} \ar[r]^{\gauge} \ar[d]_{\Atiyah} & \mathbb{LG} \ar[d]^{\LieF} \\
\mathbb{LA} 
& \mathbb{LA}
}
\end{equation*}
Here we use the following notation: 
\begin{itemize}
	\item $\gauge\colon \mathbb{PB}\to \mathbb{LG}$ is a functor, 
	whose action on morphisms is described in \cite[Ex. 1.2.9]{Ma05}; 
	\item $\LieF\colon \mathbb{LG}\to \mathbb{LA}$ is the Lie functor 
	that is takes every Lie groupoid to its corresponding Lie algebroid cf. \cite[\S 3.5, page 125]{Ma05}; 
	\item $\Atiyah\colon \mathbb{PB}\to\mathbb{LA}$ 
	is the Atiyah functor that takes a principal bundle $\beta\colon P\to M$ with its structure group~$G$ to the Lie algebroid $\Atiyah(\beta):=(TP)/G$, while its action on morphisms of principal bundles is 
	described in Definition~\ref{AP_def_morph_algbd} and Proposition~\ref{AP_morph_algbd_prop}. 
\end{itemize}
The Ehresmann isomorphism gives an isomorphism of functors (natural isomorphism)
denoted 
\begin{equation*}
	\Ehr\colon \LieF\circ\gauge\to\Atiyah. 
\end{equation*}
That is, the Lie algebroid isomorphism 
$\Ehr(\beta)\colon (\LieF\circ\gauge)(\beta)\to\Atiyah(\beta)$ 
is canonical in the sense that 
for every pair of smooth principal bundles $\beta_j\colon P_j\to M_j$ 
if $\varphi$ is a morphism from $\beta_1$ to $\beta_2$, then 
the diagram 
\begin{equation*}
	\xymatrix{
(\LieF\circ\gauge)(\beta_1) \ar[rr]^{\Ehr(\beta_1)} \ar[d]_{\Atiyah(\varphi)} & & \Atiyah(\beta_1) \ar[d]^{\Atiyah(\varphi)}\\
(\LieF\circ\gauge)(\beta_2) \ar[rr]_{\Ehr(\beta_2)}  & &  \Atiyah(\beta_2)
}
\end{equation*}
is commutative.

 We now show that the above classical facts actually carry over from smooth principal bundles to principal \Q-bundles. 
 To this end we prove the following result, which is applicable to every principal \Q-bundle, taking into account Theorem~\ref{GaugeQGroupoid}. 
 
 \begin{lemma}
 \label{Ehr_lemma}
 Assume the commutative diagram 
 \begin{equation*}
 	\xymatrix{
 		\widetilde{M} \ar[d]_{\pi_M} & \widetilde{P} \ar[l]_{\widetilde{\beta}} \ar[d]_{\pi_P} & \widetilde{P}\times \widetilde{G} \ar[l]_{\widetilde{A}}
 		 \ar[d]^{(\pi_P\times \pi_G)} \\
 		M & P \ar[l]^{\beta} & P\times G \ar[l]^{A}
 	}
 \end{equation*}
where 
\begin{itemize}
	\item $\widetilde{\beta}\colon\widetilde{P}\to\widetilde{M}$ is a smooth principal bundle with its structure group $\widetilde{G}$ acting via $\widetilde{A}$; 
	\item $\beta\colon P\to M$ is a principal \Q-bundle whose structure group is the \Q-group $G$, acting via $A$; 
	\item the maps $\pi_M$, $\pi_P$, and $\pi_G$ are \Q-atlases, and $\pi_G$ is additionally a group homomorphism. 
\end{itemize}
Also consider the diagonal actions of $\widetilde{G}$  on $\widetilde{P}\times\widetilde{P}$ with the corresponding quotient map $\widetilde{\Phi}\colon \widetilde{P}\times\widetilde{P}\to(\widetilde{P}\times\widetilde{P})/\widetilde{G}=:\widetilde{\Gc}$ and, 
similarly, the diagonal action of $G$ on $P\times P$ with the quotient map 
$\Phi\colon P\times P\to(P\times P)/G=\Gc$, and assume that the mapping 
\begin{equation*}
\pi_\Gc\colon \widetilde{\Gc}\to\Gc,\quad 
(\widetilde{p},\widetilde{q})\widetilde{G}\mapsto (\pi_P(\widetilde{p}),\pi_P(\widetilde{q}))G
\end{equation*}
is a \Q-atlas. 

Then the morphisms of Lie algebroids 
$\Atiyah(\pi_P)$ and $\Ac(\pi_\Gc)$   
over the \'etale map~$\pi_M$  are fibrewise isomorphisms of vector bundles, 
and there exists a unique isomorphism of Lie algebroids $\Ehr(\beta)\colon (TP)/G\to\mathcal{AG}$ for which the diagram 
\begin{equation*}
\xymatrix{T\widetilde{P} \ar[d]_{T(\pi_P)} \ar[rr]^{
	q_{T\widetilde{P}}} & & (T\widetilde{P})/\widetilde{G} \ar[d]_{
\Atiyah(\pi_P)} \ar[rr]^{\Ehr(\widetilde{\beta})} & & \mathcal{A}\widetilde{\Gc} \ar[d]^{\mathcal{A}(\pi_\Gc)} \\
TP \ar[rr]_{
q_{TP}} & &	(TP)/G \ar@{.>}[rr]_{\Ehr(\beta)}& & \mathcal{AG}
}
\end{equation*}
is commutative.
 \end{lemma}
 
\begin{proof}
The mapping	$\Atiyah(\pi_P)$ is a morphism of Lie algebroids   
	over the \'etale map $\pi_M$ by Proposition~\ref{AP_morph_algbd_prop}. 
	  In order to prove that $\Atiyah(\pi_P)$ is a fibrewise isomorphism of vector bundles, 
	we use 
the localization property of the Atiyah functor (Remark~\ref{AP_loc}) and the localization property of the Lie algebroid morphisms (Remark~\ref{trivLiealg_morph_loc}) 
to see that we may assume without loss of generality that both the smooth principal bundle $\widetilde{P}$ and the principal \Q-bundle $P$ are trivial. 
On the other hand, since $\pi_G\colon\widetilde{G}\to G$ is a \Q-atlas, 
it follows that its tangent map $T_\1(\pi_G)\colon T_\1\widetilde{G}\to T_\1G$ 
is an isomorphism of Lie algebras. 
Since moreover $\pi_M\colon\widetilde{M}\to M$ is an \'etale map, it follows by the explicit description of $\Atiyah(\pi_P)$ for trivial principal \Q-bundles in Eq.~\eqref{triv_eq3} in Example~\ref{triv} that $\Atiyah(\pi_P)$ is a fibrewise isomorphism of vector bundles. 

On the other hand, both \Q-groupoids $\widetilde{\Gc}$ and $\Gc$ are locally trivial by Theorem~\ref{GaugeQGroupoid}, 
hence the mapping	$\A(\pi_\Gc)$ is a morphism of Lie algebroids   
over the \'etale map~$\pi_M$ by Proposition~\ref{Liefunctor_prop}. 
As above, for proving that $\A(\pi_\Gc)$ is a fibrewise isomorphism of vector bundles, 
we use 
 the localization property of the Lie algebroid morphisms (Remark~\ref{trivLiealg_morph_loc}) and, moreover, the local triviality of the \Q-groupoids $\widetilde{\Gc}$ and $\Gc$ as given by Remark~\ref{decomposing} 
to see that we may assume without loss of generality that both the smooth principal bundle $\widetilde{P}$ and the principal bundle $P$ are trivial. 
We have noted above that $T_\1(\pi_G)\colon T_\1\widetilde{G}\to T_\1G$ 
is an isomorphism of Lie algebras. 
Since moreover $\pi_M\colon\widetilde{M}\to M$ is an \'etale map, it follows by the explicit description of $\A(\pi_\Gc)$ for trivial principal \Q-bundles in Eq.~\eqref{Liefunctor_prop_proof_eq1} in the proof of Proposition~\ref{Liefunctor_prop} that $\A(\pi_\Gc)$ is a fibrewise isomorphism of vector bundles. 

Now, using the above fibrewise isomorphisms $\Atiyah(\pi_P)$ and $\A(\pi_\Gc)$ along with the classical Ehresmann isomorphism $\Ehr(\widetilde{\beta})\colon(T\widetilde{P})/\widetilde{G}\to\A\widetilde{\Gc}$ 
for the smooth principal bundle $\widetilde{\beta}\colon\widetilde{P}\to\widetilde{M}$, 
it easily follows that there exists a vector bundle isomorphism $\Ehr(\beta)\colon(TP)/G\to\A\Gc$ for which the diagram in the statement is commutative. 
It remains to check that $\Ehr(\beta)$ is a morphism of Lie algebroids. 
To this end we use once again the localization property of the Atiyah functor (Remark~\ref{AP_loc}) and the localization property of the Lie algebroid morphisms (Remark~\ref{trivLiealg_morph_loc}) 
to see that we may assume without loss of generality that both the smooth principal bundle $\widetilde{P}$ and the principal \Q-bundle $P$ are trivial. 
In this special case it is straightforward to check that $\Ehr(\beta)$ is a morphism of Lie algebroids, and this completes the proof. 
\end{proof}

\begin{theorem}
\label{Ehr_th}
	The Ehresmann isomorphism $	\Ehr\colon \LieF\circ\gauge\to\Atiyah$ 
	is an isomorphism of functors on the category of principal \Q-bundles. 
\end{theorem}

\begin{proof}
We must prove that if $G$ is a \Q-group and $\beta\colon P\to M$ is a principal \Q-bundle with its structure group~$G$ and its gauge groupoid $\Gc=(P\times P)/G\tto M$, then one has the Ehresmann isomorphism of Lie algebroids $\Ehr(\beta)\colon (TP)/G\to\A\Gc$. 
To this end we apply Proposition~\ref{Ehr_lemma} for the smooth principal groupoid $\widetilde{\Gc}$ given by Theorem~\ref{GaugeQGroupoid}. 
\end{proof}

\begin{corollary}
\label{integration}
Every transitive Lie algebroid over a second countable, smooth manifold is isomorphic to the Lie algebroid of a locally trivial \Q-groupoid. 
\end{corollary}

\begin{proof}
First apply Theorem~\ref{AP_th} to integrate the Atiyah sequence of an arbitrary transitive Lie algebroid $q\colon A\to M$ to a principal \Q-bundle $\beta\colon P\to M$, and then use the Ehresmann isomorphism given by Theorem~\ref{Ehr_th} to prove that the Lie algebroid $q\colon A\to M$ is isomorphic to the Lie algebroid of the gauge groupoid of $\beta\colon P\to M$. 
\end{proof}

\appendix

\section{On embedded \Q-submanifolds}
\label{AppA}

In this appendix we discuss a phenomenon that is specific to the \Q-manifolds as compared with the smooth manifolds, 
namely the fact that the singleton subsets may not be embedded submanifolds 
(Proposition~\ref{submanif_prop}). 
The significance of this fact for \Q-groupoids is that, even in the special case of \Q-groups, the space of units may not be an embedded submanifold. 
(See Corollary~\ref{submanif_cor}.) 

We recall from \cite[I.1.10]{KlMiSl93} that if $X$ is a smooth manifold, then a subset $X_0\subseteq X$ is a \emph{submanifold} if for every $x_0\in X$ there exist a local chart $u\colon U\to \Vc$ and a linear subspace $\Vc_0\subseteq\Vc$ with $u(U\cap X_0)=u(U)\cap\Vc_0$, where $U\subseteq X$ is an open subset with $x_0\in U$, $\Vc$ is a finite-dimensional real vector space, and $u(U)\subseteq \Vc$ is an open subset. 

\begin{remark}
	\label{shrink}
	\normalfont 
	In the above notation, for every subset $U'\subseteq U$ we have $u(U'\cap X_0)=u(U')\cap\Vc_0$. 
	
	In fact, $u(U'\cap X_0)\subseteq u(U\cap X_0)=u(U)\cap\Vc_0$ and 
	$u(U'\cap X_0)\subseteq\subseteq u(U')$, hence $u(U'\cap X_0)\subseteq u(U')\cap u(U)\cap\Vc_0=u(U')\cap\Vc_0$. 
	For the converse inclusion, we have $u(U')\cap\Vc_0\subseteq u(U)\cap \Vc_0=u(U\cap X_0)$ 
	and also $u(U')\cap\Vc_0\subseteq u(U')$, hence $u(U')\cap\Vc_0\subseteq u(U')\cap u(U\cap X_0)=u(U'\cap U\cap X_0)=u(U'\cap X_0)$. 
\end{remark}

\begin{lemma}
	\label{submanif_lemma}
	If $\pi_j\colon X_j\to S$ are compatible \Q-atlases for $j=1,2$, and a subset $S_0\subseteq S$ has the property that $\pi_1^{-1}(S_0)\subseteq X_1$ is an embedded submanifold, 
	then $\pi_2^{-1}(S_0)\subseteq X_2$ is an embedded submanifold as well. 
\end{lemma}

\begin{proof}
	Let $x_2\in \pi_2^{-1}(S_0)$ arbitrary and denote $s_0:=\pi_2(x_2)\in S_0$. 
	Since the \Q-atlas $\pi_1\colon X_1\to S$ is surjective, we may select $x_1\in \pi_1^{-1}(S_0)$ with $\pi_1(x_1)=s_0$. 
	The subset $X_{01}:=\pi_1^{-1}(S_0)\subseteq X_1$ is a submanifold by hypothesis, hence 
	there exist a local chart $u_1\colon U_1\to \Vc$ and a linear subspace $\Vc_0\subseteq\Vc$ with $u_1(U_1\cap X_{01})=u_1(U_1)\cap\Vc_0$, where $U_1\subseteq X_1$ is an open subset with $x_1\in U_1$, $\Vc$ is a finite-dimensional real vector space, and $u_1(U_1)\subseteq \Vc$ is an open subset. 
	
	Using the hypothesis that the \Q-atlases $\pi_1$ and $\pi_2$ are compatible, 
	we can use \cite[Prop. 2.2]{BPZ19} to find open subsets $U_j\subseteq X_j$ for $j=1,2$ with $x_1\in U_1'\subseteq U_1$, 
	$x_2\in U_2'$, and a diffeomorphism $h\colon U_1'\to U_2'$ with $h(x_1)=x_2$ and $\pi_2\circ h=\pi_1\vert_{U_1'}$. 
	If we denote  $X_{02}:=\pi_2^{-1}(S_0)\subseteq X_2$, 
	it then follows that 
	\begin{equation}
		\label{submanif_lemma_proof_eq1}
		h(U_1'\cap X_{01})=U_2'\cap X_{02}. 
	\end{equation}
	Since $u_1(U_1\cap X_{01})=u_1(U_1)\cap\Vc_0$, it follows by Remark~\ref{shrink} 
	that 
	\begin{equation}
		\label{submanif_lemma_proof_eq2}
		u_1(U_1'\cap X_{01})=u_1(U_1')\cap\Vc_0
	\end{equation}
	as well. 
	Then, denoting $u_2:=u_1\circ h^{-1}\colon U_2'\to\Vc$, it follows that $u_2$ is a local chart of $X_2$ with $x_2\in U_2'$ and 
	\begin{align*}
		u_2(U_2'\cap X_{02})
		&=(u_1\circ h^{-1})(U_2'\cap X_{02})
		\mathop{=}\limits^{\eqref{submanif_lemma_proof_eq1}}
		u_1(U_1'\cap X_{01})
		\mathop{=}\limits^{\eqref{submanif_lemma_proof_eq2}}u_1(U_1')\cap\Vc_0 \\
		&=u_2(U_2')\cap\Vc_0.
	\end{align*}
	This completes the proof of the fact that the subset $X_{02}=\pi_2^{-1}(S_0)\subseteq X_2$ is an embedded submanifold. 
\end{proof}

\begin{definition}
	\label{submanif_def}
	\normalfont 
	If $S$ is a \Q-manifold, then a subset $S_0\subseteq S$ is called an \emph{embedded \Q-submanifold} if there exists a \Q-atlas $\pi\colon X\to S$ with the property that the subset $\pi^{-1}(S_0)\subseteq X$ is an embedded submanifold of the smooth manifold~$X$. 
\end{definition}

\begin{remark}
	\label{submanif_rem} 
	\normalfont 
	We collect a few elementary facts for later use. 
	\begin{enumerate}[{\rm(i)}]
		\item\label{submanif_rem_item1} 
		It follows by Lemma~\ref{submanif_lemma} that the notion of embedded \Q-submanifold is correctly defined, in the sense that it does not depend on the choice of the \Q-atlas in Definition~\ref{submanif_def}. 
		\item\label{submanif_rem_item2} 
		It is well known that every embedded submanifold of a smooth manifold has the property of initial submanifold as in \cite[I.2.9]{KlMiSl93}. 
		Therefore every embedded \Q-submanifold $S_0\subseteq S$ is an immersed \Q-submanifold 
		and moreover, for every \Q-atlas $\pi\colon X\to S$, the mapping $\pi\vert_{\pi^{-1}(S_0)}\colon \pi^{-1}(S_0)\to S_0$ is in turn a \Q-atlas 
		by \cite[Ch. 1, \S 2, no. 4, Prop., page 242]{Br73}. 
		\item\label{submanif_rem_item3}  If $S,S'$ are \Q-manifolds and  $\psi\colon S\to S'$ is a \Q-diffeomorphism (i.e., bijective and \Q-\'etale) and $S_0\subseteq S$ is an embedded \Q-submanifold, then $\psi(S_0)\subseteq S'$ is again an embedded \Q-submanifold. 
		
		In fact, let $\pi\colon X\to S$ be any \Q-atlas. 
		Since $\psi\colon S\to S'$ is \Q-\'etale, it follows that $\psi\circ \pi\colon X\to S'$ is a \Q-atlas of $S'$, by \cite[Ch. 1, \S 2, no. 2, Lemme, page 241]{Br73}. 
		Since $\psi$ is bijective, we have $(\psi\circ \pi)^{-1}(\psi(S_0))=\pi^{-1}(\psi^{-1}(\psi(S_0)))=\pi^{-1}(S_0)$, which is an embedded submanifold of $X$ by the hypothesis on $S_0$. 
		Since $\psi\circ \pi\colon X\to S'$ is a \Q-atlas of $S'$, it thus follows that $\psi(S_0)\subseteq S'$ is an embedded \Q-submanifold, 
		by \eqref{submanif_rem_item1}  above. 
	\end{enumerate}
\end{remark}

\begin{remark}
	\label{pseudodiscr_gen} 
	\normalfont 
	\emph{Every singleton subset of a \Q-manifold is an immersed \Q-sub\-manifold of dimension zero.} 
	In order to prove that fact it suffices to prove the following assertion: 
	If $\pi\colon X\to S$ is a \Q-chart and $s_0\in \pi(S)$, then for every connected smooth manifold $T$ and smooth mapping $f\colon T\to X$ with $f(T)\subseteq \pi^{-1}(s_0)$ 
	the mapping $f$ is constant. 
	This would show that $\pi^{-1}(s_0)$ is an immersed submanifold of $X$ of dimension zero, and then \cite[Ch. 1, \S 2, no. 4, Prop., page 242]{Br73} is applicable. 
	
	In fact, let $t_0\in T$ arbitrary and denote $x_0:=f(t_0)$. 
	If we define the constant mapping $f_0\colon T\to X$, $f(t):=x_0$, then we have 
	$\pi\circ f=\pi\circ f_0$ by the hypothesis $f(T)\subseteq \pi^{-1}(s_0)$. 
	Then, since $\pi\colon X\to S$ is a \Q-chart, it follows that the set $T_0:=\{t\in T\mid f(t)=f_0(t)\}$ is open in $T$. 
	We also have $t_0\in T_0$, hence $T_0\ne\emptyset$. 
	Moreover, since both $f$ and $f_0$ are continuous and $X$ is Hausdorff, 
	it follows that the subset $T_0\subseteq T$ is also closed. 
	Since the smooth manifold $T$ is connected, we then obtain $T_0=T$, hence $f=f_0$. 
	That is, $f$ is a constant mapping. 
\end{remark}

\begin{proposition}
	\label{submanif_prop}
	If $S$ is a \Q-manifold, then every singleton subset of $S$ is an embedded \Q-submanifold if and only if $S$ is a smooth  manifold. 
\end{proposition}

\begin{proof}
	If $S$ is a smooth manifold, then it is clear that for every $s\in S$, using any local chart around $s$, one can easily see that the singleton subset $\{s\}\subseteq S$ is an embedded submanifold. 
	
	Conversely, let us assume that every singleton subset of $S$ is an embedded \Q-submanifold. 
	In order to prove that $S$ is a smooth manifold it suffices to prove that an arbitrary \Q-atlas $\pi\colon X\to S$ is locally injective. 
	(See \cite[Prop. 2.6]{BPZ19}.) 
	To this end let $x_0\in X$ arbitrary and denote $s_0:=\pi(x_0)\in S$. 
	The singleton subset $\{s_0\}\subseteq S$ is an embedded \Q-submanifold by hypothesis, 
	hence, by Remark~\ref{submanif_rem}, the subset $X_0:=\pi^{-1}(s_0)\subseteq X$ is an embedded submanifold of the smooth manifold~$X$. 
	Then there exist a local chart $u\colon U\to \Vc$ and a linear subspace $\Vc_0\subseteq\Vc$ with $u(U\cap X_0)=u(U)\cap\Vc_0$, where $U\subseteq X$ is an open subset with $x_0\in U$, $\Vc$ is a finite-dimensional real vector space, and $u(U)\subseteq \Vc$ is an open subset. 
	
	We claim that $\dim\Vc_0=0$.  
	In fact, $u(U)\cap\Vc_0$ is an open nonempty subset of the vector space $\Vc_0$ with $u(x_0)\in u(U\cap X_0)=u(U)\cap\Vc_0$. 
	If $W_0$ denotes the connected component of $u(U)\cap\Vc_0$ with $u(x_0)\in W_0$, 
	then $W_0$ is again an open subset of $\Vc_0$ with $u(x_0)\in W_0$. 
	Since $\Vc_0\subseteq \Vc$ is a linear subspace of $\Vc$, it then directly follows that $W_0$ is an embedded submanifold of the open subset $u(U)$ of the linear space~$\Vc$.
	If we define $f:=u^{-1}\vert_{W_0}\colon W_0\to X$, then $f$ is smooth (since $u$ is a local chart, hence its inverse  $u^{-1}\colon u(U)\to X$ is smooth). 
	Moreover, 
	$$f(W_0)\subseteq u^{-1}(W_0)\subseteq u^{-1}(u(U)\cap\Vc_0)\subseteq U\cap X_0\subseteq X_0=\pi^{-1}(s_0)$$
	hence Remark~\ref{pseudodiscr_gen} implies that $f$ is constant. 
	Then, since $f=u^{-1}\vert_{W_0}$ is injective, it follows that the set $W_0$ is a singleton. 
	Recalling that $W_0$ is an open subset of the linear space $\Vc_0$, we obtain $\dim\Vc_0=0$, as claimed. 
	
	Now the equality  $u(U\cap X_0)=u(U)\cap\Vc_0$ implies $U\cap X_0=\{x_0\}$. 
	Since   $X_0=\pi^{-1}(s_0)$, it follows that the mapping $\pi\vert_U\colon U\to S$ 
	is injective. 
	That is, the \Q-atlas $\pi\colon X\to S$ is injective on the neighborhood $U$ of $x_0$. 
	Thus $\pi\colon X\to S$ is locally injective, and  the assertion follows by \cite[Prop. 2.6]{BPZ19}, as already mentioned above.  
\end{proof}

\begin{corollary}
	\label{submanif_cor}
	If $G$ is a \Q-group, then the singleton subset $\{\1\}\subseteq G$ is an embedded \Q-submanifold if and only if $G$ is a Lie group. 
\end{corollary}

\begin{proof}
	If $\{\1\}\subseteq G$ is an embedded \Q-submanifold of $G$, then for every $g\in G$ the singleton subset $L_g(\{\1\})=\{g\}\subseteq G$ is again an embedded \Q-submanifold by Remark~\ref{submanif_rem}\eqref{submanif_rem_item2} since the translation $L_g\colon G\to G$, $k\mapsto gk$, is a \Q-diffeomorphism. 
	Thus Proposition~\ref{submanif_prop} is applicable and implies that $G$ is a smooth manifold, and then $G$ is a Lie group. 
	
	The converse assertion directly follows by Proposition~\ref{submanif_prop}, too. 
\end{proof}

\section{On vector distributions on \Q-manifolds} 
\label{AppB}

In this section we establish some properties of vector distributions on \Q-manifolds 
that are needed in Section~\ref{Sect6} in order to establish the smoothness property of the Lie algebroid morphism that corresponds to a \Q-groupoid morphism via the Lie functor. 

Let $S$ be a \Q-manifold with a \Q-atlas $\pi\colon M\to S$. 

\begin{definition}
	\label{distrib1}
	\normalfont
	A subset $\Vc\subseteq TS$ is called a \emph{distribution} on the \Q-manifold~$S$ 
	if $(T\pi)^{-1}(\Vc)\subseteq TM$ is a distribution on the smooth manifold $M$, 
	that is, a subbundle of the tangent vector bundle $TM$. 
	For every $s\in S$ we denote $\Vc_s:=\Vc\cap T_sS$. 
	
	If this is the case and $N$ is a smooth manifold, then for every \Q-smooth function $f\colon N\to S$ we define the \emph{pullback} 
	\begin{equation}
		\label{distrib1_eq1}
		f^{!}(\Vc):=N\times_S\Vc
		=\{(n,v)\in N\times\Vc\mid v\in \Vc_{f(n)}\}.
	\end{equation}
\end{definition}

\begin{example}
	\label{distrib2}
	\normalfont 
	Let $\tau_M\colon TM\to M$ be the tangent bundle of the smooth manifold~$M$. 
	We describe the pullback of the distribution $\Vc\subseteq TS$ through the \Q-atlas $\pi\colon M\to S$, which is in particular a the \Q-smooth mapping. 
	Namely, it is easily seen that the mapping 
	$$\theta\colon (T\pi)^{-1}(\Vc)\to \pi^{!}(\Vc),\quad \theta(w):=(\tau_M(w),(T\pi)(w))$$ 
	is a fibrewise linear bijection, 
	with its inverse mapping 
	$$\theta^{-1}\colon \pi^{!}(\Vc)\to (T\pi)^{-1}(\Vc), \quad 
	\theta^{-1}(m,v):=(T_m\pi)^{-1}(v).$$
	Therefore $\pi^{!}(\Vc)$ has the unique structure of a locally trivial, smooth, vector bundle for which $\theta$ is a vector bundle isomorphism over the base map $\id_M$. 
\end{example}

\begin{remark}
	\label{distrib2.5}
	\normalfont
	In Example~\ref{distrib2}, 
	let us denote
	\begin{equation}
		\label{distrib3_proof_eq1}
		\Wc:=(T\pi)^{-1}(\Vc)\subseteq TM, 
	\end{equation}
	which is a locally trivial, smooth, vector bundle over $M$. 
	We note the following facts, for later use: 
	\begin{enumerate}[{\rm(i)}]
		\item The map $T\pi\colon TM\to TS$ is a \Q-atlas by \cite[Ch. 1, \S 1, no. 4, Prop.]{Br73}
		while the subbundle $\Wc\subseteq TM$ is in particular an embedded submanifold, hence 
		$(T\pi)\vert_\Wc\colon\Wc\to\Vc$ is a \Q-atlas for $\Vc$ 
		by Remark~\ref{submanif_rem}\eqref{submanif_rem_item2}, 
		and then $\Vc$ is a \Q-manifold. 
		\item Since the subbundle $\Wc\subseteq TM$ is an embedded submanifold of the smooth vector bundle $TM$, it follows that the subset $\Vc\subseteq TS$ is an embedded \Q-submanifold, 
		hence an immersed \Q-submanifold, of $TS$, 
		cf. Remark~\ref{submanif_rem}\eqref{submanif_rem_item2}.  
		\item If $\tau_S\colon TS\to S$ is the canonical projection, 
		then $\tau_S\circ (T\pi)\vert_\Wc\colon \Wc\to S$ is \Q-smooth, 
		which shows that $\tau_S\vert_{\Vc}\colon\Vc\to S$ is \Q-smooth.
	\end{enumerate} 
\end{remark}

\begin{proposition}
	\label{distrib3}
	If $\Vc\subseteq TS$ is a distribution and $N$ is a smooth manifold, then for every \Q-smooth function $f\colon N\to S$ the mapping $p_f\colon f^{!}(\Vc)\to N$, $p(n,v):=n$, is a locally trivial, smooth, vector bundle. 
\end{proposition}

\begin{proof}
	The assertion has a local character on $N$ hence, since  $f\colon N\to S$ is \Q-smooth, 
	we may assume without loss of generality that there exists a smooth lift $\widehat{f}\colon N\to M$ with $f=\pi\circ \widehat{f}$. 
	Since $\Wc$ in \eqref{distrib3_proof_eq1} is a locally trivial, smooth, vector bundle over~$M$, 
	by Example~\ref{distrib2}, 
	its pullback 
	\begin{align*}
		\widehat{f}^{!}(\Wc)
		&=\{(n,w)\in N\times TM\mid 
		w\in \Wc_{\widehat{f}(n)}\} 
	\end{align*}
	is a locally trivial, smooth, vector bundle over~$N$. 
	On the other hand 
	\begin{align*}
		f^{!}(\Vc)
		& =\{(n,v)\in N\times\Vc\mid v\in \Vc_{\pi(\widehat{f}(n))}\} \\
		& =\{(n,v)\in N\times\Vc\mid v\in (T_{\widehat{f}(n)}\pi)(\Wc_{\widehat{f}(n)})\} 
	\end{align*}
	and the mapping $T\pi\colon TM\to TS$ is a fibrewise linear isomorphism, 
	hence the mapping 
	\begin{equation}
		\label{distrib3_proof_eq2}
		\theta_f\colon \widehat{f}^{!}(\Wc)\to f^{!}(\Vc),\quad (n,w)\mapsto (n, (T_{\widehat{f}(n)}\pi)(w))
	\end{equation}
	is bijective and fibrewise isomorphism. 
	This implies that $f^{!}(\Vc)$ has the unique structure of a locally trivial, smooth, vector bundle for which $\theta_f$ is a vector bundle isomorphism over the base map $\id_M$. 
\end{proof}

\begin{proposition}
	\label{distrib4}
	In Proposition~\ref{distrib3}, if $f\colon N\to S$ is an immersed submanifold of the \Q-manifold~$S$, 
	then the mapping $\widetilde{f}\colon f^{!}(\Vc)\to\Vc$, $\widetilde{f}(n,v):=v$, 
	an immersed submanifold of the \Q-manifold~$\Vc$. 
\end{proposition}

\begin{proof}
	Using the fact that $f$ is injective, it is easily checked that $\widetilde{f}$ is injective, too. 
	
	In order to check that  $\widetilde{f}$ is \Q-smooth and its tangent map  at any point $(n_0,v_0)\in f^{!}(\Vc)$ is injective, we may assume, restricting $f$ to a suitable neighborhood of $n_0\in N$, that there exists a smooth lift 
	$\widehat{f}\colon N\to M$ with $f=\pi\circ \widehat{f}$. 
	Then, using the isomorphism of smooth vector bundles $\theta_f$ in~\eqref{distrib3_proof_eq2}, it is necessary and sufficient to show that 
	the mapping 
	\begin{equation}
		\label{distrib4_proof_eq1}
		\widetilde{f}\circ\theta_f \colon \widehat{f}^{!}(\Wc)\to \Vc,\quad 
		(n,w)\mapsto (T_{\widehat{f}(n)}\pi)(w)
	\end{equation}
	is a \Q-immersion. 
	Recalling the definition of $\Wc$ in \eqref{distrib3_proof_eq1}, 
	we see that the mapping $(T\pi)\vert_\Wc\colon\Wc\to\Vc$ is a \Q-atlas for $\Vc$ 
	and the smooth mapping 
	\begin{equation}
		\label{distrib4_proof_eq2}
		\widehat{\widetilde{f}\circ\theta_f} \colon \widehat{f}^{!}(\Wc)\to \Wc,\quad 
		(n,w)\mapsto w
	\end{equation}
	is a lift of $\widetilde{f}\circ\theta_f$ with respect to that atlas. 
	Thus $\widetilde{f}\circ\theta_f \colon \widehat{f}^{!}(\Wc)\to \Vc$ is \Q-smooth. 
	Moreover, since $f$ is an immersion, its lift $\widehat{f}$ is an immersion, 
	and then it easily follows that the lift $\widehat{\widetilde{f}\circ\theta_f}$ is an immersion. 
	Hence the mapping $\widetilde{f}\circ\theta_f$ in \eqref{distrib4_proof_eq1} is a \Q-immersion.
	
	It remains to check that the mapping $\widetilde{f}$ in the statement has the initial manifold property:  
	If $N_0$ is a smooth manifold and $\varphi\colon N_0\to f^{!}(\Vc)$ is a mapping such that $\widetilde{f}\circ\varphi\colon N_0\to \Vc$ is \Q-smooth, then $\varphi$ is smooth. 
	To this end, let $\tau_S\colon TS\to S$ be again the canonical projection. 
	Then, by Remark~\ref{distrib2.5},  the mapping $\tau_S\vert_{\Vc}\colon\Vc\to S$ is \Q-smooth, hence $\tau_S\circ \widetilde{f}\circ\varphi\colon N_0\to S$ is \Q-smooth. 
	By the definition of $f^{!}(\Vc)$ in~\eqref{distrib1_eq1} we have $\tau_S\circ \widetilde{f}=f$, hence we obtain that $f\circ\varphi\colon N_0\to S$ is \Q-smooth. 
	Since $f\colon N\to S$ is an immersed submanifold, hence has the initial submanifold property, by definition, it then follows that $\varphi$ is smooth, and this completes the proof. 
\end{proof}

\subsection*{Acknowledgment} 
The numerous constructive remarks made by the Referee are gratefully acknowledged. 
We wish to thank Professor Peter Michor for pointing out to us the occurrence of quotients of by non-closed subgroups in his work on integrability of Lie algebra actions on manifolds, 
Professor Raymond Barre for pointing out several important references 
and Professor Jean-Pierre Magnot for useful comments on differential spaces. 
The work of the first-named author was supported by a grant of the Ministry of Research, Innovation and Digitization, CNCS--UEFISCDI, project number PN-IV-P1-PCE-2023-0264, within PNCDI IV. 
The second-named author acknowledges financial support from 
the Centre Francophone en Math\'ematiques de Bucarest and the GDRI ECO-Math.

\end{document}